\documentclass[5p,times]{elsarticle}

\usepackage[hidelinks]{hyperref}
\usepackage{enumitem}
\usepackage{amssymb,amsmath,amsthm,mathrsfs}
\usepackage{xcolor}
\usepackage{amsfonts}
\usepackage{soul,color}
\usepackage{subfig}
\usepackage{microtype}
\usepackage{graphicx}

\usepackage{csquotes}

\usepackage{algorithm}
\usepackage{algorithmicx}  
\usepackage{algpseudocode}
\usepackage{booktabs}
\usepackage{dsfont}
\renewcommand{\mathbb}{\mathds}

\DeclareMathOperator*{\closure}{cl}

\newcommand{\mbb}[1]{\mathbb{#1 }}

\newcommand{\pce}[1]{\mathsf{#1}}

\newcommand{\inst}[1]{_{#1}}

\newcommand{\spl}
{{L}^2((\Omega, \mathcal{F}, \mu), \mathbb{R})}

\newcommand{\I}{\mathbb{I}}
\newcommand{\N}{\mathbb{N}_0}
\newcommand{\R}{\mathbb{R}}

\newtheorem{theorem}{Theorem}
\newtheorem{lemma}[theorem]{Lemma}
\newtheorem{corollary}[theorem]{Corollary}
\newtheorem{remark}[theorem]{Remark}
\newtheorem{definition}[theorem]{Definition}

\newtheorem{example}[theorem]{Example}
\newtheorem{proposition}[theorem]{Proposition}

\DeclareMathOperator{\mean}{\mathbb{E}}
\DeclareMathOperator{\var}{\mathrm{Var}}

\DeclareMathOperator{\covar}{\mathrm{Cov}}
\DeclareMathOperator{\prob}{\mathbb P}

\DeclareFontFamily{U}{mathx}{\hyphenchar\font45}
\DeclareFontShape{U}{mathx}{m}{n}{
      <5> <6> <7> <8> <9> <10>
      <10.95> <12> <14.4> <17.28> <20.74> <24.88>
      mathx10
      }{}
\DeclareSymbolFont{mathx}{U}{mathx}{m}{n}
\DeclareMathSymbol{\bigtimes}{1}{mathx}{"91}

\usepackage{tikz}
\usetikzlibrary{positioning,shapes,arrows.meta}
\usetikzlibrary{backgrounds}

\definecolor{gray1}{rgb}{0.9,	0.91,	0.93}
\definecolor{lightblue}{rgb}{0.6, 0.76, 0.92}
\definecolor{gray2}{rgb}{0.95, 0.95,0.95	}
\definecolor{gray3}{rgb}{0.85, 0.85, 0.85	}
\definecolor{lightblue1}{rgb}{0.87, 0.92, 0.97	}
\definecolor{lightblue2}{rgb}{0.62, 0.76, 0.90	}
\definecolor{lightblue3}{rgb}{0.36	0.61	0.84}	
\setlist[enumerate,1]{label={(\roman*)}}

\journal{Journal of \LaTeX\ Templates}

\biboptions{authoryear}

\usepackage{color}

\begin{document}

\begin{frontmatter}

\title{Behavioral Theory for Stochastic  Systems? \\A Data-driven Journey from Willems to Wiener and Back Again\tnoteref{mytitlenote}}

\author[mymainaddress]{Timm Faulwasser
\corref{mycorrespondingauthor}}
\cortext[mycorrespondingauthor]{Corresponding author}
\ead{timm.faulwasser@ieee.org}
\author[mymainaddress]{Ruchuan Ou}
\ead{ruchuan.ou@tu-dortmund.de}
\author[mymainaddress]{Guanru Pan}
\ead{guanru.pan@tu-dortmund.de}
\author[mysecondaryaddress]{Philipp Schmitz}
\ead{philipp.schmitz@tu-ilmenau.de}
\author[mysecondaryaddress]{Karl Worthmann}
\ead{karl.worthmann@tu-ilmenau.de}

\address[mymainaddress]{Institute of Energy Systems, Energy Efficiency and Energy Economics, TU Dortmund University, 44227 Dortmund, Germany}
\address[mysecondaryaddress]{Optimization-based Control Group, Institute of Mathematics, Technische Universität Ilmenau, 98693 Ilmenau, Germany}

\begin{abstract}
The fundamental lemma by Jan C.\ Willems and co-workers is deeply rooted in behavioral systems theory and it has become one of the supporting pillars of the recent progress on data-driven control and system analysis. 
This tutorial-style paper combines recent insights into stochastic and descriptor-system formulations of the lemma to further extend and broaden the formal basis for behavioral theory of stochastic linear systems. We show that series expansions---in particular  Polynomial Chaos Expansions (PCE) of $L^2$-random variables, which date back to Norbert Wiener's seminal 
work---enable equivalent behavioral characterizations of linear stochastic systems. Specifically, we prove that  under mild assumptions the behavior of the dynamics of the $L^2$-random variables is equivalent to the behavior of the dynamics of the series expansion coefficients and that it entails the behavior composed of sampled realization trajectories. We also illustrate the short-comings of the behavior associated to the time-evolution of the statistical moments. The paper culminates in the formulation of the stochastic fundamental lemma for linear time-invariant systems, which in turn enables numerically tractable formulations of data-driven stochastic optimal control combining Hankel matrices in realization data (i.e. in measurements) with PCE concepts. 
\end{abstract}


\end{frontmatter}

\section{Introduction — Spring has arrived} \label{sec:Intro}
In popular culture \textit{winter is coming} is an established meme. In the context of the past and infamous AI winter~\citep{AIwinter} and its effect on research addressing data-driven methods, the recent rapid increase of interest and available results evidences the opposite. That is, we are currently experiencing an ongoing extremely vibrant growth period of research activities on data-driven methods for systems and control at least partially catalyzed by AI.  

From the control perspective, data-driven methods at large can be understood as an attempt to overcome the traditional two-step procedure of modeling---either first principles combined with parameter estimation or system identification based on measurement data---and subsequently designing a controller for the obtained system/process model, see Figure~\ref{fig:concept}. Building on the success and impact of subspace identification methods~\citep{verhaegen2015subspace,MarkovskyWillemsHuffelDeMoor06}, a recent stream of papers---which, e.g., includes~\cite{dePersis19,Favoreel99,Fiedler2021,Van20} and the tutorial survey~\cite{MarkDorf21}---has analyzed the so-called direct approach to control design, see  Figure~\ref{fig:concept}. This approach aims at obtaining a non-parametric system description directly from data, i.e. without identifying system matrices, and at the design of feedback laws directly using this description.

\begin{figure}[tb]
    \centering
    \begin{tikzpicture}[my triangle/.style={-{Triangle[width=\the\dimexpr1.8\pgflinewidth,length=\the\dimexpr0.8\pgflinewidth]}}]
    \footnotesize
    \node[rectangle, draw=none, fill=gray1, minimum height=2.5em, text centered, text width=7em, outer sep=2.5pt] (A) {Measure data};
    \node[right=of A, rectangle, draw=none, fill=gray1, minimum height=2.5em, text centered, text width=7em, outer sep=2.5pt] (B) {Identify system dynamics};
    \node[right=of B,rectangle, draw=none, fill=gray1, minimum height=2.5em, text centered, text width=7em, outer sep=2.5pt] (C) {Design \\ controller};
    \node[below=2pt of A] (A1) {$\{(u_k,y_k)\}_{k=0}^{T-1}$};
    \node[below=2pt of B] (B1) {$\begin{aligned}x_{k+1} & = A x_k + Bu_k\\ y_k&=Cx_k + Du_k\end{aligned}$};
    \node[below=2pt of C] (C1) {$u=Ky$};
    \node[below=5.5em of A, rectangle, draw=none, fill=gray1, minimum height=2.5em, text centered, text width=7em, outer sep=2.5pt] (A2) {Measure data};
    \node[below=5.5em of C,rectangle, draw=none, fill=gray1, minimum height=2.5em, text centered, text width=7em, outer sep=2.5pt] (C2) {Design \\controller};

    \node[above=1pt of B, text width=250pt] (D) {The established approach \textcolor{gray}{(indirect, two-step)}};
    \node[below=1pt of B1, text width=250pt] (E) {The data-driven approach \textcolor{gray}{(direct)}};
    
    \draw[line width=6pt,my triangle, draw=lightblue](A) -- (B);
    \draw[line width=6pt,my triangle, draw=lightblue](B) -- (C);
    \draw[line width=6pt,my triangle, draw=lightblue](A2) -- (C2);
\end{tikzpicture}
    \caption{Direct and indirect approaches to data-driven controller design.}
    \label{fig:concept}
\end{figure}

In turn, subspace identification techniques  are closely linked to behavioral concepts in systems and control, which have been conceived by \textit{Jan C. Willems}; see~\cite{willems86i,willems86ii,willems87} for original references,    \cite{PoldermanWillmes98,MarkovskyWillemsHuffelDeMoor06} for textbook expositions, and \cite{willems07} for an introduction. It stands to reason that behavioral systems theory is indeed foundational for a wide range of recent results on data-driven control~\citep{Markovsky21,MarkDorf21}.

\textit{The} pivotal result in this context  shows for controllable finite-dimensional Linear Time-Invariant (LTI) systems that all input-output-state trajectories of finite length lie in the column space of suitable Hankel matrices constructed directly from data. It appeared already in the seminal paper by~\cite{WRMDM05}---hence it is commonly referenced as \textit{the Fundamental Lemma of Willems' et al}.\footnote{We remark that the name \emph{fundamental lemma} as such was not used by~\cite{WRMDM05}; it appears to have been coined by \cite{katayama2006system}. }
Given the recent and seemingly exponential growth of publications extending/tailoring the fundamental lemma to specific system classes and using it for data-driven control,
we postpone a more detailed literature review with respect to the former to Section~\ref{sec:OverviewFundLem}. With respect to the latter, we remark that the two most frequent usages of the fundamental lemma for control design are output-feedback predictive control and direct feedback design~\cite{MarkDorf21}. The major conceptual advantage of output-feedback predictive control is that it alleviates the need to design state estimators. The earliest conception of Model Predictive Control (MPC) based on Hankel matrices, which did not receive widespread attention, appears to be~\cite{Yang15a}.
This line of research bifurcated to exponential growth with~\cite{coulson2019data} and numerous follow-ups. With respect to data-driven feedback design, we refer to~\cite{dePersis19,MarkDorf21} and the references therein. Both lines of research are linked by the concept of optimizing closed-loop behaviors, cf.\ \cite{Dorfler22a}. So far, applications of data-driven control based on behavioral concepts include power systems~\citep{Schmitz22,Huang19,Carlet20}, autonomous driving~\citep{Wang22}, and building control~\citep{Lian21,ODwyer21,BILGIC2022}.

Interestingly, the common scope of behavioral systems theory and its major success in data-driven control in discrete-time settings go along with the orthogonal trend of limited progress with respect to behavioral approaches for stochastic systems, cf. the open problems identified by~\cite{MarkDorf21}.
The ultimate paper of Jan C.\ Willems discusses behavioral ideas for stochastic systems~\citep{Willems12}. Therein the main focus is on open stochastic static systems, their interconnection, and the construction of appropriate $\sigma$-algebras. \cite{Baggio17} extend this to a canonical kernel representation of stochastic LTI processes.
Following a different route, \cite{Pola15a,Pola16a} use behavioral ideas to study equivalence concepts for stochastic linear systems in discrete-time and continuous-time settings without actually defining the stochastic behavior as such. This has been extended to descriptor systems, i.e. discrete-time LTI systems subject to linear algebraic constraints, by  \cite{Pola17a}.
Yet, none of these works covers data-driven representations of stochastic LTI systems.

Moreover, several approaches to behavioral concepts in infinite-dimensional settings have proposed: \cite{Pillai99a,Pillai99b} discuss behavioral kernel representations of systems with distributed parameters using an algebraic approach. This is extended to dissipative systems by \cite{Pillai02a}, while \cite{yamamoto2008behavioral} focus on controllability in a behavioral setting. 
\cite{ball2006conservative} discuss conservative realizations of linear systems on Hilbert spaces in a behavioral setting, while \cite{seiler2015algebraic} give an overview on algebraic theory for linear systems.
Argu\-ably, data-driven control of infinite-dimensional systems would require measurement data in appropriate infinite-dimensional spaces. Likewise  data-driven approaches to stochastic systems are seemingly limited in applicability as stochastic processes are typically described via cumulative distributions,  probability densities, or random variables---all of which are, in general, infinite dimensional.

Statistical moments, i.e., expectation, co-variance, skewness etc., provide an alternative representation of random variables, which is finite dimensional under rather specific assumptions on the underlying distribution. 
One may claim that the frequently-used modeling in terms of the first two moments (expectation and co-variance) does not stretch far beyond i.i.d. (independent and identically distributed) Gaussian uncertainties as moment closures for nonlinear systems are difficult and as manifold applications require modeling non-Gaussian uncertainties.  In turn, non-Gaussian random variables frequently induce the need for higher-order moments, see, e.g., \citep{Kuehn16,Singh10}. 
Even in cases where expectation and co-variance capture sufficient information about the uncertainty, the fact that they parameterize random variables in nonlinear fashion---i.e., any scalar Gaussian is given as the sum of its mean and square root of the variance times a standard normal distribution---often complicates their use.

Remarkably, already in the 1930s \textit{Norbert Wiener}'s most cited journal paper proposed an alternative avenue \citep{Wiener38}. Therein, Wiener suggested to represent random variables via series expansions expressed in suitable polynomial bases of the underlying probability space. The required structure is the Hilbert space of $L^2$-random variables, i.e., the linear function space equipped with an inner product which contains all random variables of finite variance. This approach is commonly denoted as \textit{Wiener chaos expansion}, as polynomial chaos, or as generalized polynomial chaos. For the sake of brevity, we gloss here over the subtle distinctions of these methods. The obtained series  are denoted as Polynomial Chaos Expansions~(PCE). Importantly, PCEs and related expansions parameterize a large class of Gaussian and non-Gaussian random variables in a linear structure, while moments lead, in general,  to nonlinear representations. 
Without any detailed elaboration, we remark that polynomial chaos approaches are established for uncertainty quantification and uncertainty propagation~\citep{o2013polynomial,Sullivan15}. Very often one leverages that the finiteness of the PCE is closely related to the appropriate choice of the basis, i.e., a suitable choice allows for a finite series expansion of non-Gaussian random variables~\citep{Xiu02,Muehlpfordt18}. PCEs have also seen widespread and continued use in systems and control, e.g., for system analysis~\citep{Nagy07,Kim13a,Ahbe20}, for stochastic MPC~\citep{Mesbah15a,Mesbah16a,fagiano2012nonlinear},  and for optimization in power systems~\citep{kit:muehlpfordt17b,kit:muehlpfordt18c}. 

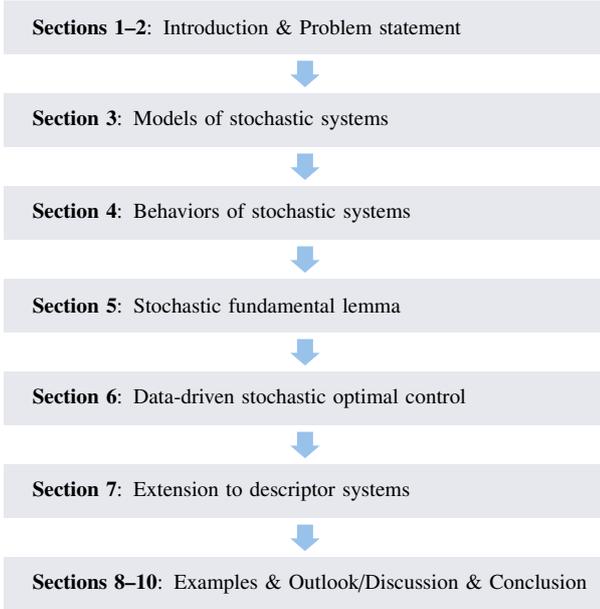
\begin{figure}[t]
    \centering
        \begin{tikzpicture}[my triangle/.style={-{Triangle[width=\the\dimexpr1.8\pgflinewidth,length=\the\dimexpr0.8\pgflinewidth]}}]
    \footnotesize
    \node[rectangle, draw=none, fill=gray1, minimum height=2.5em,  text width=27.5em, outer sep=2.5pt] (S1) {\quad \textbf{Sections 1--2}: Introduction \&
    Problem statement};
    \node[below=10pt  of S1, rectangle, draw=none, fill=gray1, minimum height=2.5em,  text width=27.5em, outer sep=2.5pt] (S3) {\quad\textbf{Section 3}: Models of stochastic  systems};
   \node[below =10pt of S3,rectangle, draw=none, fill=gray1, minimum height=2.5em, text width=27.5em, outer sep=2.5pt] (S4) {\quad\textbf{Section 4}: Behaviors of stochastic systems};
    \node[below =10pt  of S4,rectangle, draw=none, fill=gray1, minimum height=2.5em,  text width=27.5em, outer sep=2.5pt] (S5) {\quad\textbf{Section 5}: Stochastic  fundamental lemma};
   
    \node[below=10pt of S5,rectangle, draw=none, fill=gray1, minimum height=2.5em, text width=27.5em, outer sep=2.5pt] (S7) {\quad\textbf{Section 6}: Data-driven stochastic optimal control};
      \node[below =10pt  of S7, rectangle, draw=none, fill=gray1, minimum height=2.5em, text width=27.5em, outer sep=2.5pt] (S6) {\quad\textbf{Section 7}: Extension to  descriptor systems};
    \node[below =10pt of S6,rectangle, draw=none, fill=gray1, minimum height=2.5em, text width=27.5em, outer sep=2.5pt] (S8) {\quad\textbf{Sections 8--10}: Examples \&  Outlook/Discussion \& Conclusion};

    \draw[line width=6pt,my triangle, draw=lightblue](S1.south) -- (S3);
      \draw[line width=6pt,my triangle, draw=lightblue](S3) -- (S4);  
      \draw[line width=6pt,my triangle, draw=lightblue](S4) -- (S5);
      \draw[line width=6pt,my triangle, draw=lightblue](S5) -- (S7);
      \draw[line width=6pt,my triangle, draw=lightblue](S7) -- (S6);
      \draw[line width=6pt,my triangle, draw=lightblue](S6) -- (S8);

\end{tikzpicture}
    \caption{Illustration of contents.}
    \label{fig:Structure}
\end{figure}

In light of the growing success of behavioral systems theory in data-driven control this  paper moves towards behavioral and data-driven concepts for stochastic linear systems. Hence, it extends the scope of the earlier and excellent overview \citep{MarkDorf21}, which has explicitly called for further investigation of the stochastic setting. Moreover, \cite{MarkDorf21} provide a detailed exposition on issues and remedies surrounding the use of data corrupted by measurement noise and on the estimation of missing data. Hence, we do not revisit these aspects here. Specifically, the outline of the present paper is as follows (see also Figure~\ref{fig:Structure}): 

In Section~\ref{sec:recapI} we recapitulate the deterministic fundamental lemma and we give a problem statement for stochastic systems. Moreover, we provide a concise overview on \textit{Willems}-inspired fundamental-lemma type results. 
To support the further developments, Section~\ref{sec:model_Stochastic} turns to model-based representations of stochastic dynamics, it recalls the basics of \textit{Wiener} chaos expansions and PCE, and it revisits the co-variance propagation for LTI systems.

In Section~\ref{sec:StochBehav} the insights of \textit{Willems and Wiener are combined} to derive our main contributions, i.e., we discuss different behavioral characterizations of stochastic systems. We introduce the behavior of the dynamics of the series expansion coefficients (in short \textit{expansion coefficient behavior}) before we turn to the behavior of the dynamics in random variables (in short \textit{random variable behavior}) and to the  behavior associated to the time-evolution of the statistical moments (in short \textit{moment behavior}). We also introduce the concept of behavioral lifts, i.e., we characterize the underlying relations between the behaviors referring to expansion coefficients, to random variables, and to realizations. Moreover, we show why and how moment characterisations fall short in the behavioral context. 

Section~\ref{sec:StochFundLem} then continues to data-driven representations  via the  fundamental lemma tailored to stochastic systems~\citep{Pan21s} before the results are extended to stochastic descriptor systems in the subsequent section. The crucial observation, which carries over from~\cite{Pan21s}, is that due to the behavioral lift the Hankel matrices can be constructed from measurements of realization data (or sampled/measured trajectories) while the PCE framework allows to capture random variables exactly.

We prove that, under rather mild assumptions, for discrete-time stochastic systems  measurements of finite-dimensional realization data (i.e. sampled trajectories of finite length) suffice to characterize the evolution of stochastic input, state, and output variables via an appropriate image representation obtained in the fundamental lemma for stochastic systems.\footnote{Notice that upon observing or measuring stochastic processes, one obtains information about specific realizations of random variables. Hence, subsequently  \textit{realization} refers to sampled outcomes of the uncertainties as modelled by random variables.   For the remainder of this manuscript, the word realizations should not be confused with the classic notion of state-space realizations of systems.  } 
To put it in the words of \cite{Willems12}, our findings support the viewpoint  that 
\begin{displayquote}
\textit{[in terms of behavioral systems theory and in terms of the fundamental lemma] deterministic systems emerge as special cases of stochastic systems, as they should},
\end{displayquote}
which is also fully aligned with~\cite{Baggio17}.

Section~\ref{sec:StochOCPs} discusses data-driven stochastic optimal control, i.e., the computation of optimal non-anticipatory policies. Further, model-based analysis is put into action to ensure causality of the obtained optimal solutions.

Section~\ref{sec:ext} transfers the obtained insights on the stochastic fundamental lemma and on data-driven stochastic optimal control to descriptor systems.
While from a model-based viewpoint descriptor LTI representations are more general than explicit ones, from the behavioral perspective both are (modulo the choice of inputs and outputs) equivalent representation instances of linear time-invariant systems, cf.~\citep{willems86i,willems07}.
We show that it is the causality requirement of stochastic optimal control that marks the watershed between descriptor and explicit LTI systems in data-driven settings. Put differently, in data-driven settings with underlying descriptor structures the computation of non-anticipatory control actions requires careful analysis.

Section~\ref{sec:Examples} draws upon numerical examples to illustrate our findings. 
The paper closes with an outlook on open problems and conclusions in Sections~\ref{sec:Discussion} and~\ref{sec:Outlook}.

\noindent \textbf{Notation}. %
The non-negative and the positive integers are denoted by $\mathbb{N}$ and $\mathbb Z_+$, respectively.  For $n,m\in\mathbb N$ with $n\leq m$ we define $\mathbb I_{[n,m]}\doteq [n,m]\cap \mathbb N$. For two sets~$\Omega$ and $M$, $\Omega^M$ denotes the set of all mappings $f: M \rightarrow \Omega$.  Further, for two sets $\Omega_1$ and $\Omega_2$, we identify the Cartesian product $\Omega_1^M\times \Omega_2^M$ with $(\Omega_1\times\Omega_2)^M$. The restriction of a function $f\in\Omega^M$ to a subset $M_0\subset M$ is denoted by $f|_{M_0}$. In the case $M=\mathbb I_{[n,m]}$ we use for $f\in\Omega^M$ the notation $f_k\doteq f(k)$ and, when there is no possibility of confusion, also $f^k=f(k)$ for $k\in M$. Moreover, for $f\in \Omega^M$, where $\Omega=\mathbb R^d$ and $M\subset \mathbb N$ such that $\mathbb I_{[n,m]}\subset M$, we define $\mathbf f_{[n,m]}\doteq\begin{bmatrix}f_n^\top & \dots & f_m^\top \end{bmatrix}{}^\top$.
The identity matrix and zero matrix are designated by $I_{n}$ and $0_{n\times m}$. For a matrix $A$ we use the notations $\operatorname{rk}(A)$, $\operatorname{colsp}(A)$, and $A^\dagger$ for the rank, the column span, and the Moore--Penrose inverse of $A$, respectively. The Kronecker product of two matrices $A$ and $B$ is denoted by $A\otimes B$.

\section{Preliminaries and Problem Statement}
\label{sec:recapI}

We first recall some basic concepts of behavioral systems theory and briefly revisit the fundamental lemma. We refer the reader to~\cite{MarkDorf21} for a recent and much-more-detailed survey on behavioral systems theory and its use for data-driven control of deterministic systems. For readers familiar with this survey, the first subsection only sets the notation used hereafter. 
Then, we give a problem statement and point out upcoming novel results in the stochastic setting. Finally, we conclude this section with a concise overview of fundamental-lemma-type results.

From an abstract point of view, a system $\Sigma$ is a triplet $\Sigma=(\mathbb{T}, \mathbb W, \mathfrak B)$ with time axis $\mathbb T\subseteq\mathbb R$, signal space~$\mathbb W$, and behavior~$\mathfrak B\subseteq \mathbb W^{\mathbb T}$. 
Specifically, we consider linear time-invariant discrete-time dynamics represented by
\begin{subequations}\label{sys}
    \begin{align}
        x_{k+1} &= A x_k + B u_k \label{sysa} \\
        y_k &= Cx_k + D u_k, \label{sysb}
    \end{align}
\end{subequations} where $x_k\in\mathbb R^{n_x}$, $u_k\in\mathbb R^{n_u}$, and $y_k\in\mathbb R^{n_y}$ are the state, the input, and the output at time instant $k \in \mathbb N$, respectively. Further, we have $A\in\mathbb R^{n_x\times n_x}$, $B\in\mathbb R^{n_x\times n_u}$, $C\in\mathbb R^{n_y\times n_x}$, and $D\in\mathbb R^{n_y\times n_u}$. In this situation the time-axis is given by $\mathbb T=\mathbb N$, and the signal space by $\mathbb W = \mathbb R^{n_x+n_u+n_y}$ or, if only input and output variables are considered, by $\mathbb W=\mathbb R^{n_u+n_y}$.

Subsequently, we may occasionally deviate from the rigorous behavioral wording, that is whenever no confusion can arise we use the word \textit{system} to  refer to a system or to its representation in synonymous fashion. The dynamics~\eqref{sys} can be described in the language of behavioral systems theory advocated by \cite{PoldermanWillmes98}, see also \cite{Willems91,MarkovskyWillemsHuffelDeMoor06}. The \emph{(full) behavior} of  system~\eqref{sys} is
\begin{equation}
    \mathfrak B_\infty \doteq \left\{(x,u,y)\,\middle|\, \begin{gathered} 
         x\in(\mathbb R^{n_x})^{\mathbb N}, u\in(\mathbb R^{n_u})^{\mathbb N}, y\in(\mathbb R^{n_y})^{\mathbb N}
         \\ (x_{k+1},x_k,u_k,y_k) \text{ satisfies \eqref{sys} } \forall k\in\mathbb N
        \end{gathered}\right\}
\end{equation}
which contains all state-input-output trajectories of~\eqref{sys}. The trajectories of finite horizon $T \in \mathbb N$ are collected in the finite-horizon behavior
\begin{equation}
    \label{finHorbehav}
   \mathfrak B_{T} = \bigl\{b|_{[0,T]}\,\big|\, b\in\mathfrak B_\infty\bigr\}.
\end{equation}
For all $T \in \mathbb{N} \cup \{\infty\}$, the set $\mathfrak B_{T}$ is a vector space. Furthermore, the behavior $\mathfrak B_\infty$ is \emph{complete}, i.e.,\ 
$b|_{[0,T]}
\in \mathfrak B_{T}$ for all $T\in\mathbb N$ implies $b\in\mathfrak B_\infty$.

An appealing aspect of the behavioral perspective is the analysis of properties of a system $(\mathbb{T}, \mathbb W, \mathfrak B)$ can done without reference its representation \eqref{sys}. Specifically,  controllability, cf. \cite{MarkovskyWillemsHuffelDeMoor06} and \cite{WoodZerz99}, can be characterized by the following behavioral definition. 
\begin{definition}[Behavioral controllability] \label{def:behav_controlability}
    $\mathfrak B_\infty$ is said to be \emph{controllable} if for each two trajectories $b$,~$\tilde b\in\mathfrak B_\infty$ and every $T\in\mathbb Z_+$ there exists $b'\in\mathfrak B_\infty$ and $T'\in\mathbb N$ such that \begin{equation}
    \label{intermediate}
    b'|_{[0,T-1]} = b|_{[0,T-1]},\quad b'|_{[T+T',\infty)} = \text{$\tilde{b}|_{[0,\infty)}$}.
\end{equation}
\end{definition}
Vividly, controllability of~$\mathfrak B_\infty$ allows to switch from one trajectory to another, however, with some \emph{transition delay~$T'$}. The behavior~$\mathfrak B_\infty$ is controllable if and only if the system representation~\eqref{sys} is controllable. In this case the transition delay is bounded by $T' = n_x$; independently from the particular trajectories.

In general, the state variable $x$ is latent and only input-output trajectories of~\eqref{sys} are directly accessible through measurements. The input-output trajectories are collected in the \emph{manifest behavior}
\begin{equation*}
    \mathfrak B_\infty^\text{i/o} = \left\{(u,y) \,\middle|\, (x,u,y)\in\mathfrak B_\infty\text{ for some }x\in (\mathbb R^{n_x})^{\mathbb N}\right\},
\end{equation*}
while the input-output trajectories of finite length $T\in\mathbb N$ yield
\begin{equation}
    \label{finHorbehavIO}
 \mathfrak B_{T}^\text{i/o} = \bigl\{b|_{[0,T]}\,\big|\, b\in\mathfrak B_\infty^\text{i/o}\bigr\}.
\end{equation}
Further, there exists a lower bound on the length of input-output trajectories such that the state~$x$ is uniquely determined assuming observability, i.e., if $(x,u,y)$, $(\tilde x, u, \tilde y)\in\mathfrak B_{n_x-1+T}$ for some $T\in\mathbb N$ such that $y|_{[0,n_x-1]}=\tilde y|_{[0,n_x-1]}$, then $x|_{[0,n_x-1+T]} = \tilde x|_{[0,n_x-1+T]}$.

\subsection{The Fundamental Lemma} 

The fundamental lemma by \cite{WRMDM05} proposed a non-parametric representation of linear time-invariant systems by means of Hankel matrices containing only data available through measurements, which carries sufficiently rich information. This richness is specified in the following definition.
\begin{definition}[Persistency of excitation]
    The input trajectory $u:\{0,\dots,T-1\}\rightarrow \mathbb R^{n_u}$ is said to be \emph{persistently exciting} of order $L$, where $L\in\mathbb N$ with $T\geq L(n_u+1) -1$, if the \emph{Hankel matrix}
    \begin{equation*}
        \mathcal H_L(\mathbf u_{[0,T-1]}) \doteq \begin{bmatrix}
            u_0 & \dots &  u_{T-L}\\
            \vdots & \ddots& \vdots\\
            u_{L-1} & \dots & u_{T-1}
        \end{bmatrix}\in\mathbb R^{n_uL\times (T-L+1)}
    \end{equation*}
    has row rank $n_uL$.
\end{definition}
The original statement of the fundamental lemma by \cite{WRMDM05} is given in a behavioral setting. 
Here, we focus only on input-output trajectories. A corresponding statement including the state variables can be derived by imposing stronger assumptions on the output, i.e. $C=I$, cf. \cite{WDPCT20}.
\begin{lemma}[Fundamental lemma]
    \label{lem:FL_des}
    Suppose that the behavior $\mathfrak B_\infty$ of the system $(\mathbb{T}, \mathbb W, \mathfrak B_{\infty})$ is controllable. Let $(u,y)\in\mathfrak B^\text{i/o}_{T-1}$ be such that $u$ is persistently exciting of order $L+n_x$, then $(\tilde u, \tilde y)\in\mathfrak B^\text{i/o}_{L-1}$ if and only if there exists $g\in\mathbb R^{T-L+1}$ such that
    \begin{equation} \label{eq:fundlem_exp}
        \begin{bmatrix}
            \mathbf{\tilde u}_{[0,L-1]}\\
            \mathbf{\tilde y}_{[0,L-1]}
        \end{bmatrix} = 
        \begin{bmatrix}
            \mathcal H_{L} (\mathbf{u}_{[0,T-1]})\\
            \mathcal H_{L} (\mathbf{y}_{[0,T-1]})
        \end{bmatrix} g.
    \end{equation}
\end{lemma}
The appeal of the above results is that the Hankel matrices can be constructed directly from measured data. 

\subsection{Overview of Deterministic Variants of the Lemma}
\label{sec:OverviewFundLem}

\begin{table*}[t]
\caption{Overview of recent fundamental lemma type results. \vspace*{-3mm}
}
\label{tab:overview}
\begin{center}
\scalebox{0.8}{
\begin{tabular}{lll}
\toprule
 System Class & Comments &  References  \\ 
 \midrule
linear&       &\cite{WRMDM05}, \cite{dePersis19}\\
 &    data segmentation & \cite{WDPCT20}\\
  &   affine structure & \cite{Martinelli22}\\
 & uncontrollable & \cite{Yu21} \\
 &  descriptor& \cite{SchmitzFaulwasserWorthmann22} \& Section~\ref{sec:ext}\\
  &   stochastic& \cite{Pan21s} \& Section~\ref{sec:StochFundLem}\\
   &   stochastic, descriptor& Section~\ref{sec:ext}
   \\
  &  linear parameter-varying&     \cite{Verhoek21}\\
  &   linear time-varying   & \cite{Nortmann21} \\
   &   time delay   & \cite{rueda2021data} \\
\midrule
 non-linear &  bi-linear &  \cite{Yuan22}\\
 & differentially flat &\cite{Alsalti21}\\
    & approximate  trajectories via kernel regression &
  \cite{Huang22, lian2021nonlinear}\\
    &     Wiener-Hammerstein & \cite{Berberich20}\\
    &     Volterra systems & \cite{rueda2020data}\\
    &     polynomial nonlinearities & \cite{Markovsky21}\\
    \bottomrule
\end{tabular}
 }
\end{center}
\end{table*}

Naturally, the previous exposition motivates to ask whether the fundamental lemma can be tailored or extended to other settings and how to formalize such extensions? Table~\ref{tab:overview} gives an overview on both aspects. As one can see, there exist manifold variants and extensions of the original result of~\cite{WRMDM05}. This includes more easily accessible proofs in the explicit state-space setting~\citep{dePersis19}, the extension to data segmentation in the Hankel matrices~\citep{WDPCT20}, the relaxation of the usual controllability assumption~\citep{Yu21}, and the consideration of affine structures, i.e. linear dynamics with constant additive inhomogeneities. There also have been recent refinements to linear descriptor systems---see~\cite{SchmitzFaulwasserWorthmann22} and Section~\ref{sec:ext}---and extensions to linear stochastic systems---cf.~\cite{Pan21s} and Section~\ref{sec:StochFundLem}.

Besides the usual LTI setting, there are tailored variants for linear parameter-varying systems~\citep{Verhoek21}, for linear time-varying ones~\citep{Nortmann21}, and for linear time delay systems~\citep{rueda2021data}. Likewise, the nonlinear domain is of considerable research interest as \cite{Alsalti21} cover discrete-time non-linear differentially flat systems,  \cite{lian2021nonlinear,Huang22} suggest using reproducing kernel Hilbert spaces, \cite{Berberich20} investigate Wiener-Hammerstein systems, while \cite{rueda2020data} address second-order Volterra systems. Recently, \cite{Markovsky21} has proposed a general framework for time-invariant systems with polynomial non-linearities. 

We conclude this overview with a crucial observation:  while the issues surrounding measurement noise in fundamental lemma type results have seen manifold research activities---cf.,  e.g., \cite{coulson2019regularized,yin2021data,yin2021maximum}---the question of hand\-ling stochastic uncertainties has received much less attention. Specifically, we mention 
that \cite{Dorfler21} discuss certainty equivalence LQR design. However, while a data-driven surrogate of the closed-loop  state-feedback system matrix $A+BK$ is provided therein, co-variance propagation as such is not discussed. Moreover, the stochastic variant of the lemma  proposed by~\cite{Pan21s} is revisited later. 

\subsection{Problem Statement}
The main object of our further investigations are linear time-invariant systems subject to stochastic uncertainties a typical representation of which reads
    \begin{align*}
    	X\inst{k+1} &= AX\inst{k} +BU\inst{k}+FW\inst{k}\\
   		Y\inst{k} &= C X\inst{k}+ D U\inst{k} + HW\inst{k}.
    \end{align*}
The game changer moving from \eqref{sys} to the representation above are the stochastic disturbances $W$ appearing in the state and output equations. Their occurrence induces the need to model the states $X$, the outputs $Y$, and the controls  $U$ also as appropriate stochastic processes. 
Indeed, one can show   that finite dimensional series descriptions of exogenous stochastic uncertainties enable the formulation of a stochastic fundamental lemma~\citep{Pan21s}.\footnote{A related structure wherein the Hankel matrices are expanded by disturbance data is used by~\cite{Kerz21}. However, therein no means of propagating stochastic uncertainties in data-driven fashion is provided.}
This earlier analysis is largely driven by the properties of the system representations. It does, however, not answer the long standing question of how to conceptualize the behavior of stochastic linear systems, cf. \citep{Willems12,Baggio17}.
Subsequently we address this point and we provide insights on the following aspects:
\begin{itemize}
    \item  How to define and to characterize the behavior of stochastic linear systems (Section \ref{sec:StochBehav})?
    \item How to formulate a stochastic fundamental lemma (Section \ref{sec:StochFundLem})?
        \item How to formulate stochastic optimal control problems in data-driven fashion without jeopardizing causality (Section \ref{sec:StochOCPs})?
    \item How to handle causality and Markov properties if the underlying structure fixes the choice of inputs and outputs and thus descriptor representations arise (Section \ref{sec:ext})?
\end{itemize}
\section{Representations of Stochastic Linear Systems}
\label{sec:model_Stochastic}

As mentioned above, we model the trajectories of the system as stochastic processes. Specifically, let $(\Omega,\mathcal F,\prob)$ be a probability space consisting of the sample space~$\Omega$ (a.k.a.\ set of outcomes), the $\sigma$-algebra~$\mathcal F$, and the probability measure~$\prob$. The space $L^2((\Omega,\mathcal F,\prob),\mathbb R^d)$ consists of (equivalence classes of) vector-valued random variables with realizations in $\mathbb R^d$, $d\in\mathbb N$, and finite variance. For $X$,~$Y\in L^2((\Omega,\mathcal F,\prob),\mathbb R^d)$ the expected value and the covariance matrix are 
\begin{gather*}
    \mean [X] = \int_{\Omega} X(\omega) \,\mathrm d\prob(\omega),\\
    \covar[X,Y] = \mean [XY^\top] - \mean [X] (\mean [Y])^\top.
\end{gather*}
Equipped with the usual scalar product defined by $\langle X,Y\rangle = \mean[X^\top Y]$ and its induced norm given by $\lVert X\rVert=\sqrt{\langle X,X\rangle}$ the space $L^2((\Omega,\mathcal F,\prob),\mathbb R^d)$ is a Hilbert space. 
Recall that for $X,Y\in L^2(\Omega,\mathcal F,\prob,\mathbb R^d)$ the equality $X=Y$ is to be understood in the $L^2$-norm sense, $\lVert X-Y\rVert=0$, or equivalently $X(\omega)=Y(\omega)$ for $\prob$-almost all ($\prob$-a.a.) $\omega\in\Omega$, i.e.\ all $\omega\in\Omega$ but those in a $\prob$-nullset. For stochastic processes $X=(X_k)_{k\in\mathbb N}$,~$Y=(Y_k)_{k\in\mathbb N}\in L^2((\Omega,\mathcal F,\prob),\mathbb R^d)^{\mathbb N}$ the expectation operator $\mean$ and the covariance operator $\covar$ are defined step-wise in time, i.e.\ $\mean [X] = (\mean[X_k])_{k\in\mathbb N}$ and $\covar[X,Y]=(\covar[X_k,Y_k])_{k\in\mathbb N}$.

Subsequently, we assume that the probability space $(\Omega,\mathcal F,\prob)$ is \emph{separable}, i.e., the $\sigma$-algebra $\mathcal F$ is generated by a countable family of sets $\Omega_k \subset \Omega$, $k \in \mathbb{N}$. 
Then $L^2((\Omega,\mathcal F,\prob),\mathbb R^d)$ is a separable Hilbert space and it possesses a countable orthogonal basis, cf.\ \cite{Brezis11}.

To the end of recalling representations of stochastic linear systems, we consider 
\begin{subequations}\label{eq:RVdynamics}
    \begin{align}
    	X\inst{k+1} &= AX\inst{k} +\widetilde BV\inst{k}\\
        \label{eq:RVdynamicsb}
   		Y\inst{k} &= C X\inst{k}+\widetilde D V\inst{k}
    \end{align}
\end{subequations}
where the state, exogenous input, and output signals are modelled as stochastic processes, that is $X_k\in L^2((\Omega,\mathcal F,\prob),\mathbb R^{n_x})$, $Y_k\in L^2((\Omega,\mathcal F,\prob),\mathbb R^{n_y})$, and $V_k\in L^2((\Omega,\mathcal F,\prob),\mathbb R^{n_v})$. Notice that the variable~$V$ contains all exogenous inputs of the system, i.e., the manipulated control inputs~$U$ as well as the exogenous process disturbances~$W$ are related by
\begin{equation}\label{eq:inputsplit}
    \begin{bmatrix}\widetilde B \\
     \widetilde D \end{bmatrix} V = \begin{bmatrix}B & F \\
    D & H \end{bmatrix}\begin{bmatrix}U \\ W \end{bmatrix}.
\end{equation} 
For many of our formal developments, we do not elaborate on the distinction between controls~$U$ and disturbances~$W$ as this fits well to the behavioral viewpoint. Yet, from a stochastic systems point of view further considerations are in order.\vspace*{3mm}

\noindent \textbf{Stochastic Filtrations and the Markov Property}.
    The state variable of system~\eqref{eq:RVdynamics} possesses the Markov property, i.e., the state $X_{k+1}$ depends only on $X_k$ and $V_k$. Specifically, let $X=(X_k)_{k\in\mathbb N}$ be a solution to~\eqref{eq:RVdynamics} together with its natural filtration 
    $(\sigma(X_0,\dots, X_k))_{k\in\mathbb N}$, where $\sigma(X_0,\dots,X_k)$ denotes the $\sigma$-algebra generated by $X_0,\dots, X_k$. We think of a control law which assigns the new input action $U_{k}$ on the basis of the current value of the state $X_k$, i.e.\ $U_{k}=K_k(X_{k})$ with some measurable map $K_k$. For the exogenous noise assume that $W_0,W_1,\dots$ and $X_0$ are mutually independent. Then $X$ satisfies the Markov property
    \begin{equation}
    \label{eq:markov1}
        \prob[X_{k+1}\in\mathcal A\,|\, \sigma(X_0,\dots, X_k)] = \prob[\text{$X_{k+1}$}\in\mathcal A\,|\, \sigma(X_k)]
    \end{equation}
    for every Borel-measurable set $\mathcal A\subset \R^{n_x}$. The conditional probability in \eqref{eq:markov1} (see e.g. \cite{klenke13prob}) describes the chance that the event $\{\omega\in\Omega\,|\, X_{k+1}(\omega)\in \mathcal A\}$ occurs under the (additional) information carried by a sub-$\sigma$-algebra of $\mathcal F$ generated by predecessors of $X_{k+1}$. The Markov property~\eqref{eq:markov1} can be rephrased as
    \begin{align*}
        \prob[X_{k+1} & \in \mathcal A
        \,|\, X_0\in\mathcal A_0,\dots, X_k\in\mathcal A_k]
        = \prob[X_{k+1} \in\mathcal A
        \,|\, X_k\in\mathcal A_k].
    \end{align*}
    Similarly, a Markov property in terms of the output can be shown. Provided that the system~\eqref{eq:RVdynamics} is observable, the value of the latent state can be revealed by an output signal whose length matches the system order $n_x$. Therefore, the stochastic process $\Gamma=(\Gamma_k)_{k\in\mathbb N}= (\mathbf Y_{[k,k+n_x]})_{n\in\mathbb N}$ together with its natural filtration $(\sigma(\Gamma_0,\dots,\Gamma_k))_{k\in\mathbb N}$ is Markovian as well.

\vspace*{3mm}
\noindent \textbf{Realizations and Moments}. 
A common way to handle the stochastic system~\eqref{eq:RVdynamics} is via path-wise dynamics.  The outcome (sample) $\omega \in \Omega$ implicitly defines the realizations 
\[
    x_k=X_k(\omega),\quad y_k=Y_k(\omega), \quad\text{and}\quad v_k=V_k(\omega).
\]
Hence, we arrive at the realization dynamics
\begin{subequations}\label{eq:Realdynamics}
    \begin{align}
    	x\inst{k+1} &= Ax\inst{k} +\widetilde B v\inst{k}, \\
   		y\inst{k} &= C x\inst{k}+\widetilde D v\inst{k}.
    \end{align}
\end{subequations}
Modulo the notation change from $u$  to $v$, respectively, $V$ in \eqref{eq:RVdynamics}, this resembles the previous deterministic system~\eqref{sys}.
Put differently, any realization (i.e.\ sampled trajectory) triplet $(x, v, y)$ satisfies~\eqref{eq:Realdynamics}. Yet, without oracle-like knowledge of future realizations of $v_i = V_i(\omega), i \geq k$, and leaving sampling-based approaches aside, the realization dynamics are mostly helpful in an a-posteriori sense.

Alternatively, passing over to statistical moments also allows to describe the system dynamics. Even though for moments of first order the equations remain linear, higher-order moments result in non-linear equations with respect to the random variables. Specifically, for first- and second-order moments we obtain the usual  propagation
\begin{subequations}\label{eq:Momentdynamics}
    \begin{align}
    	 \mean [X\inst{k+1}] &= A\mean[X\inst{k}] +\widetilde B \mean[V\inst{k}]\\
   		\mean[Y\inst{k}] &= C \mean[X\inst{k}]+\widetilde D \mean[V\inst{k}] \\
		\covar[X_{k+1}, X_{k+1}] &= \begin{bmatrix} A \\ \widetilde B \end{bmatrix}^\top \begin{bmatrix} \covar[X_{k}, X_{k}]\phantom{^\top} & \covar[X_{k}, V_{k}] \\ \covar[X_{k}, V_{k}]^\top  & \covar[V_{k}, V_{k}]\end{bmatrix}\begin{bmatrix} A \\ \widetilde B\end{bmatrix}\\
		\covar[Y_{k}, Y_{k}] &= \begin{bmatrix} C \\ \widetilde D\end{bmatrix}^\top \begin{bmatrix} \covar[X_{k}, X_{k}]\phantom{^\top} & \covar[X_{k}, V_{k}] \\ \covar[X_{k}, V_{k}]^\top  & \covar[V_{k}, V_{k}]\end{bmatrix}\begin{bmatrix} C \\ \widetilde D\end{bmatrix}.
    \end{align}
\end{subequations}

\subsection{Representations of $L^2$-Random Variables}

At this point, it is fair to ask which benefits the chosen setting of $L^2$-random variables actually implies. 
As we will see below, this approach enables a linear representation of Gaussian and non-Gaussian random variables, which turns out to also be numerically tractable---under suitable assumptions. 

In the following and whenever there is no ambiguity, we use the short-hand notation $L^2(\Omega, \mathbb{R}^d)$ instead of $L^2((\Omega,\mathcal F,\prob),\mathbb R^d)$. The space $L^2(\Omega,\mathbb R^{d})$ can be identified with the orthogonal sum 
\[
    \bigoplus_{i=1}^{d} L^2(\Omega,\mathbb R).
\] 
Therefore, any fixed scalar-valued orthogonal basis $(\phi^i)_{i\in\mathbb N}$ of $L^2(\Omega,\mathbb R)$ gives rise to a series expansion for each random vector $X \in L^2(\Omega,\mathbb R^d)$, i.e.\ $X=\sum_{i\in\mathbb N} \mathsf x^i \phi^i$, where the series converges in the $L^2$-sense. The coefficients are uniquely determined by the quotient
\[
    \mathsf x^i =\dfrac{1}{\mean[\phi^i\phi^i]}\mean [\phi^i X]\in \mathbb R^{d}.
\] 
Moreover, the sequence of coefficients $\mathsf x=(\mathsf x^0,\mathsf x^1,\dots)$ is square summable, i.e.
\[
    \mathsf x\in \ell^2(\mathbb R^d)\doteq \left\{\mathsf{\tilde x}\in (\mathbb R^d)^{\mathbb N}\,\middle|\, \sum_{i\in\mathbb N} (\mathsf{\tilde x}^i)^\top \mathsf{\tilde x}^i<\infty\right\}.
\]
Using this notation, we obtain
\begin{equation}\label{expansion}
	Z_k = \sum_{i\in\mathbb N} \mathsf z_k^i \phi^i \quad\text{with}\quad \mathsf z_k^i = \frac{\mean [\phi^i Z_k]}{\mean[\phi^i\phi^i]},\
	\mathsf z_k \in \ell^2(\mathbb R^{n_z}),\quad 
\end{equation}
for $(Z,\mathsf z, n_z) \in \{(X,\mathsf x, n_x), (V,\mathsf v, n_v), (Y,\mathsf y, n_y)\}$ and all $k \in \mathbb N$. The $L^2$-formulation of the dynamics~\eqref{eq:RVdynamics} combined with \eqref{expansion} gives
\begin{subequations}\label{eq:Coeffdynamics}
    \begin{align}
    	 \mathsf x^i\inst{k+1} &= A\mathsf x^i\inst{k} +\widetilde B\mathsf v^i\inst{k}\\
   		\mathsf y^i\inst{k} &= C \mathsf x^i\inst{k}+\widetilde D \mathsf v^i\inst{k}
    \end{align}
\end{subequations}
for all $k\in\mathbb N$ and~$i\in\mathbb N$. The transition from~\eqref{eq:RVdynamics} to~\eqref{eq:Coeffdynamics} using~\eqref{expansion} is referred to as \textit{Galerkin projection} as it technically means to project~\eqref{eq:RVdynamics} onto the basis functions $\phi^i$ used in \eqref{expansion}, cf.\ \cite{GhanSpan03, Pan21s}.

To sum up and as depicted in Figure~\ref{fig:modelOverview}, the stochastic dynamics~\eqref{eq:RVdynamics} admit at least four different model-based representations:
\begin{itemize}
    \item in $L^2$-random variables \eqref{eq:RVdynamics},
    \item in realizations~\eqref{eq:Realdynamics}, i.e.\ sampled trajectories, 
    \item in statistical moments \eqref{eq:Momentdynamics}, and
    \item in series expansion coefficients \eqref{eq:Coeffdynamics}. 
\end{itemize}
At this point it is fair to ask for the advantages and disadvantages of the representations listed above. The dynamics in $L^2$-random variables~\eqref{eq:RVdynamics} are a very general setting, which goes well beyond the usual linear-Gaussian setting. However, while conceptually allowing for exact forward propagation, any numerical implementation has to address the fact that $L^2$-random variables as such are, in general, infinite-dimensional objects living in a Hilbert space. 

In contrast, the realization dynamics \eqref{eq:Realdynamics} are finite dimensional. Yet, forward propagation via \eqref{eq:Realdynamics} requires either unrealistic knowledge of future disturbance realizations or  sampling-based strategies~\citep{campi2009scenario,campi2018introduction}. The latter are subject to approximation errors and often do not scale particularly well for non-Gaussians. The representation in terms of statistical moments~\eqref{eq:Momentdynamics} is a standard approach in systems and control. One should note that most frequently only the first two moments (expectation and covariance) are considered as this suffices for Gaussians. Yet, in non-Gaussian settings this might induce a loss of information. Moreover, one may regard the research efforts around distributional robustness and Wasserstein uncertainty sets as an implicit indication of the pitfalls of moment-based characterizations of stochastic distributions~\citep{givens1984class,mohajerin2018data}. We show in Section~\ref{sec:StochBehav} that in terms of behavioral characterizations of stochastic systems, moments are of limited usefulness. 
Also notice that for the linear stochastic system~\eqref{eq:RVdynamics}, the dynamics of the higher-order moments are of different structure then the original system. In contrast, for the dynamics of series expansions coefficients obtained via Galerkin projection, i.e. \eqref{eq:Coeffdynamics}, we observe that for all expansion coefficients the structure of the dynamics remains the same as the original dynamics in random variables \eqref{eq:RVdynamics} and as the realization dynamics \eqref{eq:Realdynamics}. As such this is not surprising as linear dynamics are combined with the linear structure of a series expansion. Additionally, as we discuss below, under rather mild assumptions and for a rich class of stochastic distributions the exactness of the series expansion with finitely many terms can be guaranteed.

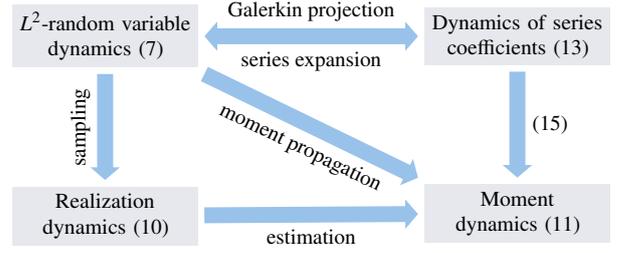
\begin{figure}[t]
    \centering
        \begin{tikzpicture}[my triangle/.style={-{Triangle[width=\the\dimexpr1.8\pgflinewidth,length=\the\dimexpr0.8\pgflinewidth]}}, my double triangle/.style={{Triangle[width=\the\dimexpr1.8\pgflinewidth,length=\the\dimexpr0.8\pgflinewidth]}-{Triangle[width=\the\dimexpr1.8\pgflinewidth,length=\the\dimexpr0.8\pgflinewidth]}}]
    \footnotesize
    \node[rectangle, draw=none, fill=gray1, minimum height=2.5em, text centered, text width=8em, outer sep=2.5pt] (A) {$L^2$-random\ variable\ dynamics~\eqref{eq:RVdynamics}};
    \node[right=80pt of A, rectangle, draw=none, fill=gray1, minimum height=2.5em, text centered, text width=8em, outer sep=2.5pt] (B) {Dynamics of series coefficients~\eqref{eq:Coeffdynamics}};
    \node[below= 40pt of A,rectangle, draw=none, fill=gray1, minimum height=2.5em, text centered, text width=8em, outer sep=2.5pt] (C) {Realization dynamics~\eqref{eq:Realdynamics}};
    \node[below=40pt of B,rectangle, draw=none, fill=gray1, minimum height=2.5em, text centered, text width=8em, outer sep=2.5pt] (D) {Moment \mbox{dynamics~\eqref{eq:Momentdynamics}}};
    
    \draw[line width=6pt,my double triangle, draw=lightblue](A) -- (B) node [above, midway] (n1) {Galerkin projection};
    \draw[line width=6pt,my double triangle, draw=none](A) -- (B) node [below, midway] (n2) {series expansion};
    \draw[line width=6pt,my triangle, draw=lightblue](A) -- (C) node [midway,sloped, above, rotate=180] (n3) {sampling};
    \draw[line width=6pt,my triangle, draw=lightblue](C) -- (D) node [midway, below] (n4) {estimation};
    \draw[line width=6pt,my triangle, draw=lightblue](B) -- (D) node [midway, right] (n5) {\eqref{eq:moments_PCE}};
    \draw[line width=6pt,my triangle, draw=lightblue](A.south east) -- (D.north west) node [midway, sloped, below] (n6) {moment propagation};
\end{tikzpicture}
    \caption{Model-based representations of stochastic systems used in this paper and their relation to each-other.}
    \label{fig:modelOverview}
\end{figure}

\subsection{Exactness and Polynomial Chaos Expansions}
\label{sec:PCE}

From a numerical and from a conceptual point of view, it is desirable if series expansions are of finite order, i.e., that the random variables are represented by finitely many basis functions.
\begin{definition}[Exact series expansions] \label{def:exactPCE}
	We say a function $Z\in L^2(\Omega,\mathbb R^d)$ admits an \emph{exact series expansion of order $p\in\mathbb Z^+$} if the expansion coefficients $\mathsf z^i=\mean[\phi^i Z]/\mean[\phi^i \phi^i]$ vanish for all $i\geq p$, i.e.
	\begin{equation*}
		Z = \sum_{i=0}^{p-1} \phi^i \mathsf z^i.
	\end{equation*}
\end{definition}
\noindent Observe that in the above definition, the series expansion stops at a finite order. Put differently, for $i >p-1$, the coefficients satisfy~$\mathsf{z}^i = 0$ while in the formal statement~\eqref{expansion} we have $i \in \mathbb N$. Hence, one may ask if such finite series representations can be attained. 

A popular approach for uncertainty quantification and uncertainty propagation are Polynomial Chaos Expansions (PCE), which date back to \cite{Wiener38}. Given a family $\mathfrak F$ of \emph{basic} random variables the pivotal idea of PCE is to choose a specific orthogonal system $(\phi^i)_{i\in\mathbb N}$ consisting of polynomials with indeterminates in $\mathfrak F$. 
 
Consider, e.g., a finite or countable family~$\mathfrak F$ of independent normally distributed random variables. Each random variable in $\mathfrak F$ possesses moments of all orders. For a finite selection of random variables of $\mathfrak F$ mixed moments are, due to stochastic independence, simply the product of the individual moments. Therefore, for $n\in\mathbb N$, the set
\begin{equation*}
	\Pi_n =\left\{ \pi  (\xi_1,\dots, \xi_m)\,\middle|\, \begin{gathered} 
	    \text{$\pi$ is a $m$-variate real polynomial}\\
	    \text{ with degree at most $n$},\\  
	    m \in \mathbb N,\ \xi_1,\dots,\xi_m\in\mathfrak F
    \end{gathered} \right\}
\end{equation*}
is a linear subspace of $\spl$. The space $\Pi_0$ consists of almost surely constant random variables, and all elements of $\Pi_1$ are normally distributed. For $n>1$ the space $\Pi_n$ contains also non-Gaussian random variables. Using the Gram--Schmidt process one can construct an orthogonal basis $(\phi^i)_{i\in\mathbb N}$ of $\bigcup_{n=0}^\infty \Pi_n$. In the simplest nontrivial case where $\mathfrak F$ consists of one standard normally distributed random variable $\xi$ this can be given in terms of Hermite polynomials, that is $\phi^i = H_i(\xi)$, where $H_i$ is the $i$th Hermite polynomial. Moreover, the theorem by \cite{cameron1947orthogonal} states that
\begin{equation}
	\label{pce}
	L^2((\Omega,\sigma(\mathfrak F), \prob), \mathbb R) = \closure\left(\bigcup_{n=0}^\infty \Pi_n\right) = \closure\left(\operatorname{span}{(\phi^i)}_{i\in\mathbb N}\right),
\end{equation}
where $\sigma(\mathfrak F)$ is the $\sigma$-algebra generated by the family $\mathfrak F$ and $\closure$ denotes the topological closure. As a consequence every random variable in $L^2((\Omega,\sigma(\mathfrak F), \prob),\mathbb R)$ admits an $L^2$-convergent series expansion in terms of the orthogonal polynomial basis $( \phi^i )_{i\in\mathbb N}$. This series is referred to as Wiener--Hermite polynomial chaos expansion.  We note that in general only $\sigma(\mathfrak F)\subset \mathcal F$ holds, that is, 
\[
L^2((\Omega,\sigma(\mathfrak F), \prob), \mathbb R)\subset L^2((\Omega,\mathcal F, \prob), \mathbb R).
\]
In particular, the orthogonal system $(\phi^i)_{i\in \mathbb N}$ does not span the Hilbert space $L^2((\Omega, \mathcal F,\prob),\mathbb R)$, but rather $L^2((\Omega, \sigma(\mathfrak F),\prob),\mathbb R)$.
Furthermore, every random variable in $\Pi_n$ admits an exact series expansion. Note that the orthogonal basis $(\phi^i)_{i\in\N}$ of $L^2((\Omega, \sigma(\mathfrak F), \prob),\mathbb R)$ can be always chosen such that $\phi^0 = 1$ and, hence due to orthogonality, $\mean[\phi^i] = \mean[\phi^i \phi^0]=0$ for all $i>0$.

Under certain conditions regarding the underlying distributions also a family $\mathfrak F$ of non-Gaussian basic random variables allows the construction of an orthogonal basis $(\phi^i)_{i\in\mathbb N}$ consisting of polynomials in $\mathfrak F$ such that \eqref{pce} holds. Such a polynomial basis, which we refer to as \textit{PCE basis}, allows to expand each function in $L^2((\Omega,\sigma(\mathfrak F), \prob),\mathbb R)$ into a convergent series. For more details on this so-called generalized polynomial chaos expansion, in particular for its convergence properties, we refer the reader to \cite{ErnstMuglerStarkloffUllmann12,Sullivan15,xiu2010numerical}. For details on truncation errors and error propagation we refer to \cite{Field04} and to \cite{Muehlpfordt18}. Fortunately, random variables that follow some widely used distributions admit exact finite-dimensional PCEs with only two terms in suitable polynomial bases. For Gaussian random variables Hermite polynomials are preferable. For other distributions, we refer to Table~\ref{tab:askey_scheme} for the usual basis choices that allow exact PCEs \citep{Koekoek96, Xiu02}.
Observe that in order to use the PCE framework, one needs to specify the basis functions \textit{and} their random-variable arguments. Hence Table~\ref{tab:askey_scheme} also lists those arguments.
\begin{table}[t] 
\caption{Correspondence of random variables, underlying orthogonal polynomials, and their arguments.}
\label{tab:askey_scheme}
\centering
\scalebox{0.85}{
\begin{tabular}{ccc c}
	\toprule
	Distribution  & Support & Orthogonal basis & Argument\\
	\midrule
	Gaussian & $(-\infty, \infty)$ & Hermite &$ \xi \sim \mathcal N(0,1) $ \\
	Uniform  & $[a,b]$ & Legendre & $\xi \sim \mathcal U ([-1,1])$\\
	Beta & $[a,b]$ & Jacobi &$ \xi \sim \mathcal{B}(\alpha,\beta,[-1,1])$\\
	Gamma & $(0,\infty)$ & Laguerre & $ \xi \sim \Gamma(\alpha,\beta,(0,\infty))$ \\
	\bottomrule
\end{tabular}
	\vspace*{-3mm}
}
\end{table}

Before concluding this part we give an example to show how exactness preserving bases for polynomial mappings can be constructed. We also comment on the relation of PCE and reproducing kernel Hilbert spaces.  
\begin{example}[PCE basis construction]
    For the sake of illustration, consider a scalar-valued map $f:L^2(\Omega, \mathbb{R}) \to L^2(\Omega, \mathbb{R})$
    \[
        Y = f(X) = X^2.
    \]
    Let $X$ be Gaussian. Then, it admits the exact PCE
    \[
        X = \mathsf x^0 \phi^0 + \mathsf x^1 \phi^1,
    \]
    where $\phi^0 =1$ and $\phi^1 = \xi$, $\xi \sim \mathcal N (0,1)$. 
    Consider $X^2$ which in terms of  the PCE of $X$ reads
    \[
    X^2 = \left( \mathsf x^0 \phi^0\right)^2 + 2 \mathsf x^0 \mathsf x^1 \phi^0\phi^1 + \left(\mathsf x^1 \phi^1\right)^2.
    \]
    Apparently, the last term cannot be expressed in the subspace spanned by $\phi^0, \phi^1$. 
    Hence, we extend the bases for $Y$ by $\phi^2 = \xi^2-1$ which is the third Hermite orthogonal polynomial. The exact PCE of $Y$ reads
    \[     Y =\left( \left(\mathsf{x}^0\right)^2 +\left(\mathsf{x}^1\right)^2 \right)\phi^0 + 2 \mathsf{x}^0 \mathsf{x}^1 \phi^1 + \left(\mathsf{x}^1 \right)^2\phi^2.
    \]
    Further details on how the structure of explicit polynomial maps can be exploited to construct exactness preserving bases are given by \cite{Muehlpfordt18}.
\end{example}

\begin{example}[Structural uniformity of the PCE]\label{ex:PCE_uniformity}
    Consider a sequence $(\xi_i)_{i\in\mathbb N}$ of i.i.d.\ standard normally distributed random variables. Let $X_k = \begin{bmatrix}
        \xi_{2k} & \xi_{2k+1}
     \end{bmatrix}^\top$, that is the $X_k$ are i.i.d.\ as well. Given $k \in \mathbb N$, for the PCE $X_k = \sum_{i\in\mathbb N} \mathsf x_k^i \xi_i$ we find $\mathsf x_k^i = \begin{bmatrix}
      1 & 0
     \end{bmatrix}^\top$ for $i=2k$, $\mathsf x_k^i = \begin{bmatrix}
      0 & 1
     \end{bmatrix}^\top$ for $i=2k+1$ and $\mathsf x_k^i = 0$ else. One observes that despite some shift the PCE coefficient sequences $(\mathsf x_k^i)_{i\in\mathbb N}$ and $(\mathsf x_j^i)_{i\in\mathbb N}$ corresponding to $X_k$ and $X_j$, respectively, share the same structure. 
\end{example}

\begin{remark}[Link of PCE and RKHS]
	Besides the fundamental lemma, data-driven control using Reproducing Kernel Hil\-bert Spaces (RKHS) is also of tremendous interest, see, e.g., \cite{yan2018gaussian,Zhang20}. A well-known example of an RKHS are Gaussian Processes which are often used as function approximators in machine learning~\citep{Rasmussen, Muandet17}. Hence, it is natural to ask for the link between RKHS and the $L^2$-probability spaces considered here. In general, an $L^2$-space is a Hilbert space of equivalence classes of functions but not a Hilbert space of functions. Thus, without further refinement an $L^2$-probability space is not an RKHS. 
	
	A straightforward observation, which points towards this issue, is that if two random variables have the same PCE in a given basis this does not mean they are point-wise identical. Indeed, they could differ on sets of measure zero,	e.g., in terms of countable many outcomes~$\omega$ with probability zero. However, using the concept of exact PCEs, one may close the gap, i.e.,  random variables in exact and finite PCE bases are connected to learning methods in RKHS. Consider the case where the family $\mathfrak F$ consists of one standard normally distributed random variable.  Then, the underlying Hermite polynomials $\{H_i\}_{i\in\mathbb N}$ give rise to an RKHS  induced by a \emph{Mehler kernel}
	\begin{equation*}
		k_\alpha(t_1,t_2) \doteq\sum_{n\in\mathbb N} \alpha^n\frac{H_n(t_1)H_n(t_2)}{n!} = \frac{\operatorname{exp}\left(t_2^2 - \frac{(t_2-2t_1\alpha)^2}{1-4\alpha^2}\right)}{\sqrt{1-4\alpha^2}}
	\end{equation*}
	for $\alpha\in(0,1)$, cf. \citep[p.196]{Rainville60}.
\end{remark}

\subsection{Moments, Densities, and Conditional Probabilities in PCE}

Almost surely, the careful reader has  recognized that  the overview of Figure~\ref{fig:modelOverview} does not mention the modeling of stochastic dynamics in terms of probability densities which is, e.g., common in case of Markov chains on measurable spaces, i.e., whenever the densities are treated as state variables. In a purely Gaussian and linear setting the moment description~\eqref{eq:Momentdynamics} allows capturing the density evolution, while in general non-Gaussian and nonlinear cases, it is difficult to analytically capture the evolution of densities. Hence, we now briefly discuss how moments and densities can be accessed in the PCE framework. 
\vspace*{1mm}

\noindent\textbf{Moments and PCEs}.
Consider $X$,~$Y\in L^2((\Omega,\mathcal F,\prob),\mathbb R^d)$ with exact PCEs of order $p$, cf. Definition~\ref{def:exactPCE}, with respect to the orthogonal system $(\phi^i)_{i\in\N}$, where $\phi^0 =1$. The expected value and the covariance  of $X$ and $Y$ can be obtained from the PCE coefficients as 
\begin{equation} \label{eq:moments_PCE}
    \mean [X] = \mathsf{x}^0, \quad  \covar[X,Y] = \sum_{i=1}^{p-1} \mathsf{x}^i \mathsf{y}^{i\top}{\mean[\phi^i\phi^i]}.
\end{equation}
For further insights into the computation of higher-order moments from PCEs we refer to \cite{Lefebvre20}.\vspace*{1mm}

\noindent\textbf{Probability Densities and PCE}.
For the sake of illustration, consider $X \in L^2((\Omega,\mathcal F,\prob), \mathbb R)$. Let $Y = f(X)$ be the image of the non-Gaussian random variable $X$ with the PCE
    \[
        X= \sum_{i=0}^{p_x-1} \mathsf x^i \phi^i,\quad \phi^i= H_i(\xi)
    \]
in the Hermite polynomial basis with standard normally distributed argument, i.e.\ $\xi\in\mathcal N(0,1)$. A straight-forward sampling-based strategy to approximate the density of $Y=f(X)$ over $\Omega = \mathbb R$ is given by
\[
    Y(\omega) = f\left(\sum_{i=0}^{p_x-1} \mathsf x^i \phi^i({\omega})\right) = f\left(\sum_{i=0}^{p_x-1} \mathsf x^i H_i({\xi(\omega)})\right).
\]
That is, one samples realizations $\xi(\omega)$, $\omega \in\mathbb R$, of the argument and, then, evaluates~$f$ to obtain samples of the realizations of the non-Gaussian random variable $Y(\omega)$. Without further elaboration we remark that structural knowledge on $X$ or $f$ can and should be exploited. Moreover, we observe a very helpful property of PCEs, i.e., the series expansion separates  the real-valued and deterministic series coefficients $\mathsf x^i$ from the basis functions~$\phi^i$ which capture the stochasticity.
\vspace*{1mm}

\noindent \textbf{Conditional Probabilities and PCE}.
Another aspect that deserves clarification in the PCE framework is how to capture conditional probabilities, which are of major importance in many systems and control contexts.
To this end, consider $Y = X + W$, where $X \in L^2((\Omega,\mathcal F,\prob), \mathbb R)$ and $W \in  L^2((\Omega,\mathcal F,\prob), \mathbb R)$ are independent. Let the PCEs of $X, W$ be given as
\[
    X = \mathsf x^0 \phi^0_x + \mathsf x^1 \phi^1_x \quad \text{and} \quad W = \mathsf w^0 \phi^0_w + \mathsf w^1 \phi^1_w,
\]
where the subscripts $\cdot_x$, $\cdot_w$ in $\phi^i_x$ and $\phi^i_w$, $i =0,1$ emphasize that the arguments of the PCE bases are independent random variables as    $X$ and $W$ are independent.

We are interested in the conditional probability 
\[
    \prob[Y \in \mathbb{Y} \,|\, W= w]
\]
for some given set $\mathbb Y \subset \mathbb R$ and given $w \in \mathbb R$. 
Suppose that $W(\tilde \omega) = w$, i.e., $\tilde \omega \in \Omega$ is the unique outcome (sample) which realizes the event $ W= w$.
Then,  we have
\begin{align*}
\prob\left[Y \in \mathbb{Y} \,|\, W = w\right] &= 
\prob\left[Y  \in \mathbb{Y} 
\,|\, 
 \mathsf w^0 \phi^0_w + \mathsf w^1 \phi^1_w = w\right] \\
&=\prob\left[\left(X+ \mathsf w^0 \phi^0_w(\tilde \omega) + \mathsf w^1 \phi^1_w(\tilde \omega)\right)  \in \mathbb{Y}\right]
 \\
&=\prob\left[\left(\mathsf x^0 \phi^0_x + \mathsf x^1 \phi^1_x + w\right) \in \mathbb{Y}\right] .
\end{align*}
Observe that from the first to the second line, we replace the conditional probability with an unconditional one in which the PCE for $W$ is evaluated to meet the condition. 
Put differently, in the PCE framework conditional probabilities can be captured by evaluating the bases functions corresponding to the condition at the considered outcomes---in the example the basis $\phi^i_w$ is evaluated at $\tilde \omega$---while still regarding the arguments $\xi_x$ of the other bases polynomials as random variables.
\vspace*{2mm}

Concluding this excursion, we note that combining the conceptual ideas expressed in the two examples above illustrates how conditional densities are captured in the PCE framework. Yet, in the general case they might not be directly and analytically accessible, while numerical approximation is immediate. Importantly, provided the PCE propagation is exact in the sense of Definition~\ref{def:exactPCE}, no information about unconditional and conditional densities is lost. 
We conclude this part by pointing the interested reader to the introductory text by~\cite{o2013polynomial} which provides insights into the relation of PCE to other techniques such as Karhunen-Loève expansions or Gram-Charlier series. 
\section{Behaviors of Stochastic Linear Systems} \label{sec:StochBehav}

The previous section has recalled representations of stochastic systems in conceptually different settings, cf.\ Figure~\ref{fig:modelOverview}. Next we establish the connections between these settings and the behavioral framework briefly touched upon in Section~\ref{sec:recapI}.
Doing so, we observe that
\textit{deterministic systems emerge as special cases of stochastic systems, as they should}~\citep{Willems12}.

To this end, an intrinsically algebraic approach would aim at  stochastic kernel representations such as the one discussed in \eqref{eq:detKernelBehav} in Section \ref{sec:ext}. Extending the earlier discussion of \cite{Willems12}, \cite{Baggio17} pursued exactly this route for stochastic LTI systems regarding the stochastic effects as an affine inhomogeneity of a suitable kernel representation.\footnote{Moreover, \cite{Willems12} and \cite{Baggio17} explore in-depth how underlying behavioral relations affect the construction of appropriate---so-called coarse---sigma-algebras and how one can formalize interconnections of open systems accordingly. We do not detail these aspects and instead refer to the mentioned original references} In contrast, we do not rely on kernel representations. Instead we define the behavior of the series expansion representation and of the original stochastic system as appropriate high-dimensional lifts of the underlying deterministic (realization) behavior.

\subsection{Behavioral Representations}
Henceforth and with slight abuse of notation, $\mathfrak B_\infty$, $\mathfrak B^\text{i/o}_\infty$, $\mathfrak B_T$, $\mathfrak B^\text{i/o}_T$ refer to the full and the manifest, respectively, the finite-horizon and the manifest finite-horizon behaviors of~\eqref{eq:Realdynamics}. In short, we say that the realization dynamics~\eqref{eq:Realdynamics} generate the \textit{realization behavior} $\mathfrak B_\infty$, whereby in view of \eqref{eq:inputsplit} the variable $v$ combines controls and disturbances. \vspace{2mm}

\noindent \textbf{The Behavior of the Expansion Coefficients}. %
One approach to a behavioral description of the stochastic system~\eqref{eq:RVdynamics} is via the coefficients of the series expansion of the random variables in terms of a fixed scalar-valued orthogonal basis $( \phi^i )_{i\in\mathbb N}$. We emphasize, that at this point $(\phi^i)_{i\in\mathbb N}$ can conceptually be any orthogonal basis of the space $L^2(\Omega,\mathbb R) = L^2((\Omega,\mathcal F,\prob),\mathbb R)$. Formally, it does not need to be associated with a PCE as described in Subsection~\ref{sec:PCE}.

We first discuss the behavior of expansion coefficients induced by the dynamics \eqref{eq:Coeffdynamics}. To this end, consider the expansion coefficients
\[
    \mathsf z_k = (\mathsf z_k^0, \mathsf z_k^1,\mathsf z_k^2,\dots) =  (\mathsf z^i_k)_{i\in\mathbb N}
    , \quad \mathsf z_k \in \{ \mathsf x_k, \mathsf v_k, \mathsf y_k\},
\]
where we collect for a given time instant $k$ the coefficients corresponding to the basis vectors $\phi^i$ in a sequence. Note that $\mathsf z = (\mathsf z_0,\mathsf z_1,\mathsf z_2,\dots) = (\mathsf z_k)_{k\in\mathbb N}$ is a sequence with respect to time, whereby each element is a sequence. A natural way to think of $\mathsf z$ is in a scheme where its elements are arranged in a infinite-dimensional matrix with one axis representing time and the other representing the series expansion index,
\begin{equation*}
    \mathsf z = \bigl((\mathsf z_k^i)_{i\in\mathbb N}\bigr)_{k\in\mathbb N}\simeq \hspace{-3.5ex}\parbox[c]{27ex}{\hspace*{3ex}\begin{tikzpicture}[my triangle/.style={-{Triangle[width=\the\dimexpr1.8\pgflinewidth,length=\the\dimexpr0.8\pgflinewidth]}}]
    \node (M) {$\displaystyle\begin{bmatrix}
        \mathsf z_0^0 & \mathsf z_0^1 & \dots & \mathsf z_0^i & \dots\\[0.5em]
        \mathsf z_1^0 & \mathsf z_1^1 & \dots & \mathsf z_1^i& \dots\\
        \vdots & \vdots & \ddots &\vdots & \\
        \mathsf z_k^0 & \mathsf z_k^1 & \dots & \mathsf z_k^i& \dots\\
        \vdots & \vdots & & \vdots & \ddots
    \end{bmatrix}$};
    \draw[line width=6pt,my triangle, draw=lightblue] ([yshift=-10pt]M.north west) -- ([yshift=10pt]M.south west) node[midway, rotate=90] {\scriptsize time};
    \draw[line width=6pt,my triangle, draw=lightblue] ([xshift=10pt]M.north west) -- ([xshift=-10pt]M.north east) node[midway] {\scriptsize expansion index};
\end{tikzpicture}}.
\end{equation*}
We define the full and the manifest \emph{behavior of the dynamics of the expansion coefficients}
by
\begin{subequations} \label{eq:CoeffBehav}
\begin{align}
    \mathfrak C_\infty &\doteq\left\{(\mathsf x,\mathsf v,\mathsf y)\,\middle |\,\begin{gathered} 
    \mathsf x\in (\ell^2(\mathbb R^{n_x}))^{\mathbb N}, \mathsf v\in (\ell^2(\mathbb R^{n_v}))^{\mathbb N},\\ \mathsf y\in (\ell^2(\mathbb R^{n_y}))^{\mathbb N} \text{ with } \\
         (\mathsf x^i,\mathsf v^i, \mathsf y^i)\in\mathfrak B_\infty\text{ for all } i\in\mathbb N \end{gathered}\right\},
         \end{align}
         \begin{align}
    \mathfrak C_\infty^\text{i/o} &\doteq\left\{ (\mathsf v,\mathsf y) \,\middle|\, (\mathsf x,\mathsf v,\mathsf y)\in\mathfrak C_\infty\text{ for some }\mathsf x\in \ell^2(\mathbb R^{n_y})^{\mathbb N}\right\}.
\end{align}
\end{subequations}
In short, and whenever no confusion can arise, we refer to $\mathfrak C_\infty$ and $\mathfrak C_\infty^\text{i/o}$ as the \textit{expansion coefficient behavior}, respectively, as the \textit{manifest expansion coefficient behavior}.
Observe that $\mathbb N$ in $(\ell^2(\mathbb R^{n_z}))^{\mathbb N}$, $z \in \{x,v,y\}$, refers to the considered time horizon, while  $\mathbb N$ in $ (\mathsf x^i,\mathsf v^i, \mathsf y^i)$, $i\in \mathbb N$ refers to the expansion order (which in case of finite-dimensional expansions is denoted as $p$). 
In the case of a PCE basis $(\phi^i)_{i\in\N}$ (cf.\ Subsection~\ref{sec:PCE}), placing emphasis in its origin we refer to the expansion coefficient behavior as \textit{PCE coefficient behavior}.
The finite-horizon behaviors $\mathfrak C_T$ and $\mathfrak C_T^\text{i/o}$ are defined similarly as in the deterministic case, i.e. by truncating $(\mathsf x,\mathsf v,\mathsf y) \in \mathfrak C_\infty$ to the bounded time horizon $[0, T]$, cf.\ \eqref{finHorbehav} and~\eqref{finHorbehavIO}.

The expansion coefficient behavior $\mathfrak C_\infty$ inherits  desirable properties from the realization behavior $\mathfrak B_\infty$.
\begin{lemma}[Completeness of $\mathfrak C_\infty$]
    \label{stochbehav_complete}
    The expansion coefficient behavior $\mathfrak C_\infty$ is complete, i.e.\ $\mathsf c|_{[0,T]}\in \mathfrak C_T$ for all $T\in\mathbb N$ implies $\mathsf c\in \mathfrak C_\infty$.
\end{lemma}
\begin{proof}
    For every $T$ there is $\mathsf c^T\in\mathfrak C_\infty$ such that $\mathsf c|_{[0,T]} = \mathsf c^T|_{[0,T]}$. Therefore, $\mathsf c^i|_{[0,T]} = \mathsf c^{T,i}|_{[0,T]}\in\mathfrak B_{[0,T]}$ for all $T$, $i\in\mathbb N$. By completeness of $\mathfrak B_\infty$, we obtain $\mathsf c^i\in\mathfrak B_\infty$ for all $i\in\mathbb N$, that is $\mathsf c\in\mathfrak C_\infty$.
\end{proof}
\noindent The next lemma deals with the controllability of the expansion coefficient behavior. It results as a (straightforward) corollary of Lemma~\ref{ctrlRbehav} formulated and proven for descriptor systems  in Section~\ref{sec:ext}; hence we skip the proof.
\begin{lemma}[Controllability of $\mathfrak C_\infty$]
    \label{stochbehav_ctrl}
    Suppose that the realization behavior $\mathfrak B_\infty$ is controllable (with transition delay $T'=n_x$). Then $\mathfrak C_\infty$ is controllable with transition delay $T'=n_x$.
\end{lemma}

\noindent \textbf{The Behavior of $L^2$-Random Variables}. 
A reasonable requirement for a behavior of the stochastic system~\eqref{eq:RVdynamics} is that its trajectories considered path-wise, that is evolving in time for a fixed outcome $\omega\in\Omega$, satisfy the realization dynamics~\eqref{eq:Realdynamics}. To this end, we define the full and the manifest \emph{behavior of the dynamics in $L^2$-random variables} corresponding to  system~\eqref{eq:Realdynamics} as
\begin{subequations} \label{eq:StochBehav}
\begin{align}
    \mathfrak S_\infty &\doteq \left\{(X,V,Y)\,\middle |\, \begin{gathered}X\in (L^2(\Omega,\mathbb R^{n_x}))^\mathbb{N},\\ V\in (L^2(\Omega,\mathbb R^{n_v}))^\mathbb{N},\\ Y\in (L^2(\Omega,\mathbb R^{n_y}))^\mathbb{N} \text{ with} \\ (X(\omega), V(\omega), Y(\omega))\in\mathfrak B_\infty\\\text{ for $\prob$-a.a.\ }\omega\in\Omega\end{gathered}\right\},\\
    \mathfrak S^\text{i/o}_\infty &\doteq \left\{(V,Y)\,\middle |\,(X,V,Y)\in\mathfrak S_\infty\text{ with } X\in (L^2(\Omega,\mathbb R^{n_x}))^{\mathbb N}\right\}.
\end{align}
\end{subequations}
The corresponding finite horizon behaviors $\mathfrak S_T$ and $\mathfrak S_T^\text{i/o}$ are defined similarly to the deterministic case, see \eqref{finHorbehav} and \eqref{finHorbehavIO}. For the sake of brevity, we refer to  $\mathfrak S_\infty$ and $\mathfrak S_\infty^\text{i/o}$ as the \textit{random variable behavior}, respectively, as the \textit{manifest random variable behavior}. Note that in the random variable behavior $\mathfrak S_\infty$ besides existence of second-order moments no further assumptions on the statistical nature of the considered stochastic processes are made. In particular, neither stochastic independence nor adaptation to a certain filtration is required.

\subsection{Behavior Inclusion,  Lift, and Equivalence}
Notice that the realization trajectories in $\mathfrak B_\infty$ can be considered also as stochastic processes which are at each time instant almost surely constant. Hence, the random variable behavior as defined in \eqref{eq:StochBehav} includes the realization behavior,
\begin{equation*}
    \mathfrak B_\infty\subset \mathfrak S_\infty, 
\end{equation*}
which can also be extended to finite-time or manifest behaviors. Likewise, the expansion coefficient behavior~\eqref{eq:CoeffBehav} satisfies
\begin{equation} \label{eq:CoeffBehavContruct}
    \mathfrak C_\infty {\subset}\bigtimes_{i\in\mathbb N} \mathfrak B_\infty.
\end{equation}
The relation \eqref{eq:CoeffBehavContruct} holds with equality if a square summability condition on the elements of $\bigtimes_{i\in\mathbb N} \mathfrak B_\infty$ along the expansion index $i$ is imposed. In the case of finite time horizon and exactness of the series expansion this additional condition is not necessary for the equality. The inclusion \eqref{eq:CoeffBehavContruct} expresses the fact that for all dimensions $i\in \mathbb N$ of the expansion (which correspond to basis directions $\phi^i$) the realization behavior $\mathfrak B_\infty$ describes the dynamics. Put differently, in the context of model-based system representations, for all basis dimensions $i \in \mathbb N$, the expansion coefficients satisfy identical linear system equations~\eqref{eq:Coeffdynamics} which in turn correspond to the realization dynamics~\eqref{eq:Realdynamics}.
Consequently, it is fair to ask for the relation between $ \mathfrak S_\infty$ and $ \mathfrak C_\infty$. \vspace{2mm}

\noindent \textbf{Equivalence and Lift}. %
In the above definition of the random variable behavior the realm of finite-dimensional systems is left behind. Below we derive relationships which sustain our conception of random variable behavior. 
The next theorem shows how elements of the behaviors can be mapped onto each other.

\begin{theorem}[Behavioral lift]\label{thm:lift}~\\
    \begin{enumerate}\vspace*{-6mm}
    \item The linear map $\Phi:\mathfrak C_\infty\rightarrow \mathfrak S_\infty$
    \[ 
    (\mathsf x, \mathsf v,\mathsf y)\mapsto\Phi(\mathsf x, \mathsf v,\mathsf y)\doteq \left(\sum_{i\in\mathbb N} \phi^i\mathsf x^i, \sum_{i\in\mathbb N} \phi^i\mathsf v^i, \sum_{i\in\mathbb N} \phi^i\mathsf y^i\right)
    \] 
    is bijective. Its inverse is given by $\Phi^{-1}(X,V,Y)=(\mathsf x, \mathsf v,\mathsf y)$, where
    \begin{equation}
    \label{eq:coeffdef}
        \mathsf x^i = \frac{\mean[\phi^i X]}{\mean[\phi^i\phi^i]},\quad  \mathsf v^i = \frac{\mean[\phi^i V]}{\mean[\phi^i\phi^i]},\quad  \mathsf y^i = \frac{\mean[\phi^i Y]}{\mean[\phi^i\phi^i]},\quad i\in \mathbb N.
    \end{equation}
    \item For fixed $\omega\in\Omega$, the linear map $\Psi_\omega:\mathfrak S_\infty \rightarrow \mathfrak B_\infty$, 
    \[(X,V,Y)\mapsto \Psi_\omega(X,V,Y)\doteq (X(\omega), V(\omega), Y(\omega))\]
    is surjective. Further, the concatenation $\Psi_\omega\circ \Phi$ satisfies
    \begin{equation*}
        (\Psi_\omega\circ \Phi) (\mathsf x,\mathsf v,\mathsf y) = \left(\sum_{i\in\mathbb N} \phi^i(\omega)\mathsf x^i, \sum_{i\in\mathbb N} \phi^i(\omega)\mathsf v^i, \sum_{i\in\mathbb N} \phi^i(\omega)\mathsf y^i\right)
    \end{equation*}
    and $\Psi_\omega\circ \Phi$ is surjective.
\end{enumerate}
\end{theorem}
\begin{proof}
    Assertion~(i) is a consequence of the isometric isomorphism between $\ell^2(\mathbb R)$ and $L^2((\Omega,\mathcal F,\prob),\mathbb R)$ established by series expansion in terms of the orthogonal basis $(\phi^i)_{i\in\mathbb N}$, cf.\ the Fischer--Riesz theorem. We show that the image of $\Phi$ under $\mathfrak C_\infty$ is contained in $\mathfrak S_\infty$. Let $(\mathsf x, \mathsf v,\mathsf v)\in\mathfrak C_\infty$ and set $(X,V,Y)=\Phi(\mathsf x, \mathsf v,\mathsf v)$. By definition $(\mathsf x^i, \mathsf v^i,\mathsf y^i)\in\mathfrak B_\infty$ for every $i\in\mathbb N$. Therefore, \eqref{eq:Coeffdynamics} holds for all $i,k\in\mathbb N$. Together with the orthogonality of $(\phi^i)_{i\in\mathbb N}$ we obtain for all $k\in\mathbb N$
    \begin{subequations}
    \label{eq:zeroVariance}
    \begin{align}
        0&= \sum_{i\in\mathbb N} \mean[\phi^i\phi^i](\mathsf x^i_{k+1} - A\mathsf x^i_k - \widetilde B \mathsf v^i_k)^\top (\mathsf x^i_{k+1} - A\mathsf x^i_k - \widetilde B \mathsf v^i_k) \nonumber\\
        &=\mean\bigl[(X_{k+1} - AX_k - \widetilde B V_k)^\top (X_{k+1} - AX_k - \widetilde B V_k)\bigr],\label{eq:zeroVariancea}\\
        0&= \sum_{i\in\mathbb N} \mean[\phi^i\phi^i](\mathsf y^i_{k} - C\mathsf x^i_k - \widetilde D \mathsf v^i_k)^\top (\mathsf y^i_{k} - C\mathsf x^i_k - \widetilde D \mathsf v^i_k), \nonumber\\
        &=\mean\bigl[(Y_{k} - CX_k - \widetilde D V_k)^\top (Y_{k} - CX_k - \widetilde D V_k)\bigr]\label{eq:zeroVarianceb}
    \end{align}
    \end{subequations}
    Hence, \eqref{eq:RVdynamics} holds for all $k\in\mathbb N$ and $(X,V,Y)\in\mathfrak S_\infty$. Next, we show that $\Phi$ is surjective. Let $(X,V,Y)\in\mathfrak S_\infty$ and take the series expansion with respect to $(\phi^i)_{i\in\mathbb N}$, that is $(\mathsf x,\mathsf v,\mathsf y)$ is given by \eqref{eq:coeffdef}. As $(X,V,Y)\in\mathfrak S_\infty$, \eqref{eq:RVdynamics} holds for all $k\in\mathbb N$. Therefore, \eqref{eq:zeroVariance} holds, which implies \eqref{eq:Coeffdynamics} for all $i,k\in\mathbb N$. Hence, $(\mathsf x,\mathsf v,\mathsf y)\in\mathfrak C_\infty$ and $\Phi(X,V,Y)=(\mathsf x,\mathsf v,\mathsf y)$. The injectivity of $\Phi$ follows from the uniqueness of the series expansion. This shows (i).
    
    The surjectivity of $\Psi_\omega\circ \Phi$ in (ii) follows from $\mathfrak B_\infty\subset \mathfrak S_\infty$. 
\end{proof}
\noindent The relations between the behaviors and the maps derived in Theorem~\ref{thm:lift} are sketched in Figure~\ref{fig:BehavOverview}. Observe the structural similarity of the relations between the behaviors in Figure~\ref{fig:BehavOverview} and the relation between the models in Figure~\ref{fig:modelOverview}.
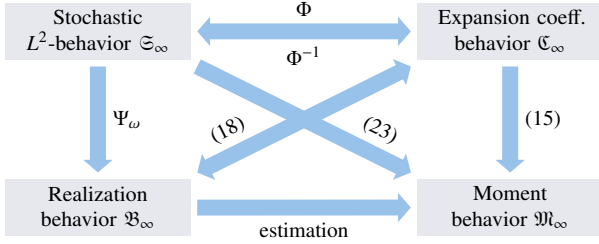
\begin{figure}[t]
    \centering
        \begin{tikzpicture}[my triangle/.style={-{Triangle[width=\the\dimexpr1.8\pgflinewidth,length=\the\dimexpr0.8\pgflinewidth]}}, my double triangle/.style={{Triangle[width=\the\dimexpr1.8\pgflinewidth,length=\the\dimexpr0.8\pgflinewidth]}-{Triangle[width=\the\dimexpr1.8\pgflinewidth,length=\the\dimexpr0.8\pgflinewidth]}}]
    \footnotesize
    \node[rectangle, draw=none, fill=gray1, minimum height=2.5em, text centered, text width=8em, outer sep=2.5pt] (A) {Stochastic $L^2$-behavior $\mathfrak S_\infty$};
    \node[right=80pt of A, rectangle, draw=none, fill=gray1, minimum height=2.5em, text centered, text width=8em, outer sep=2.5pt] (B) {Expansion coeff.\ behavior $\mathfrak C_\infty$};
    \node[below= 40pt of A,rectangle, draw=none, fill=gray1, minimum height=2.5em, text centered, text width=8em, outer sep=2.5pt] (C) {Realization behavior $\mathfrak B_\infty$};
    \node[below=40pt of B,rectangle, draw=none, fill=gray1, minimum height=2.5em, text centered, text width=8em, outer sep=2.5pt] (D) {Moment \\ behavior $\mathfrak M_\infty$};
    
    \draw[line width=6pt,my double triangle, draw=lightblue](A) -- (B) node [above, midway] (n1) {$\Phi$};
    \draw[line width=6pt,my double triangle, draw=none](A) -- (B) node [below, midway] (n2) {$\Phi^{-1}$};
    \draw[line width=6pt,my triangle, draw=lightblue](A) -- (C) node [midway,right] (n3) {$\Psi_\omega$};
    \draw[line width=6pt,my triangle, draw=lightblue](C) -- (D) node [midway, below] (n4) {estimation};
    \draw[line width=6pt,my triangle, draw=lightblue](B) -- (D) node [midway, right] (n5) {\eqref{eq:moments_PCE}};
    \draw[line width=6pt,my triangle, draw=lightblue](A.south east) -- (D.north west) node [midway, sloped, above right=0pt and 15pt] (n6) {\eqref{top}};
    \draw[line width=6pt,my double triangle, draw=lightblue](B.south west) -- (C.north east) node [midway, sloped, above left=0pt and 15pt] (n6) {\eqref{eq:CoeffBehavContruct}};
\end{tikzpicture}
    \caption{Relations and maps between the behaviors $\mathfrak B_\infty, \mathfrak C_\infty$, $\mathfrak S_\infty$, and $\mathfrak M_\infty$.}
    \label{fig:BehavOverview}
\end{figure}
\begin{remark}[Behavioral lift for PCE]\label{rem:liftPCE}
    Given an orthogonal system $(\phi^i)_{i\in\mathbb N}$ derived from a family $\mathfrak F$ of random variables in the context of PCE (cf.\ Subsection~\ref{sec:PCE}) the behavioral lift is subject to subtle limitations. Due to the fact that $\sigma(\mathfrak F)$ is in general only a sub-$\sigma$-algebra of $\mathcal F$, the map $\Phi$ defined in Theorem~\ref{thm:lift}~(i) is still injective, but not surjective. However, in this case
    \[
        \Phi(\mathfrak C_\infty) = \mathfrak S_\infty \cap L^2((\Omega,\sigma(\mathfrak F),\prob), \mathbb R^{n_x+n_v+n_y})^{\mathbb N}.
    \]
    On the other hand, the surjectivity of the map $\Psi_\omega$ given in Theorem~\ref{thm:lift}~(ii) remains unchanged.
    In conclusion, it stands to reason that the potential loss of equivalence of descriptions,  stemming from lifting the PCE coefficient behavior $\mathfrak C_\infty$ by $\Phi$, does not imply significant loss of information in actual applications. 
\end{remark}

\noindent Similar to Theorem~\ref{thm:lift} one finds comparable maps between the manifest behaviors as well as behaviors with respect to finite-time horizons.

As an immediate consequence of the behavioral lift we find the following relationship between the realization behavior and expansion coefficient behavior.

\begin{corollary}[$\mathfrak B_\infty \to \mathfrak C_\infty$]
\label{cor:embbedingCoeff}
    One has $\Phi^{-1}(\mathfrak B_\infty)\subset \mathfrak C_\infty$. For any $(x,v,y)\in\mathfrak B_\infty$ the coefficients of $\Phi^{-1}(x,v,y)=(\mathsf x, \mathsf v,\mathsf y)$ are given by
    \[\mathsf x^i = \frac{\mean[\phi^i]}{\mean[\phi^i\phi^i]}x,\quad  \mathsf v^i = \frac{\mean[\phi^i ]}{\mean[\phi^i\phi^i]}v,\quad  \mathsf y^i = \frac{\mean[\phi^i]}{\mean[\phi^i\phi^i]}y,\quad i\in \mathbb N.\]
    In particular, if $\phi^0\equiv 1$, that is $\mean [\phi^i]=0$ for $i\geq 1$, we find $\mathsf x^0=x$, $\mathsf v^0=v$, $\mathsf y^0=y$ and $\mathsf x^i=0$, $\mathsf v^i=0$, $\mathsf y^i=0$ for all $i\geq 1$.
\end{corollary}
\noindent By Corollary~\ref{cor:embbedingCoeff} the sequence $\mathsf c\in\ell^2(\mathbb R)$ given by 
\[\mathsf c^i = \dfrac{\mean[\phi^i]}{\mean[\phi^i\phi^i]}
\] allows, in accordance with the identification~\eqref{eq:CoeffBehavContruct}, the embedding of $\mathfrak B_\infty$ into $\mathfrak C_\infty$, that is for each $b\in\mathfrak B_\infty$ one has $\mathsf c b = (\mathsf c^0 b, \mathsf c^1 b,\dots)\in\mathfrak C_\infty$. The next theorem illustrates the equivalence of behaviors. 
\begin{theorem}[Behavioral equivalence]
    \label{thm:equivalence} 
    Let $X\in (L^2(\Omega,\mathbb R^{n_x}))^{\mathbb N}$, $V\in (L^2(\Omega,\mathbb R^{n_v}))^{\mathbb N}$ and $Y\in (L^2(\Omega,\mathbb R^{n_y}))^{\mathbb N}$ with its corresponding expansion coefficients $\mathsf x\in (\ell^2(\mathbb R^{n_x}))^{\mathbb N}$, $\mathsf v\in (\ell^2(\mathbb R^{n_v}))^{\mathbb N}$ and $\mathsf y\in (\ell^2(\mathbb R^{n_y}))^{\mathbb N}$, and for $\omega\in\Omega$ its realizations $ (X(\omega),V(\omega),Y(\omega)) \in \mathbb (\R^{n_x}\times \mathbb \R^{n_v}\times \mathbb \R^{n_y})^\mathbb{N}$.
    
    Then, for $T\in\mathbb N\cup\{\infty\}$, the following statements are equivalent:
    \begin{enumerate}[label=(\roman*)]
        \item $(X(\omega),V(\omega),Y(\omega))\in\mathfrak B_T$ for $\prob$-a.a.\ $\omega\in\Omega$;
        \item $(X,V,Y)\in\mathfrak S_T$;
        \item $(\mathsf x,\mathsf v, \mathsf y)\in\mathfrak C_T$;
        \item $(\mathsf x^i, \mathsf v^i, \mathsf y^i)\in \mathfrak B_T$ for all $i\in\mathbb N$.
    \end{enumerate}
    A similar proposition holds true for the manifest behaviors.
\end{theorem}
\begin{proof}
    We show the equivalence for $T=\infty$. The equivalence between (i) and (ii) as well as between (iii) and (iv) follows by definition. The equivalence of (ii) and (iii) follows with the behavioral lift, Theorem~\ref{thm:lift}~(i).
    
    By restricting the elements of $\mathfrak B_\infty$, $\mathfrak S_\infty$, and $\mathfrak C_\infty$ to a finite time horizon $T$ one obtains the equivalences for $\mathfrak B_T$, $\mathfrak S_T$, and $\mathfrak C_T$.
\end{proof}

The next lemma featuring controllability and completeness of the random variable behavior $\mathfrak S_\infty$ is a direct consequence of Lem\-ma~\ref{stochbehav_complete} in combination with Theorems~\ref{thm:lift} and~\ref{thm:equivalence}. 

\begin{lemma}[Completeness and controllability of $\mathfrak S_\infty$]~\\
    \label{lem:compl_ctrl_stochBehav}
    The random variable behavior $\mathfrak S_\infty$ is complete, i.e.\ $ S|_{[0,T]}\in \mathfrak S_T$ for all $T\in\mathbb N$ implies $ S \in\mathfrak S_\infty$. If the realization behavior $\mathfrak B_\infty$ is controllable  with delay $T'=n_x$, then $\mathfrak S_\infty$ is controllable with delay $T'=n_x$.
\end{lemma}
\begin{proof}
    Let $S=(X,V,Y)$ and suppose that $S|_{[0,T]}\in \mathfrak S_T$ for all $T\in\mathbb N$. Taking the corresponding coefficients of the series expansion $\mathsf c=(\mathsf x,\mathsf v,\mathsf y)$ one has $\mathsf c|_{[0,T]}\in\mathfrak C_T$ by Theorem~\ref{thm:equivalence}. The completeness of $\mathfrak C_\infty$, see Lemma~\ref{stochbehav_complete}, yields $\mathsf c\in \mathfrak C_\infty$ and with Theorem~\ref{thm:equivalence} we see $S\in\mathfrak S_\infty$. The controllability of $\mathfrak S_\infty$ follows in a similar way employing Lemma~\ref{stochbehav_ctrl}.
\end{proof}
\noindent Completeness and controllability of the random variable behavior are both important features in the context of control design and in view of optimal control. Controllability guarantees that given an initial condition there exists an intermediate trajectory steering the system for instance into some equilibrium. Provided a sufficiently long time horizon this implies feasibility of the optimal control problem (provide no other constraints are considered). Completeness on the other hand ensures that successively solving the optimal control problem, while stepping  forward in time, leads to a valid trajectory in the infinite time horizon. In essence, completeness of the behavior gives completeness of the dynamic system as such \citep{Sontag98a}.\vspace{2mm}

\noindent \textbf{Non-Equivalence of Moment Behaviors}. 
The propagation of moments with respect to the system dynamics, i.e.\ \eqref{eq:Momentdynamics}, leads to the following definition of the \emph{behavior associated to the dynamics of the second-order moments} (in short \emph{moment behavior}),

\begin{equation}
    \label{eq:MomentBehavior}
    \mathfrak M_\infty \doteq \mathfrak B_\infty \times \left\{~ c~\middle|\,\begin{gathered}  c=(c^{xx}, c^{xv}, c^{vv}, c^{yy}) \text{ with}\\ 
    c^{pq}\in (\mathbb R^{n_p\times n_q})^{\mathbb N}\text{ for all } p,q\in\{x,v,y\}\\
    \text{ and } c\text{ satisfies for all }k\in\mathbb N\\
    c_{k+1}^{xx}  =
    \begin{bmatrix}
    A  \\ \widetilde B
    \end{bmatrix}^\top
       \begin{bmatrix}
    c^{xx}_k & c^{xv}_k \\
    (c^{xv}_k)^\top & c^{vv}_k  
    \end{bmatrix}
       \begin{bmatrix}
    A  \\ \widetilde B 
    \end{bmatrix}
    \\
    c^{yy}_k =\begin{bmatrix}
    C  \\ \widetilde D  
    \end{bmatrix}^\top
       \begin{bmatrix}
    c^{xx}_k & c^{xv}_k \\
    (c^{xv}_k)^\top & c^{vv}_k  
    \end{bmatrix}
       \begin{bmatrix}
    C  \\ \widetilde D  
    \end{bmatrix}
\end{gathered}\right\}.
\end{equation}
Despite nonlinearity in the definition of higher-order moments, the (second-order) moment behavior $\mathfrak M_\infty$ as such is a linear vector space. However, comparison of Figure~\ref{fig:modelOverview} to Figure~\ref{fig:BehavOverview} raises the question of how the moment behavior fits into the latter? 

Let $(X,V,Y)\in\mathfrak S_\infty$ and set $c^x=\covar[X,X]$, $c^{v}=\covar[V,V]$, $c^{xv}=\covar[X,V]$, etc. From \eqref{eq:Momentdynamics} it is not difficult to see that 
\begin{equation}
\label{down}
    m = (\mean [X] ,\mean [U],\mean [Y], c^{xx}, c^{xv}, c^{vv}, c^{yy})\in\mathfrak M_\infty.
\end{equation}
Consider the nonlinear map 
\begin{equation}
    \label{top}
 \Xi: \mathfrak S_\infty \to \mathfrak M_\infty,\quad S \mapsto  m
\end{equation}
which assigns to a stochastic process $ S \doteq (X,V,Y) \in \mathfrak S_\infty$ its sequence of moments $m$ in \eqref{down}. 
However, even in the case of Gaussian processes which are fully determined by first and second-order moments, the map $\Xi$ is not injective. 

\begin{example}[Non-equivalence of $\mathfrak M_\infty$ and $\mathfrak S_\infty$]
    Let $\mathfrak B_\infty$, $\mathfrak S_\infty$, and $\mathfrak M_\infty$ be the realization, stochastic, and moment behavior, respectively, corresponding to system~\eqref{eq:Realdynamics} with 
    \[
    \sqrt{2} A=\sqrt{2} \widetilde B=C= I_1 \quad \text{and}\quad  \widetilde D=0_{1\times 1}.
    \]
    Consider independently standard normally distributed random variables $X_0,\dots$, $V_0,\dots$ and set $Y_k=X_k$. Define $c^{xx}$, $c^{xv}$, etc.\ as before. Then $c^{xx}=c^{vv}=c^{yy}=(1,1,\dots)$, $\mean[X]=\mean[V]=\mean[Y]=c^{xv}=(0,0,\dots)$ and \eqref{down} holds. On the other hand, for each $k\in\mathbb N$ the random variable 
    \[\Delta_k=X_{k+1}-1/\sqrt{2}(X_k+V_k)\]
    is normally distributed with zero mean. Its variance reads 
    \[\var[\Delta_k] =\var[X_{k+1}]+ \var[1/\sqrt{2}\, X_k] + \var[1/\sqrt{2}\, V_k] = 2.\]
    With $\prob[X_{k+1} = 1/\sqrt{2} (X_k + V_k)]=\prob[ \Delta_k=0]=0$ for $k\in\mathbb N$ we see that $(X(\omega),V(\omega),Y(\omega))\notin\mathfrak B_\infty$ for $\prob$-a.a.\ $\omega\in\Omega$ and, therefore, $(X,V,Y)\notin\mathfrak S_\infty$.
\end{example}
\noindent The above example illustrates that even if the series of moments belongs to the moment behavior, it might happen that with probability one all realizations of the corresponding stochastic process are incompatible with the system dynamics.
This intrinsic difficulty of working with moments and their behaviors can also be seen in Figure~\ref{fig:BehavOverview}: specifically,  a structured means to map an element of the moment behavior $\mathfrak M_\infty$ to the random variable behavior $ \mathfrak S_\infty$ or to the expansion coefficient behavior $ \mathfrak C_\infty$ is not available.
\section{The Stochastic Fundamental Lemma}
\label{sec:StochFundLem}

The behavioral lifts and behavioral equivalence as established in Theorem~\ref{thm:lift} and~\ref{thm:equivalence}, respectively, enable to formulate a version of the fundamental lemma applicable to stochastic systems. As in Section~\ref{sec:StochBehav} we consider LTI systems.

\subsection{Equivalence and Inclusion of Column Spaces}

We begin with a result for LTI systems that has been given by~\cite{Pan21s}.
\begin{lemma}[Column-space equivalence]\label{lem:colEqui} 
    Let there exist a controllable realization behavior~$\mathfrak B_\infty$ based on a realization model given by an  LTI system with state dimension~$n_x$. For $T \in \mathbb{Z}_+$, let $(V,Y) \in \mathfrak S_{T-1}^\text{i/o}$ with corresponding expansion coefficients $(\mathsf v, \mathsf y)\in\mathfrak C_{T-1}^\text{i/o}$ and $(\hat v,\hat y)\in\mathfrak B_{T-1}^\text{i/o}$. Further, assume that $\hat v$ and the coefficients~$\mathsf v^i$, $i \in \mathbb N$, are persistently exciting of order $L+n_x$.
\begin{itemize}
\item[(i)] Then, for all $i \in \mathbb N$
\begin{subequations}
\begin{equation} \label{eq:colEqui}
\text{$\operatorname{colsp}
    \begin{bmatrix}
        \mathcal H_{L}(\mathsf v^i_{[0,T-1]})\\
        \mathcal H_L(\mathsf y^i_{[0,T-1]})
    \end{bmatrix}
    = \operatorname{colsp}
    \begin{bmatrix}
        \mathcal H_{L}(\hat{\mathbf{v}}_{[0,T-1]})\\
        \mathcal H_L(\hat{\mathbf{y}}_{[0,T-1]})
    \end{bmatrix}.$}
\end{equation}
\item[(ii)] Moreover, for all $g \in \R^{T-L+1}$, there exists a function $G\in L^2(\Omega, \mathbb R^{T-L+1})$ such that
\begin{equation}\label{eq:colEquiRV}
    \begin{bmatrix}
        \mathcal H_{L}(\mathbf V_{[0,T-1]})\\
        \mathcal H_L(\mathbf Y_{[0,T-1]})
    \end{bmatrix} g = \begin{bmatrix}
        \mathcal H_{L}(\hat{\mathbf{v}}_{[0,T-1]})\\
        \mathcal H_L(\hat{\mathbf{y}}_{[0,T-1]})
    \end{bmatrix} G. 
\end{equation}
\end{subequations}
\end{itemize}
\end{lemma}

Observe that from the applications point of view, the above lemma is subject to a severe limitation since persistency of excitation is assumed \textit{for all expansion indices} $i \in \mathbb N$ in \eqref{eq:colEqui}. That is, in case of stochastic processes composed by i.i.d.\ random variables, the expansion coefficients have to be constant. Put differently, for i.i.d.\ random variables persistency of excitation of the corresponding expansion coefficients does not hold, cf.\ Example~\ref{ex:PCE_uniformity}. Hence, the assumption of persistency of excitation is, in general, too strong.

The next result recaps an observation made by~\cite{Pan21s}, which relaxes the relation of column-spaces from equivalence to inclusion. Doing so tremendously fosters the applicability since persistency of excitation is not required for any $\mathsf v^i$, where $i \in \mathbb{N}$ corresponds to the PCE dimension.
\begin{corollary}[Column-space inclusion] \label{cor:inclusion} 
    Let there exist a controllable realization behavior~$\mathfrak B_\infty$ based on a realization model given by an LTI system with state dimension~$n_x$. For $T \in \mathbb{Z}_+$, let $(V,Y) \in \mathfrak S_{T-1}^\text{i/o}$ with corresponding expansion coefficients $(\mathsf v, \mathsf y)\in\mathfrak C_{T-1}^\text{i/o}$ and $(\hat v,\hat y)\in\mathfrak B_{T-1}^\text{i/o}$. Assume that $\hat v$ is persistently exciting of order $L+n_x$. Then, for all $i \in \mathbb N$, we have the following inclusion 
    \begin{equation} \label{eq:colInclusion}
    \operatorname{colsp} 
    \begin{bmatrix}
    \mathcal H_{L}(\mathsf v^i_{[0,T-1]})\\
    \mathcal H_L(\mathsf y^i_{[0,T-1]})
\end{bmatrix}
\subseteq \operatorname{colsp}
\begin{bmatrix}
    \mathcal H_{L}(\hat{\mathbf{v}}_{[0,T-1]})\\
    \mathcal H_L(\hat{\mathbf{y}}_{[0,T-1]})
\end{bmatrix}.
\end{equation}
\end{corollary}

Based on the two results stated above, we can now formulate the fundamental lemma for stochastic LTI systems of \cite{Pan21s}. 
\begin{lemma}[Stochastic fundamental lemma] 
    \label{lem:FL_stoch}
    Let there exist a controllable realization behavior~$\mathfrak B_\infty$ based on a realization model given by LTI system with state dimension~$n_x$. Let $(v,y)\in\mathfrak B_{T-1}^\text{i/o}$ be such that $v$ is persistently exciting of order $L+n_x$. Then, the following statements hold:
     \begin{subequations}
    \begin{enumerate}
        \item $(\tilde{\mathsf v},\tilde{\mathsf y})\in\mathfrak C_{L-1}^\text{i/o}$ if and only if there is $\mathsf g\in \ell^2(\mathbb R^{T-L+1})$ such that
        \begin{equation}
            \label{eq:FL_coef}
            \begin{bmatrix}
                \mathsf{\tilde v}^i_{[0,L-1]}\\ \mathsf{\tilde y}^i_{[0,L-1]}
            \end{bmatrix} = \begin{bmatrix}
                \mathcal H_{L} (\mathbf{v}_{[0,T-1]}) \\ \mathcal H_{L} (\mathbf{y}_{[0,T-1]})
            \end{bmatrix}\mathsf g^i
        \end{equation}
        for all $i\in\mathbb N$.
        \item $(\tilde{V},\tilde{Y})\in\mathfrak S_{L-1}^\text{i/o}$ if and only if there is $G\in L^2(\Omega,\mathbb R^{T-L+1})$ such that
        \begin{equation}
            \label{eq:FL_stoch_eq}
            \begin{bmatrix}
                \mathbf{\tilde V}_{[0,L-1]}\\ \mathbf{\tilde Y}_{[0,L-1]}
            \end{bmatrix} = \begin{bmatrix}
                \mathcal H_{L} (\mathbf{v}_{[0,T-1]}) \\ \mathcal H_{L} (\mathbf{y}_{[0,T-1]})
            \end{bmatrix} G.
        \end{equation}
    \end{enumerate}
     \end{subequations}
\end{lemma}

\begin{proof}
    The first statement leverages the column space inclusion of Corollary~\ref{cor:inclusion} and follows together with Lemma~\ref{lem:FL_des} and Theorem~\ref{thm:lift}. The linear equation~\eqref{eq:FL_coef} is under-determined. Therefore, given some trajectory $(\tilde{\mathsf v},\tilde{\mathsf y})\in\mathfrak C_{L-1}^\text{i/o}$ a particular solution $\mathsf g$ is given in terms of the pseudo-inverse by
    \begin{equation*}
        \mathsf g^i = \begin{bmatrix}
            \mathcal H_{L} (\mathbf{v}_{[0,T-1]}) \\ \mathcal H_{L} (\mathbf{y}_{[0,T-1]})
        \end{bmatrix}^\dagger\begin{bmatrix}
            \mathsf{\tilde v}^i_{[0,L-1]}\\ \mathsf{\tilde y}^i_{[0,L-1]}
        \end{bmatrix}.
    \end{equation*}
    The square summability of this $\mathsf g^i, i\in \mathbb N$ is obvious.
    
    The second assertion follows combining the first one and Theorem~\ref{thm:lift}, where the relationship between $G$ and $\mathsf g$ is established by $G=\sum_{i\in\mathbb N} \phi^i \mathsf g^i$.
\end{proof}

\subsection{Comments}
Several comments on the stochastic fundamental lemma and its context are in order.

\noindent \textbf{Disturbance Modeling and Data Acquisition}. 
For starters, we emphasize that \textit{the Hankel matrices in \eqref{eq:FL_coef} and \eqref{eq:FL_stoch_eq}} are in terms of \textit{realization data}. That is, \textit{they are constructed from measurements} not from PCE data nor from random variables. 
    
In this context, we remark that the conceptual split of the exogenous inputs $(v,V)$ into manipulated controls $(u,U)$ and process disturbances $(u,W)$---which is expressed in \eqref{eq:inputsplit}---can be directly translated by splitting the Hankel matrices accordingly
\[
    \mathcal H_{L} (\mathbf{v}_{[0,T-1]}) = \begin{bmatrix}
        \mathcal H_{L} (\mathbf{u}_{[0,T-1]}) \\
        \mathcal H_{L} (\mathbf{w}_{[0,T-1]})
    \end{bmatrix}.
\]
In turn, this implies that disturbance data $\mathbf{w}_{[0,T-1]}$ is required for the application of the stochastic  fundamental lemma. This can be either done via estimation or, for certain applications, via measurements. The latter is, e.g., the case for energy systems where $(w, W)$ models volatile renewables or consumer demand, i.e., stochastic processes whose realization are accessible through measurements. For first results on the estimation of past disturbance data in case of LTI systems, we refer to~\cite{Pan21s}.
    
The preceding issue hints to the fact that the stochastic process disturbance $W$ could be further split into disturbance and noise
\[
    W_k = W^\text{d}_k +  W^\text{n}_k
\]
where $W^\text{d}_k$ models non-Gaussian disturbances and $W^\text{n}_k$ reflects Gaussian/non-Gaussian noise. Put differently, $W^\text{n}_k$ can be used to capture the $i.i.d.$ zero-mean part of the disturbance, which is usually denotes as \textit{process noise}, while $W^\text{d}_k$ models further disturbances---potentially non-$i.i.d.$ component-wise and in time---acting on the system.
    
Moreover, in any real world application, measurement noise will corrupt the data entering the Hankel matrices. Indeed, in case of LTI systems the issue of robustness of the Hankel matrix descriptions with respect to measurements corrupted by noise has received widespread attention, see  \cite{Dorfler22a,yin2021data,yin2021maximum} for analysis of Hankel matrix predictions and \cite{coulson2019regularized,Berberich22} for robustness analysis in context of data-driven predictive control. \vspace{2mm}
    
\noindent \textbf{Alternative Approaches to Uncertainty Quantification}. 
It is worth to be remarked that the stochastic fundamental lemma provides a handle for uncertainty propagation and Uncertainty Quantification (UQ) in systems which does not rely on explicit model knowledge. To illustrate the appeal of the proposed approach consider 
\begin{align*}
     	x\inst{k+1} &= A(\Theta) x\inst{k} +\widetilde B(\Theta) v\inst{k}\\
   		 y\inst{k} &= C(\Theta)  x\inst{k}+\widetilde D(\Theta)  v\inst{k}
\end{align*}
where $\Theta \in L^2(\Omega,\mathbb R^{n_\theta})$ models the \emph{epistemic} uncertainty surrounding system matrices. Put differently, the realization $\Theta(\omega)$ is constant for all time steps $k\in \mathbb N$. Under the conditions outlined in Section~\ref{sec:recapI}, the deterministic fundamental lemma (Lemma~\ref{lem:FL_des}) elegantly covers these cases in the input-output setting and in the input-output-state setting, i.e., it enables UQ. Indeed, without further elaboration, we remark 
\begin{itemize}
    \item that, under the assumption $\Theta\in L^2(\Omega,\mathbb R^{n_\theta})$, PCE has seen frequent use for model-based UQ and control design for LTI counterpart of the uncertain system given above, see~\cite{Wan21,Wan22,Fisher08}, and
    \item that the data-driven setting does not require the assumption $\Theta\in L^2(\Omega,\mathbb R^{n_\theta})$ as long as any realization $\Theta(\omega)$  remains finite and constant over the considered time horizon.
\end{itemize}
Lemma \ref{lem:FL_stoch} pushes UQ and uncertainty propagation for uncertain systems even further as it allows to handle the case \begin{align*}
        X\inst{k+1} &= A(\Theta) X\inst{k} + B(\Theta) U\inst{k} + F(\Theta) W\inst{k}\\
   		 Y\inst{k} &= C(\Theta)  X\inst{k}+ D(\Theta) U\inst{k} +  H(\Theta) W\inst{k}
\end{align*}
with $X_k, U_K, W_K, Y_k \in L^2(\Omega,\mathbb R^{n_z})$, $n_z \in \{n_x,n_u,n_w,n_y\}$ and $\Theta \not \in L^2(\Omega,\mathbb R^{n_\theta})$ but constant over time. We remark that already a model-based PCE approach to the conceptually simpler setting with $\Theta \in L^2(\Omega,\mathbb R^{n_\theta})$ is subject to the fundamental complication that the multiplication of the uncertain system matrices with the random variable states, inputs etc.---e.g., $A(\Theta) X\inst{k}$, $B(\Theta) U\inst{k}$, \dots---leads to multiplication of PCE bases, which induce several technicalities in the Galerkin projection, cf.\ \cite{Muehlpfordt18}. Note that the data-driven approach based on the stochastic fundamental lemma alleviates such problems. Indeed the stochastic fundamental lemma allows to untangle \emph{epistemic uncertainty}---i.e. model-related uncertainty above covered by $\Theta$---from \emph{aleatoric uncertainty} which in the examples above is represented by $W\inst{k}$, cf. \citep{hullermeier2021aleatoric,umlauft2020real}. \vspace{2mm}

\noindent \textbf{Alternative Hankel Constructions in the Lemma}. 
One may wonder what happens if the column space representation in \eqref{eq:FL_stoch_eq} is altered. There are two main ways to do so:

(a) Use Hankel matrices in random variables and a real column-space selector $g$. Leaving the non-subtle technicality of defining persistency of excitation for $\mathbf{V}_{[0,T-1]}$ aside, it can be shown that  the map $\lambda: \mathbb R^{T-L+1} \to L^2(\Omega,\mathbb R^{n_v})^L \times L^2(\Omega,\mathbb R^{n_y})^L$,
\[
    g\mapsto \lambda(g) \doteq \begin{bmatrix}
        \mathcal H_{L} (\mathbf{V}_{[0,T-1]}) \\ \mathcal H_{L} (\mathbf{Y}_{[0,T-1]})
    \end{bmatrix} g,
\]
is indeed such that its codomain satisfies
\[
    \lambda(\mathbb R^{T-L+1})\subset {\mathfrak S_{L-1}^\text{i/o}}.
\]
That is, the image of $g$ under the linear map $\lambda$ specifies an element of $\mathfrak S_{L-1}^\text{i/o}$. Yet, this construction does not span the entire finite-time behavior $\mathfrak S_{L-1}^\text{i/o}$. A detailed PCE-based construction of an LTI counterexample and the proof of the inclusion statement
are given by \cite{Pan21s}. 

(b) Use  Hankel matrices in random variables and a random-variable column-space selector $G$,  i.e., consider the image of a random variable $G:\Omega\rightarrow R^{T-L+1}$ under the Hankel matrices in random variables. At this point, one may conjecture that the fundamental inclusion $\mathfrak B^\text{i/o}_\infty \subset \mathfrak S^\text{i/o}_\infty$ implies that the map $\Lambda$,
 \[
 G\mapsto \Lambda(G)\doteq       \begin{bmatrix}
                \mathcal H_{L} (\mathbf{V}_{[0,T-1]}) \\ \mathcal H_{L} (\mathbf{Y}_{[0,T-1]})
            \end{bmatrix} G,
\]
defined on an appropriate domain gives the entire manifest behavior $\mathfrak S_{L-1}^\text{i/o}$. However, in such a setting one has to handle non-linear operations on PCEs which renders the formal analysis more complicated. In particular, the square-integrability of $\Lambda(G)$ has to be insured. This is in general not the case for $G\in L^2(\Omega, \mathbb R^{T-L+1})$, but for a bounded function $G$.
\section{Data-Driven Stochastic Optimal Control} \label{sec:StochOCPs}
In the preceding sections we introduced behavioral characterizations of stochastic systems. Next we turn towards using these concepts for control. Specifically, we discuss data-driven optimal control of stochastic systems and the numerical solution of the arising problem.

\subsection{Behavioral Problem Formulations with $\mathfrak S_{N}$ and $\mathfrak C_N$ } \label{sec:OCP_explicit}

For starters, we consider the stochastic explicit LTI system~\eqref{eq:RVdynamics} with input partition~\eqref{eq:inputsplit}, we model the stochastic input $U\inst{k}$ as a stochastic process adapted to the filtration $(\mathcal G_k)_{k\in\mathbb N}$ with $\mathcal G_k = \sigma(\mathbf Y_{[k,k+n_x]})$ as in Section~\ref{sec:model_Stochastic}. In the underlying filtered probability space $(\Omega, \mathcal F, (\mathcal G_k)_{k\in\mathbb N}, \prob)$, the control input $U_k$ may only depend on the information available up to the time $k$.

Given a trajectory $(\widehat U,\widehat W,\widehat Y)\in\mathfrak S_{n_x-1}^\text{i/o}$ observed in the past we consider the conceptual behavioral Optimal Control Problem (OCP) associated with \eqref{eq:Realdynamics} for the finite optimization horizon $N$ 
\begin{subequations}\label{eq:OCPRV}
\begin{align}
  &  \operatorname*{minimize}_{U,W,Y}  \,
    \sum_{k=n_x}^{N+n_x-1} \mean[Y_k^\top QY_k + U^\top_k RU_k]\\
  &   \qquad\text{ subject to }\notag\\
  &\qquad (U,W,Y) \in\mathfrak S_{N+n_x-1}^\text{i/o}, \label{eq:OCPRV_behav}\\
      &\qquad \begin{bmatrix}
             \mathbf{ Y}_{[0,n_x-1]}\\
             \mathbf{ U}_{[0,n_x-1]},\\        
             \end{bmatrix} 
             = \begin{bmatrix}
             \mathbf{\widehat Y}_{[0,n_x-1]}\\
             \mathbf{\widehat U}_{[0,n_x-1]} \\
         \end{bmatrix},   \label{eq:OCPRV_consist_cond} \\
            &\qquad\mathbf{W}_{[0,N+n_x-1]} = \widehat{\mathbf{W}}_{[0,N+n_x-1]}, \label{eq:OCPRV_noise}
     \end{align}
\end{subequations}
where $Q$ and $R$ are symmetric positive definite matrices of appropriate dimensions. Optimal solutions are denoted by $ (U^\star,W^\star, Y^\star)$ whereby, due to the noise specification~\eqref{eq:OCPRV_noise}, we have that $W^\star = 
\widehat{W}$. 
The \emph{initial} condition~\eqref{eq:OCPRV_consist_cond} {of length $n_x$} together with the observability of the underlying system guarantees that the latent state trajectory is uniquely defined. 

As the manifest behavior $\mathfrak S_{N+n_x-1}^{\text{i/o}}$ can be characterized by Hankel matrices in realization data (cf.\ Lemma~\ref{lem:FL_stoch}), the following data-driven reformulation of OCP~\eqref{eq:OCPRV} is immediate: 
\begin{subequations}\label{eq:OCPdata}
\begin{align}
    &\operatorname*{minimize}_{U,Y, G }   \,
    \sum_{k=n_x}^{N+n_x-1} \mean[Y_k^\top QY_k + U^\top_k RU_k]\\
   &\qquad \text{ subject to }\qquad \nonumber \\
  & \qquad\begin{bmatrix}
         \mathbf{Y}_{[0,N+n_x-1]}\\
         \mathbf{U}_{[0,N+n_x-1]}\\
        \widehat{\mathbf{W}}_{[0,N+n_x-1]}
    \end{bmatrix}
    =
        \begin{bmatrix}
            \mathcal H_{N+n_x-1}\left(\mathbf{y}^\text{d}_{[0,T-1]}\right)\\
            \mathcal H_{N+n_x-1}\left(\mathbf{u}^\text{d}_{[0,T-1]}\right)\\
            \mathcal H_{N+n_x-1}\left(\mathbf{w}^\text{d}_{[0,T-1]}\right)
        \end{bmatrix}
        G \label{eq:OCPdata_hankel},
\\
    & \qquad  \begin{bmatrix}
             \mathbf{ Y}_{[0,n_x-1]}\\
             \mathbf{ U}_{[0,n_x-1]}        
             \end{bmatrix} 
             = \begin{bmatrix}
             \mathbf{\widehat Y}_{[0,n_x-1]}\\
             \mathbf{\widehat U}_{[0,n_x-1]} 
         \end{bmatrix}.   \label{eq:OCPdata_consist_cond}
     \end{align}
\end{subequations}
Formally the decision variables live in the following spaces,
\begin{align*}
 Y &\in \left( L^2(\Omega,\mathbb R^{n_v})\right)^{N+n_x},
 U \in \left( L^2(\Omega,\mathbb R^{n_u})\right)^{N+n_x}, \\
G &\in L^2(\Omega,\mathbb R^{T-N-n_x+1}).
\end{align*}
The comparison of the OCPs~\eqref{eq:OCPRV} and~\eqref{eq:OCPdata} reveals key differences:
\begin{itemize}
    \item The behavioral membership relation~\eqref{eq:OCPRV_behav} is replaced by the stochastic Hankel matrix description given by~\eqref{eq:OCPdata_hankel}, which is derived in Lem\-ma~\ref{lem:FL_stoch}.
    \item The decision variables change from the element of the behavior $( U, W, Y)$ to $(U,Y,G)$, i.e. the inputs and outputs and the column-space selector vector. 
\item The need to specify the disturbances in \eqref{eq:OCPRV_noise} is alleviated in OCP~\eqref{eq:OCPdata} as this data is directly included in \eqref{eq:OCPdata_hankel}.
\end{itemize}
Moreover, we remark that with straight-forward modifications one can change the initial conditions in OCP~\eqref{eq:OCPRV} and in~\eqref{eq:OCPdata} to (observed/measured) realization values of $Y(\omega), $ $ W(\omega), U(\omega)$ given over $[0, n_x-1]$. Importantly, in \eqref{eq:OCPdata} past disturbance data is also required in the Hankel matrix appearing in \eqref{eq:OCPdata_hankel}.

Given an orthogonal basis of $L^2(\Omega,\mathbb R)$, we want to transfer the OCP~\eqref{eq:OCPRV} and~\eqref{eq:OCPdata} into the setting of expansion coefficients by means of behavioral lift (see Theorem~\ref{thm:lift}). Hereby exactness in the series expansion (see Definition~\ref{def:exactPCE}) of the random variables plays an important role for the numerical solvability of the OCP.

\begin{lemma}[Exact uncertainty propagation via expansions]\label{lem:no_truncation_error}~\\
    Consider the stochastic explicit LTI system~\eqref{eq:RVdynamics} and suppose that $\widehat{W}_{k}$ for $k \in \I_{[0,N+n_x-1]}$, and 
    $\widehat{Y}_{k}$, $\widehat{U}_k$ for $k \in \I_{[0,n_x-1]}$ admit exact series expansion with finite dimensions~$p_w$ and~$p_{\text{ini}}$, i.e., $\widehat{W}_{k}= \sum_{i=0}^{p_w-1}\hat{\pce{w}}^i \phi_{k}^i$, $\widehat{Y}_{k} = \sum_{i=0}^{p_{\text{ini}}-1}\hat{\pce{y}}_k^i \phi_{\text{ini}}^i$, and $\widehat{U}_{k}=  \sum_{i=0}^{p_{\text{ini}}-1}\hat{\pce{u}}_k^i \phi_{\text{ini}}^i$, respectively. {Assume that $\phi_{\text{ini}}^0=\phi_{k}^0=1$ for all $k \in \I_{[0,N+n_x-1]}$.}  Then,
\begin{itemize}
    \item[(i)] the optimal solution $({U}^\star ,{Y}^\star,G^\star)$ of OCP~\eqref{eq:OCPdata} with horizon~$N$ admits exact series expansion with $p$ terms, where $p$ is given by
    \begin{equation*}
    	p = p_{\text{ini}}
    	+( N+ n_x)(p_w-1) \in \mathbb Z_+,
    \end{equation*}
    \item[(ii)] and the finite-dimensional joint basis $(\phi^i)_{i=0}^{p-1}$ reads
    \begin{align*}
        (\phi^i)_{i=0}^{p-1} = (1, \phi_\text{ini}^1,\dots, \phi_\text{ini}^{p_\text{ini}-1}&, \phi_0^1,\dots, \phi_0^{p_w-1},\dots\\
         & \phi_{N+n_x-1}^1,\dots, \phi_{N+n_x-1}^{p_w-1}).
    \end{align*}
\end{itemize}
\end{lemma}

\noindent The proof is based on the observation that the Hankel matrix description~\eqref{eq:OCPdata_hankel} is a linear map, i.e., if {the set of basis vectors needed to describe} $G$ contains the basis vectors that are necessary to represent ${\widehat W}$, ${\widehat Y}$,  and ${\widehat U}$ then the image ${ Y}$  and ${ U}$ can be expressed exactly in the joint basis. Exploiting the assumption of exact series expansion for the problem data $\widehat W_k,  \widehat Y_k, \widehat U_k$ leads directly to the basis construction. A detailed proof is given by \cite{Pan21s}. 

The previous lemma implies that as the prediction horizon $N$ grows, the expansion order required for exactness in (i) increases linearly.  The reason is that the realizations of $W_k$ are independent at each time instant $k \leq N+n_x-1$. Hence, the finite-dimensional polynomial basis in (ii) enables exact propagation of the uncertainties over finite horizons. Moreover, observe that as the number of involved basis vectors grows linearly with the horizon $N$, the number of decision variables in a expansion reformulation grows quadratically in~$N$.

\begin{figure*}[t]
    \begin{center}
        \begin{tikzpicture}[my triangle/.style={-{Triangle[width=\the\dimexpr1.8\pgflinewidth,length=\the\dimexpr0.8\pgflinewidth]}}, my double triangle/.style={{Triangle[width=\the\dimexpr1.8\pgflinewidth,length=\the\dimexpr0.8\pgflinewidth]}-{Triangle[width=\the\dimexpr1.8\pgflinewidth,length=\the\dimexpr0.8\pgflinewidth]}}]
    \footnotesize
    \node[rectangle, draw=none, fill=gray2, minimum height=13.5em, text width=0.35\textwidth, outer sep=2.5pt, inner sep=5pt] (A) {OCP~\eqref{eq:OCPRV}\\\[
        \begin{gathered}
          \operatorname*{minimize}_{U,W,Y}  \,
            \sum_{k=n_x}^{N+n_x-1} \mean[Y_k^\top QY_k + U^\top_k RU_k]\\
            \parbox[t]{5cm}{ subject to initial conditions, noise specifications, and} \\
           (U,W,Y) \in\mathfrak S_{N+n_x-1}^\text{i/o}
             \end{gathered}\]};

    \node[right=45pt of A, rectangle, draw=none, fill=gray3, minimum height=13.5em, text width=0.35\textwidth, outer sep=2.5pt, inner sep=5pt] (B) { OCP~\eqref{eq:OCPPCE}\\ \[
    \begin{gathered}
      \operatorname*{minimize}_{\mathsf u, \mathsf w, \mathsf y}  \,
        \sum_{i=0}^{p-1}\sum_{k=n_x}^{N+n_x-1} \mean[\phi^i \phi^i]\left((\mathsf y_k^i)^\top Q\mathsf y^i_k + (\mathsf u^i_k)^\top R\mathsf u^i_k\right)\\
      \parbox[t]{5cm}{ subject to initial conditions, noise specifications, and for all $i\in\mathbb I_{[0,p-1]}$} \\
       (\mathsf u^i,\mathsf w^i,\mathsf y^i) \in\mathfrak C_{N+n_x-1}^\text{i/o}\\
       \pce{u}_{k}^{i'} = 0, \forall i' \in \I_{[p_\text{ini}+k(p_w-1),p-1]},\, \forall k \in \I_{[0,N+n_x-1]}
         \end{gathered}\]};

    \node[below=20pt of A, rectangle, draw=none, fill=lightblue1, minimum height=16.5em, text width=0.35\textwidth, outer sep=2.5pt, inner sep=5pt] (C) {OCP~\eqref{eq:OCPdata}\\
    \[\begin{gathered}\operatorname*{minimize}_{U,Y, G }   \,
    \sum_{k=n_x}^{N+n_x-1} \mean[Y_k^\top QY_k + U^\top_k RU_k]\\
   \text{ subject to initial conditions and}\\
  \begin{bmatrix}
         \mathbf{Y}_{[0,N+n_x-1]}\\
         \mathbf{U}_{[0,N+n_x-1]}\\
        \widehat{\mathbf{W}}_{[0,N+n_x-1]}
    \end{bmatrix}
    =
        \begin{bmatrix}
            \mathcal H_{N+n_x}\left(\mathbf{y}^\text{d}_{[0,T-1]}\right)\\
            \mathcal H_{N+n_x}\left(\mathbf{u}^\text{d}_{[0,T-1]}\right)\\
            \mathcal H_{N+n_x}\left(\mathbf{w}^\text{d}_{[0,T-1]}\right)
        \end{bmatrix}G
    \end{gathered}\]};

    \node[right=45pt of C, rectangle, draw=none, fill=lightblue2, minimum height=16.5em, text width=0.35\textwidth, outer sep=2.5pt, inner sep=5pt] (D) {OCP~\eqref{eq:OCPdataPCE}\\
    \[\begin{gathered}\operatorname*{minimize}_{\mathsf u,\mathsf y, \mathsf g }   \,
        \sum_{i=0}^{p-1}\sum_{k=n_x}^{N+n_x-1} \mean[\phi^i \phi^i]\left((\mathsf y_k^i)^\top Q\mathsf y^i_k + (\mathsf u^i_k)^\top R\mathsf u^i_k\right)\\
        \mbox{ subject to initial conditions and for all $i\in\mathbb I_{[0,p-1]}$}\\
      \begin{bmatrix}
             \mathsf{y}_{[0,N+n_x-1]}^i\\
             \mathsf{u}_{[0,N+n_x-1]}^i\\
            \hat{\mathsf{w}}_{[0,N+n_x-1]}^i
        \end{bmatrix}
        =
            \begin{bmatrix}
                \mathcal H_{N+n_x}\left(\mathbf{y}^\text{d}_{[0,T-1]}\right)\\
                \mathcal H_{N+n_x}\left(\mathbf{u}^\text{d}_{[0,T-1]}\right)\\
                \mathcal H_{N+n_x}\left(\mathbf{w}^\text{d}_{[0,T-1]}\right)
            \end{bmatrix}\mathsf g^i\\
            \pce{u}_{k}^{i'} = 0, \forall i' \in \I_{[p_\text{ini}+k(p_w-1),p-1]},\, \forall k \in \I_{[0,N+n_x-1]}
        \end{gathered}\]};
    
    \draw[line width=6pt,my double triangle, draw=lightblue](A) -- (B) node [above, midway] (n1) {\parbox[t]{1.2cm}{\centering exact and\\ finite series expansion}};
    \draw[line width=6pt,my double triangle, draw=none](A) -- (B) node [below, midway] (n2) {Lemma~\ref{lem:no_truncation_error}};

    \draw[line width=6pt,my double triangle, draw=lightblue](C) -- (D) node [above, midway] (n1) {\parbox[t]{1.2cm}{\centering exact and\\ finite series expansion}};
    \draw[line width=6pt,my double triangle, draw=none](C) -- (D) node [below, midway] (n2) {Lemma~\ref{lem:no_truncation_error}};

    \draw[line width=6pt,my double triangle, draw=lightblue]([xshift=-40pt]A.south) -- ([xshift=-40pt]C.north) node [right, midway] (n3) {\parbox[t]{4cm}{\centering stochastic fundamental lemma\\Lemma~\ref{lem:FL_stoch}}};

    \draw[line width=6pt,my double triangle, draw=lightblue]([xshift=40pt]B.south) -- ([xshift=40pt]D.north) node [left, midway] (n3) {\parbox[t]{4cm}{\centering stochastic fundamental lemma\\Lemma~\ref{lem:FL_stoch}}};
\end{tikzpicture}
    \end{center}
     \caption{Equivalence of the stochastic OCPs~\eqref{eq:OCPRV}, \eqref{eq:OCPdata}, \eqref{eq:OCPPCE}, and \eqref{eq:OCPdataPCE}.} \label{fig:OCPs}
    \end{figure*}
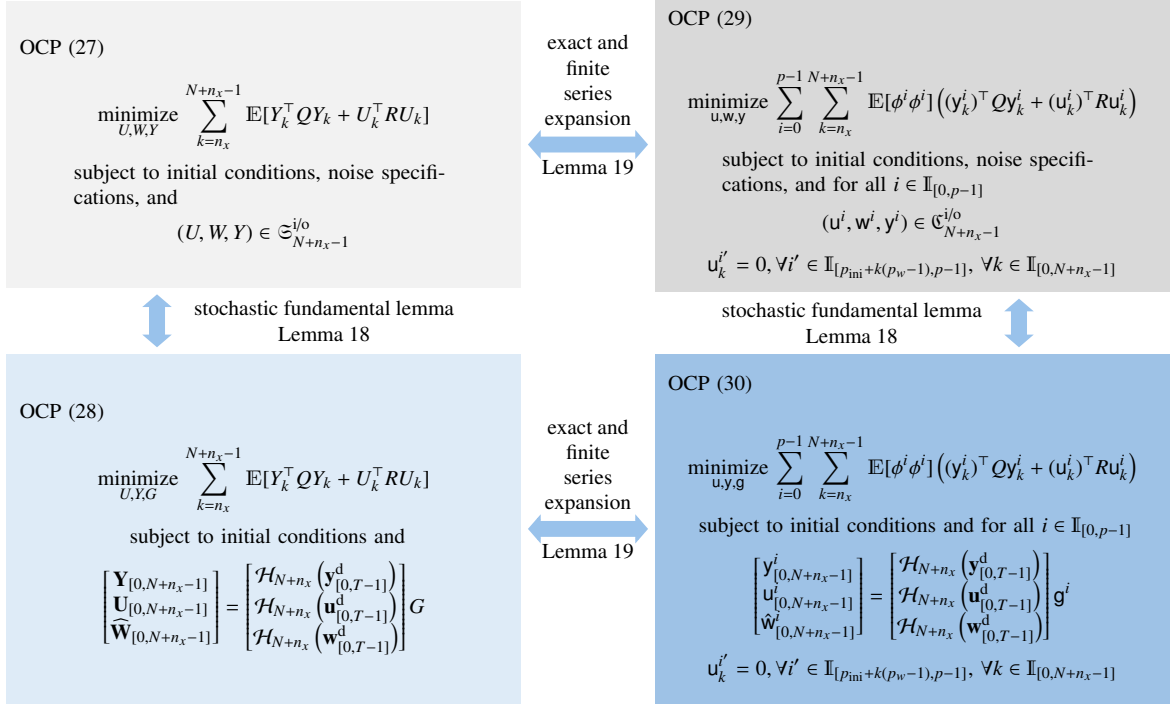

\begin{remark}[Exact PCE]
    Choosing the orthogonal basis $(\phi^i)_{i\in\mathbb N}$ in a smart way the order in the exact series expansion may be drastically reduced, which is beneficial from a numerical point of view. One approach in this direction is discussed in Subsection~\ref{sec:PCE}, where the PCE basis is selected according to the exogenous noise, cf.\ Table~\ref{tab:askey_scheme}. One drawback is that the PCE basis in general does not span the $L^2$-spaces associated with the $\sigma$-algebra $\mathcal F$, but a smaller $L^2$-space with respect to a sub-$\sigma$-algebra. Provided exactness in terms of the PCE basis, one can show similarly to  Lemma~\ref{lem:no_truncation_error} that the optimal solution $(U^\star, Y^\star, G^\star)$ of OCP~\eqref{eq:OCPdata} admits an exact series expansion with respect to the PCE basis.
\end{remark}
\noindent With regard to the previous remark we consider in the following a PCE basis $(\phi^i)_{i\in\mathbb N}$ neglecting the fact that this basis might not span the whole space $L^2((\Omega,\mathcal F,\prob), \mathbb R)$. 

Now, we reformulate the behavioral OCP~\eqref{eq:OCPRV} in the finite-horizon manifest PCE coefficient behavior $\mathfrak C^{\text{i/o}}_N$:
\begin{subequations}\label{eq:OCPPCE}
\begin{align}
    &\operatorname*{minimize}_{\mathbf{\mathsf u}, \mathbf{\mathsf w}, \mathbf{ \mathsf y} }  \,
    \sum_{i=0}^{p-1}\sum_{k=n_x}^{N+n_x-1} \mean[\phi^i \phi^i]\left((\mathsf y_k^i)^\top Q\mathsf y^i_k + (\mathsf u^i_k)^\top R\mathsf u^i_k\right) \\
    &\qquad \text{ subject to }\notag 
    \\ &\quad\quad({\mathsf u},{\mathsf w}, {\mathsf y}) \in\mathfrak C_{N+n_x-1}^\text{i/o} \label{eq:OCPPCE_behav},\\
            &\quad \quad \forall i \in \I_{[0,p-1]}\nonumber\\
&\qquad  \begin{bmatrix}
             \mathbf{ \mathsf y}^i_{[0, n_x-1]}\\
             \mathbf{ \mathsf u}^i_{[0,n_x-1]}\\ 
             \end{bmatrix} 
             = \begin{bmatrix}
             \mathbf{\hat {\mathsf y}}^i_{[0,n_x-1]}\\
             \mathbf{\hat {\mathsf u}}^i_{[0,n_x-1]} \\
         \end{bmatrix},  \label{eq:OCPPCE_consist_cond} \\
      &\qquad      \mathbf{\mathsf w}^i_{[0,N+n_x-1]} = \hat{\mathbf{\mathsf w}}^i_{[0,N+n_x-1]} \label{eq:OCPPCE_noise},\\
     & \qquad     	\pce{u}_{k}^{i'} = 0, \forall i'\in \I_{[p_\text{ini}+k(p_w-1),p-1]},\, \forall k \in \I_{[0,N+n_x-1]}. \label{eq:OCPPCE_causality}
     \end{align}
\end{subequations}
In comparison to OCP~\eqref{eq:OCPRV} which is formulated in the finite-horizon manifest behavior of expansion coefficients $\mathfrak S^{\text{i/o}}_N$, OCP~\eqref{eq:OCPPCE} is structurally similar in terms of the objective and the constraints \eqref{eq:OCPPCE_behav}--\eqref{eq:OCPPCE_noise}. That is to say,  \eqref{eq:OCPPCE_behav}--\eqref{eq:OCPPCE_noise} can be derived by applying Galerkin projections to their random-variable counterparts \eqref{eq:OCPRV_behav}--\eqref{eq:OCPRV_noise}. The main structural change occurs in \eqref{eq:OCPPCE_causality} as this constraint models the causality requirement of the filtered stochastic process $U$ in the applied orthogonal basis $(\phi^i)_{i\in\mathbb N}$. Note that in the behavioral OCP~\eqref{eq:OCPRV} in random variables this causality is implicitly handled in the choice of the filtration $(\mathcal G_k)_{k\in\mathbb N}$.

Similar to the change from OCP~\eqref{eq:OCPRV} to OCP~\eqref{eq:OCPdata} we now reformulate OCP~\eqref{eq:OCPPCE} in data-driven fashion:
\begin{subequations}\label{eq:OCPdataPCE}
\begin{align}
    &\operatorname*{minimize}_{\mathbf{\mathsf u},\mathbf{\mathsf y}, \mathbf{ \mathsf g} }  \,
    \sum_{i=0}^{p-1}\sum_{k=n_x}^{N+n_x-1} \mean[\phi^i \phi^i]\left((\mathsf y_k^i)^\top Q\mathsf y^i_k + (\mathsf u^i_k)^\top R\mathsf u^i_k\right)  \\
   &\qquad \text{ subject to } \forall i \in \I_{[0,p-1]}, \quad \nonumber\\
  &\qquad\begin{bmatrix}
         \mathbf{\mathsf y}^i_{[0,N+n_x-1]}\\
         \mathbf{\mathsf u}^i_{[0,N+n_x-1]}\\
        \hat{\mathbf{\mathsf w}}^i_{[0,N+n_x-1]}
    \end{bmatrix}
     =
     \begin{bmatrix}
        \mathcal H_{N+n_x}\left(\mathbf{y}^\text{d}_{[0,T-1]}\right)\\
        \mathcal H_{N+n_x}\left(\mathbf{u}^\text{d}_{[0,T-1]}\right)\\
        \mathcal H_{N+n_x}\left(\mathbf{w}^\text{d}_{[0,T-1]}\right)
    \end{bmatrix}
     \mathsf   g^i, \label{eq:OCPdataPCE_hankel}
     \\
     &\qquad  \begin{bmatrix}
        \mathbf{ \mathsf y}^i_{[0, n_x-1]}\\
        \mathbf{ \mathsf u}^i_{[0,n_x-1]}\\ 
        \end{bmatrix} 
        = \begin{bmatrix}
        \mathbf{\hat {\mathsf y}}^i_{[0,n_x-1]}\\
        \mathbf{\hat {\mathsf u}}^i_{[0,n_x-1]} \\
    \end{bmatrix},
      \label{eq:OCPdataPCE_consist_cond} \\
            	&\qquad\pce{u}_{k}^{i'} = 0, \forall i' \in \I_{[p_\text{ini}+k(p_w-1),p-1]},\, \forall k \in \I_{[0,N+n_x-1]}. \label{eq:OCPdataPCE_causality}
     \end{align}
\end{subequations}
We note that OCP~\eqref{eq:OCPdataPCE} gives a computationally tractable reformulation of OCP~\eqref{eq:OCPRV}. Specifically, in view of exact propagation (cf. Lemma~\ref{lem:no_truncation_error}), we emphasize the finite-dimen\-sional nature of OCP \eqref{eq:OCPdataPCE}. Indeed, \eqref{eq:OCPdataPCE} is an equality constrained quadratic program.

The entire process of reformulations and the equivalence relations between the four considered OCPs are summarized in Figure~\ref{fig:OCPs}.

\subsection{Constrained Formulations and Implementation Aspects}
Several comments  are in order: on the proposed data-driven OCPs, on how to extend their formulations with chance constraints, and on their numerical implementation. \vspace{2mm}

\noindent \textbf{Chance Constraints}.
So far the considered OCPs~\eqref{eq:OCPdata} and \eqref{eq:OCPdataPCE} involve equality constraints which model behavioral constraints, consistency conditions (a.k.a. initial conditions of the dynamics), and causality requirements. 
Yet, in many applications it is of interest to also model inequality constraints in a probabilistic/stochastic setting. \vspace{1mm}

Consider a scalar box constraint $z\in [\underline z, \bar z]$. Lifting $z\in \mbb{R}$ to a probability space, i.e. to  $Z \in L^2(\Omega, \mathbb R)$, the previous deterministic requires attention. There are three main options:

\textit{In-expectation constraint satisfaction}, i.e., one uses 
    \[\mean[Z] \in [\underline{z}, \bar z].\] This is a weak formulation in the sense that for arbitrary $ L^2(\Omega, \mathbb R)$ random variables, it might happen that 
    \[\forall\omega \in \Omega: Z(\omega) \not\in [\underline{z}, \bar z]\]
    while $\mean[Z] \in [\underline{z}, \bar z]$. In this case, the in-expectation satisfaction of the constraint might lead to erroneous conclusions.
   
\textit{Robust constraint satisfaction}, i.e., one imposes that
\[
    Z(\omega) \in [\underline{z}, \bar z] \qquad\forall\,\omega \in \Omega.
\]
In this setting the worst case outcome $\omega \in \Omega$ will likely dictate whether or not the constraint can be satisfied with certainty. Moreover, observe that in case of random variables with unbounded set of outcomes $\Omega$ (e.g. Gaussian) such a constraint can never be satisfied with certainty.

\textit{Probabilistic constraint satisfaction}, i.e., one requires the constraint to hold in probability
     \[\mbb{P}\left[Z \in [\underline{z}, \bar z]\right] \geq 1 - \varepsilon,\]
     whereby the parameter $\varepsilon \in [0,1]$ and $1 - \varepsilon$ is usually denoted as confidence level. Constraints of this type are referred to as \textit{chance constraints}. In case $\varepsilon = 0$ holds, we say the constraint is satisfied \textit{almost surely}.

In many applications chance constraints are of tremendous interest as, in particular in context of optimization problems, they allow to trade-off performance against constraint satisfaction, see \cite{Mesbah16a,Bienstock14,heirung2018stochastic,Farina16a} for tutorial introductions. We also observe that there is a subtle difference between robust and almost surely constraint satisfaction, as the later allows for constraint violation on subset of $\Omega$ with measure zero.

Naturally, this raises the question of how to formulate chance constraints in a numerically tractable fashion in the PCE framework. In stochastic MPC a common reformulation of scalar chance constraints is
\[
\mbb{E}[Z] \pm \sigma(\varepsilon)\sqrt{\operatorname{Cov}[Z,Z]}\in [\underline{z}, \bar z],   
\]
cf. \cite{Farina16a}.
A conservative choice for $ \sigma$ is given by $\sigma(\varepsilon) = \sqrt{\frac{2-\varepsilon}{\varepsilon}}$. In case of Gaussian distributions one can also choose $\varepsilon$ according to the standard normal table to avoid conservatism.

Let the series expansion of $Z$ be given as $Z = \sum_{i=0}^{p_Z-1} \mathsf z^i\phi^i$. Using \eqref{eq:moments_PCE} we obtain
\[
\mathsf z^0 \pm \sigma(\varepsilon)\sqrt{\sum_{i=1}^{p_Z-1} (z^i)^2\mean[\phi^i \phi^i]}\in [\underline{z}, \bar z].   
\]
It is straight-forward to see that this reformulation directly leads to second-order cone constraints. For further details on the reformulation of chance constraints we refer to \cite{Farina16a,calafiore06}.
\vspace{2mm}

\noindent \textbf{Initial Conditions}. 
In the OCPs~\eqref{eq:OCPRV} and \eqref{eq:OCPdata} we have phrased the initial conditions \eqref{eq:OCPRV_consist_cond} and \eqref{eq:OCPdata_consist_cond} in terms of random variables. This setting entails the (more) common situation of deterministic initial conditions obtained from measurements as a special case. Moreover, it allows to  model additive measurement noise in the PCE framework by
\[
\widehat Y_k = y_k + M_k
\]
with  $M_k \in  L^2(\Omega, \mathbb R^{n_y})$. Suppose that a finite PCE for $M_k$ is known, and that $M_k$ has zero mean, then the PCE for $\widehat Y_k$ is immediately obtained. We emphasize that, for deterministic initial data, the PCE formulation of the consistency conditions~\eqref{eq:OCPPCE_consist_cond} and~\eqref{eq:OCPdataPCE_consist_cond} is directly able to handle this. All it takes is to set the PCE coefficients $\hat y_k^i = 0$ for $i>0$. 

Moreover, it is well-known that noise corrupted data in the Hankel matrices and in the consistency constraints might lead to infeasibility or to deficient numerical solution properties of the Hankel matrix constraints~\eqref{eq:OCPdata_hankel} and \eqref{eq:OCPdataPCE_hankel}. A common remedy is to add  slack vectors $\sigma^i$ of appropriate dimension
\begin{equation} \label{eq:IniSlack}
\begin{bmatrix}
             \mathbf{ \mathsf y}^i_{[0, n_x-1]}\\
             \mathbf{ \mathsf u}^i_{[0,n_x-1]}\\ 
             \end{bmatrix} 
             = \begin{bmatrix}
             \mathbf{\hat {\mathsf y}}^i_{[0,n_x-1]}\\
             \mathbf{\hat {\mathsf u}}^i_{[0,n_x-1]} \\
         \end{bmatrix} + \begin{bmatrix}
                      \sigma^i \\ 0 
         \end{bmatrix}
\end{equation}
and to penalize them in the objective. Analysis on the implication of different penalization strategies has been done by, e.g., \cite{coulson2019regularized,yin2021data}. \vspace{2mm}

\noindent \textbf{Numerical Implementation and Toolboxes}. 
The comment on slack variables above has already addressed aspects of numerical implementation. However, this subject warrants further discussion. 

For starters, observe that the usual Hankel matrix equality constraint---\eqref{eq:OCPdata_hankel} and \eqref{eq:OCPdataPCE_hankel}---entails large dense matrices. This as such is numerically not beneficial. This is evident upon comparison to model-based linear quadratic OCP formulations with inequality constraints in which the state-recursion typically results in sparse equality constraints of favourable structure~\citep{Axehill15}. 

It is known that Hankel matrices can also be constructed from segmented data~\citep{WDPCT20}. From a numerical perspective, it even more promising to segment the time horizon, i.e., to use Hankel matrices of smaller dimension and to couple the solution pieces by continuity constraints. This idea has been suggested by \cite{ODwyer21}. It resembles the classic concept of multiple shooting in the data-driven setting~\citep{Bock84a}. 
In a recent paper, it is shown that data-driven multiple shooting can be applied to the stochastic setting of OCP~\eqref{eq:OCPdataPCE}. Specifically, one can combine the multiple shooting idea with moment matching. This way the dimension of the PCE basis, and thus the number of decision variables can be reduced considerably. We refer to \cite{tudo:ou23a} for further details. 

Another aspect which requires attention is the computation of the numerous quadratures needed to evaluate $\mean[\phi^i\phi^i]$ and to perform Galerkin projection of nonlinear equations. Fortunately, there exists a number of efficient numerical packages which can be used. This includes tools for \textsf{Matlab}~\citep{petzke2020pocet}, \textsf{julia} \citep{tudo:muehlpfordt20c}, and \textsf{python}~\citep{feinberg2015chaospy,baudin2017openturns}.
\section{Extension to Descriptor Systems}\label{sec:ext}

From a behavioral perspective algebraic equality constraints in linear systems are related to the chosen representation (i.e. the chosen state space) and they can be avoided through a procedure which permutes input and output variables~\citep{willems86i,willems07}. However, from an engineering perspective algebraic constraints and descriptor structures arise from modeling choices \citep{Kunkel06,Biegler12,campbell2019applications} while often application requirements assign inputs and outputs.

A prominent example of this crux are electrical power systems wherein the generator powers are the control inputs, while the underlying electrical grid induces algebraic constraints~\citep{Gross16,milano13}. At the same time, stochastic uncertainty surrounding renewables is a key challenge in energy systems, see~\citep{Bienstock14,milano13} for electric systems and~\citep{zavala2014stochastic} for gas networks.
Moreover, data-driven approaches to power and energy systems are of increasing interest, see, e.g., ~\citep{Huang21decentralized,Cremer2018data,Schmitz22,venzke2021efficient, BILGIC2022}.

Hence, this section investigates the implications of linear descriptor structures on data-driven stochastic optimal control. 
We first recall results on data-driven approaches to descriptor systems, before we extend the investigations to the stochastic descriptor setting. We commence by recapitulating the model-based analysis of regular descriptor systems via the quasi-Weierstraß form.

\subsection{The Quasi-Weierstraß Form}
\label{sec:qWeier}

We consider discrete-time LTI \emph{descriptor} system given by
\begin{subequations}\label{sys_dae}
    \begin{align}
        E x_{k+1} &= A x_k + B u_k \label{sys_dae_a} \\
        y_k &= Cx_k + D u_k \label{sys_dae_b}
        \end{align}
\end{subequations}
with the descriptor matrix $E \in \mathbb R^{n_x\times n_x}$, $\operatorname{rk}(E)<n_x$. Descriptor representations  allow to explicitly model algebraic constraints.  They arise, e.g., from discretization of differential-algebraic systems or  from systems with separated time-scales. If $E$ is an invertible matrix, system~\eqref{sys_dae} can straightforwardly be written as in explicit form~\eqref{sys}. 

Henceforth, we assume that the matrix pencil $\lambda E-A$ is \emph{regular}, i.e., $\det (\lambda E-A)\neq 0$ holds for some $\lambda\in \mathbb C$. In this case there exist invertible matrices $P$,~$S\in\mathbb R^{n_x\times n_x}$ such that the pencil $\lambda E-A$ can be transformed into the quasi-Weierstraß form
\begin{equation}\label{qweier}
    S(\lambda E-A)P = \lambda\begin{bmatrix}
        I_{n_J} & \\ & N
    \end{bmatrix} - \begin{bmatrix}
        J & \\ & I_{n_N}
    \end{bmatrix},
\end{equation}
where $N\in\mathbb R^{n_N\times n_N}$ is a nilpotent matrix and $J\in\mathbb R^{n_J\times n_J}$ with $n_J+n_N=n_x$, cf.\ \cite{BergerIlchmannTrenn12,Dai89,Kunkel06}. Whenever $E$ is invertible, then $n_N=0$ and $N$ is an empty matrix.

Due to regularity of the underlying matrix pencil, the system representation~\eqref{sys_dae} can be transformed into the equivalent one
\begin{subequations}\label{qsys}
    \begin{align}\label{qsysa}
            \begin{bmatrix}
                I_{n_J} & \\ & N
            \end{bmatrix} \begin{bmatrix}
                z_{k+1}^J \\ z_{k+1}^N
            \end{bmatrix} &= \begin{bmatrix}
                J & \\ & I_{n_N} 
            \end{bmatrix} 
            \begin{bmatrix}
                z_{k}^J \\ z_{k}^N
            \end{bmatrix}
            + \begin{bmatrix}
                B_J \\ B_N
            \end{bmatrix} u_k\\
            \label{qsysb}
            y_k &= \begin{bmatrix}
                C_J & C_N
            \end{bmatrix}
            \begin{bmatrix}
                z_{k}^J \\ z_{k}^N
            \end{bmatrix} + Du_k,
    \end{align}
\end{subequations}
which we refer to as \emph{quasi-Weierstraß form} and where
\begin{equation}
        \label{transform}
        SB = \begin{bmatrix}
            B_J \\ B_N
        \end{bmatrix},\quad CP = \begin{bmatrix}
            C_J & C_N
        \end{bmatrix},\quad P^{-1}x_k = z_k = \begin{bmatrix}
            z^J_k\\ z^N_k
        \end{bmatrix}.
\end{equation}
Note that the dynamics~\eqref{qsysa} are decoupled. They consists of the \emph{dynamic} part for $z^J_k$ and the \emph{algebraic} part for $z^N_k$.
        The state evolution for the quasi-Weierstraß form~\eqref{qsys} is given by
    \begin{equation}
        \label{evolution}
        z_k^J = J^k z_0^J + \sum_{i=1}^k J^{k-i} B_J u_{i-1},\quad z_k^N = - \sum_{i=0}^{\delta-1} N^i B_N u_{k+i},
\end{equation}
cf.~\cite{BelovAndrianovaKurdyukov18}, where $\delta$ is the structured nilpotency index as defined in  Definition~\ref{def:StructuredNilpotencyIndex} below. 

The solution~\eqref{evolution} indicates that descriptor systems~\eqref{sys_dae} may be considered as non-causal, i.e., the present state is influenced by input actions at subsequent future time instances. Alternatively, one may regard the choice of future input actions as constrained by the present value of the state. Commonly, the \emph{nilpotency index} $\hat\delta$ of the matrix $N$, i.e.\ $N^{\hat \delta}=0$ and $N^{\hat\delta-1}\neq 0$, is used~\citep{BelovAndrianovaKurdyukov18}. This can be further improved by using the following, slightly different definition, which fosters our further analysis of system~\eqref{sys_dae}.
    \begin{definition}[Structured nilpotency index]\label{def:StructuredNilpotencyIndex}
        The number
    \begin{equation}
    \label{inputindex}
        \delta \doteq \begin{cases} \min\{i\in\mathbb N\,|\, N^iB_N=0\} &\text{if }n_N>0 \text{ and }B_N\neq 0,\\
        1 & \text{otherwise}
        \end{cases}
    \end{equation}
    is called the \emph{structured nilpotency index}~$\delta$ of system~\eqref{sys}.
\end{definition}
    
As we will see below, in the light of data-driven control, the structured nilpotency index should be preferred over its upper bound given by the nilpotency index~$\hat \delta$, since it leads to tighter estimates on the data requirements. We further remark that, although the quasi-Weierstraß form~\eqref{qweier} is not unique, the structured nilpotency index~$\delta$, the nilpotency index $\hat \delta$, and the dimensions $n_J$, $n_N$ are uniquely determined. In particular, they do not depend on the choice of the transformation matrices~$S$ and~$P$, cf. \cite[Lemma~2.10]{Kunkel06}. Moreover, we mention the close conceptual relation between the structured nilpotency index and the input index for continuous-time descriptor system introduced by \cite{ilchmann2018model, IlchLebe19}. 

    \begin{example}[Structured nilpotency index $<$ nilpotency index]
        Consider a system given in quasi-Weierstraß form where
        \begin{equation}
        \nonumber
            N=\begin{bmatrix}
                0&1 & 0 &0\\
                0&0&1& 0\\
                0&0&0&1\\
                0&0&0&0
            \end{bmatrix},\quad B_N=\begin{bmatrix}
                1&0\\0&1\\0&0\\0&0
            \end{bmatrix}.
        \end{equation}
        The nilpotency index of $N$ is $\hat\delta=4$. The columns of $B_N$ stand orthogonal on the rows of $N^2$, while $N B_N\neq 0$. Therefore, the structured nilpotency index is $\delta=2$.
    \end{example}
    
From \eqref{evolution} one immediately observes that in the case $N B_N\neq 0$, i.e.\ $\delta\geq 2$, systems~\eqref{qsys} and \eqref{sys_dae} are non-causal, i.e., the value of the state at time~$k$ depends
on the input signal until time $k+\delta-1$. Further, the representation~\eqref{evolution} allows to describe the set of \emph{consistent initial values}
    \begin{equation}
        \label{consistentInit}
        \begin{split}
        \mathcal X^0 &\doteq \left\{P\begin{bmatrix}
            \xi^J\\ \xi^N
        \end{bmatrix} \,\middle |\, \begin{gathered}\xi^J\in\mathbb R^{n_J}, \xi^N = \sum_{i=0}^{\delta-1} N^i B_N u_{i}\\ \text{for some }u_0,\dots,u_{\delta-1}\in\mathbb R^{n_u} \end{gathered}\right\}\\
        \end{split},
    \end{equation}
    which is the column span
    of the matrix
    \[P\begin{bmatrix}
            I & 0\\
            0 & \begin{bmatrix}
                B_N&\dots & N^{\delta-1}B_N
            \end{bmatrix}
        \end{bmatrix}\in\mathbb R^{n_x\times (n_J+n_u\delta)},\]
        cf.\ \cite{BelovAndrianovaKurdyukov18}.
    Next, we recall controllability and observability concepts for descriptor systems as introduced by~\cite{Dai89} using the equivalent characterizations in accordance to~\cite{BelovAndrianovaKurdyukov18,Stykel02}.

\begin{definition}[R-controllability and R-observability]~\\
    System~\eqref{sys_dae} is said to be \emph{R-controllable}, if
    \begin{equation}
        \operatorname{rk}\left(\begin{bmatrix}
                \lambda E-A & B
            \end{bmatrix}\right) = n_x\quad\text{for all }\lambda\in\mathbb C. 
        \end{equation}
        Further, system~\eqref{sys_dae} is said to be \emph{R-observable}, if
        \begin{equation}
            \operatorname{rk}\left(\begin{bmatrix}
                \lambda E-A\\
                C
            \end{bmatrix}\right) = n_x\quad\text{for all }\lambda\in\mathbb C. 
        \end{equation}
          \end{definition}
        \begin{remark}[Controllability/observability in the quasi-Weierstraß form] \label{rem:controllability/observability_weierstrass}
            R-controllability can be equivalently characterized by the Kalman-like criterion for the quasi-Weierstraß form~\eqref{qweier},
            \begin{subequations}
            \begin{equation}
                \label{kalmanctrl}
                \operatorname{rk}\left(\begin{bmatrix}
                    B_J & JB_J & \dots & J^{n_J-1} B_J
                \end{bmatrix}\right) = n_J.
            \end{equation}
    Similarly, system~\eqref{sys_dae} is R-observable if and only if
            \begin{equation}
                \label{kalmanobsv}
                \operatorname{rk}\left(\begin{bmatrix}
                    C_J \\ C_JJ \\ \vdots \\ C_JJ^{n_J-1}
                \end{bmatrix}\right) = n_J.
            \end{equation}                            
            \end{subequations}
        \end{remark}

\begin{remark}[Equivalent explicit LTI representation]\label{rem_causal}
    It was shown by \cite{willems86i} that the manifest behavior $\mathfrak B_\infty^{i/o}$ emerging from a descriptor system~\eqref{sys_dae} can be always represented in terms of an explicit LTI system~\eqref{sys}. This equivalent representation comes at the expense of increasing the state dimension. The key is to swap the role of certain input and output components in order to resolve the non-causality raised by the descriptor representation. To quote \cite{willems86i}: \emph{causality is a matter of representation [...] one can always obtain causality by properly interpreting the variables}. This reinterpretation of inputs and outputs is completely natural from a behavioral point of view. Yet, it may debatable if the \textit{choice of specific input and output variables} is determined by application requirements. A constructive argument tailored to regular descriptor systems regarding the above equivalence is given in the appendix.
\end{remark}

The next lemma, which tightens the results of \cite[Lemma~1]{SchmitzFaulwasserWorthmann22} by using the structured nilpotency index~$\delta$, states a lower bound on the length of input-output trajectories such that the state $x$ is uniquely determined.
\begin{lemma}[Uniqueness of state trajectories]
    \label{state_align}
    Suppose that the system represented by~\eqref{sys} is R-observable. If $(x,u,y)$, $(\tilde x, u, \tilde y)\in\mathfrak B_{n_J-2+\delta+T}$ for some $T\in\mathbb N$ such that $y|_{[0,n_J-1]}=\tilde y|_{[0,n_J-1]}$, then 
    \[
    x|_{[0,n_J-1+T]}=\tilde x|_{[0,n_J-1+T]}.
    \]
\end{lemma}
\begin{proof}
    Let $z=P^{-1}x$ be the state variable of the quasi-Weier\-straß form~\eqref{qweier} given in \eqref{transform} with its components $z^J$ and $z^N$; similar for $\tilde z=P^{-1} \tilde x$. As the inputs of both trajectories $(x,u,y)$ and $(\tilde x,u,\tilde y)$ coincide until time $n_J-2+\delta$ and the outputs coincide until time $n_J-1$ one has by \eqref{qsysb} and~\eqref{evolution}
    \begin{equation*}
        0=y_k - \tilde y_k = C_J J^k (z_0^J -\tilde z_0^J )
    \end{equation*}
    for all $k=0,\dots, n_J-1$.  R-observability, i.e.\ \eqref{kalmanobsv}, implies $z_0^J=\tilde z_0^J$. According to \eqref{evolution} this yields $z^J|_{[0,n_J-1+T]}=\tilde z^J|_{[0,n_J-1+T]}$ and $z^N|_{[0,n_J-1+T]}=\tilde z^N|_{[0,n_J-1+T]}$. The assertion follows with the transformation $x=Pz$ and $\tilde x=P\tilde z$. 
\end{proof}

\subsection{Comments on the LTI Behavior}

In addition, regularity of the descriptor representation~\eqref{sys_dae} is related to the behavioral concept of autonomy introduced in \cite[Section~3.2]{PoldermanWillmes98}, i.e.,  the property of a system that the past of a trajectory completely determines its future. More precisely, regularity implies that the system with zero input is autonomous, i.e.,\ for all $(x,0,y), (\tilde x,0,\tilde y)\in\mathfrak B_\infty$ with $x(0)=\tilde x(0)$ one has $(x,0,y)= (\tilde x,0,\tilde y)$, cf.~\citep[Corollary~5.2]{BergerReis13}.

Further, we like to point out that the bound on the \emph{transition delay~$T'$}, which appears in the behavioral definition of controllability (Definition~\ref{def:behav_controlability}), can be tightened for the system representation~\eqref{sys_dae}.
\begin{lemma}[Controllability of the behavior]
    \label{ctrlRbehav}
    The behavior $\mathfrak B_\infty$ is controllable if and only if the system representation~\eqref{sys_dae} is R-controllable. In this case the transition delay can be chosen as $T'= \delta-1+n_J$ independently from the particular trajectories.
\end{lemma}
\begin{proof}
    The proof is based on the observation that the behavior $\mathfrak  B_\infty$ admits a kernel representation. Specifically, $(x,u,y)\in\mathfrak B_\infty$ if and only if 
    \begin{equation} \label{eq:detKernelBehav}
    R(\sigma)\begin{bmatrix}
        x^\top & u^\top & y^\top
    \end{bmatrix}^\top = 0,\end{equation}
    where $R$ is the polynomial matrix,
    \begin{equation*}
        R(\lambda)=\begin{bmatrix}E & 0 & 0\\ 0 & 0 & 0\end{bmatrix}\lambda+\begin{bmatrix}
            -A & -B & 0\\
            C &  D & -I
        \end{bmatrix}\in\mathbb R^{(n_x+n_y)\times (n_x+n_u+n_y)}[\lambda],
    \end{equation*}
     and $\sigma$ denotes the left shift operator on $(\mathbb R^{n_x+n_u+n_y})^{\mathbb N}$, i.e.\ $(\sigma f)_k=f_{k+1}$. Controllability of $\mathfrak B_\infty$ is equivalently characterized via the criterion that $R(\lambda)$ has constant rank for all $\lambda\in\mathbb C$, see \cite{MarkovskyWillemsHuffelDeMoor06}. That is
    \begin{equation*}
        n_x+n_y=\operatorname{rk}(R(\lambda)) =
        \operatorname{rk}\left(\begin{bmatrix}
            \lambda E-A & B
        \end{bmatrix}\right) +n_y
    \end{equation*}
    for all $\lambda\in\mathbb C$. The latter condition is, in turn, equivalent to R-controllability of~\eqref{sys_dae}.

    It remains to show that the transition delay does not depend on the particular trajectories. Let $b=(x,u,y)$,~$\tilde b=(\tilde x, \tilde u,\tilde y)\in\mathfrak B_\infty$ and set $z=P^{-1}x$, $\tilde z =P^{-1}\tilde x$, where $P$ is the transformation matrix used to obtain~\eqref{qweier}. Furthermore, fix $T\in\mathbb Z_+$ and set $T' =\delta-1 + n_J$. We choose
\begin{equation}
    \label{baum}
    u'_k=u_k\text{ for }k\leq T-1+\delta-1,\quad u'_k=\tilde u_k\text{ for }k\geq T+T'.
\end{equation}
Since the system representation is R-controllable, that is \eqref{kalmanctrl} holds, we find values $u'_k$ for $T+\delta-1\leq k\leq T+T'-1$ such that
\begin{equation}
    \label{pilz}
    J^{T+T'} z_{0}^J + \sum_{i=1}^{T+T'} J^{T+T'-i} B_J u'_{i-1} = \tilde z^J_{0}.
\end{equation}
A solution of \eqref{baum} is, e.g.,\ given via the pseudo-inverse,
\begin{equation}
    \label{wald}
\begin{gathered}
    \begin{bmatrix}
        u'_{T+T'-1} \\ \vdots \\ u'_{T+\delta-1}
    \end{bmatrix}
    =  \\ \begin{bmatrix}B_J \dots J^{n_J-1} B_J\end{bmatrix}^\dagger\left(\tilde z^J_{0} - J^{T+T'} z_{0}^J - \sum_{i=1}^{T+\delta-1} J^{T+T'-i} B_J u'_{i-1}\right).
    \end{gathered}
\end{equation}
Let $z' \in(\mathbb R^{n_x})^{\mathbb N}$ with $z'^J_0 = z^J_0$ evolve according to the state evolution of~\eqref{qsys}, cf.\ \eqref{evolution}, subject to the input $u'$. Moreover, choose $y'\in(\mathbb R^{n_y})^{\mathbb N}$ with $y'_k = C_J z'^J_k + C_N z'^N_k + D u'_k$. Then $b'=(Pz',u',y')\in\mathfrak B_\infty$ and \eqref{intermediate} holds by construction.
\end{proof}

\begin{remark}[Controllability of the manifest behavior]~\\
    Regarding the manifest behavior~$\mathfrak B_\infty^\text{i/o} $ assuming R-con\-troll\-ability and R-observability of the underlying system is no limitation. Due to regularity the transfer function $G(z)=C(zE-A)^{-1}B+ D$ can by decomposed into $G(z)=W(z)+\widetilde W(z)+D$ with a strictly proper rational matrix $W(z)$ related to the dynamical part and a polynomial matrix $\widetilde W(z)$ related to the algebraic part, see \cite[Theorem~2-6.2]{Dai89}. The rational matrix $W(z)$ gives rise to a minimal representation $W(z)= C_J(zI-J)^{-1} B_J$ with matrices $J$, $B_J$, $C_J$ satisfying \eqref{kalmanctrl} and \eqref{kalmanobsv}. For the polynomial matrix $\widetilde W(z)$ there exists matrices $N$, $B_N$, $C_N$, where $N$ is nilpotent, such that $\widetilde W(z)=C_N(zN-I)^{-1}B_N$, see \cite[Lemma~2-6.2]{Dai89}. This implies the existence of a state-space representation corresponding to the manifest behavior in terms of an R-controllable and R-observable system.
\end{remark}

\subsection{The Fundamental Lemma for LTI Descriptor Systems}

Next, we briefly recap the fundamental lemma for input-output trajectories of the discrete-time descriptor systems, cf.\ its counterpart Lemma~\ref{lem:FL_des} for the explicit case. We note that a corresponding statement including the state variables can be derived by imposing stronger assumptions on the output, i.e. $C=I$. 
\begin{lemma}[Fundamental lemma for $\operatorname{rk}(E)<n_x$]
    \label{lem:FL_dae}
    Suppose that system~\eqref{sys_dae} is R-controllable. Let $(u,y)$ be a trajectory of length $T-\delta$ such that $u$ is persistently exciting of order $L+n_J+\delta-1$. Then $(\tilde u, \tilde y)$ is a trajectory of~\eqref{sys_dae} of length $L$ if and only if there exists $g\in\mathbb R^{T-L-\delta+2}$ such that
    \begin{equation}
    \label{FL_des_eq}
        \begin{bmatrix}
            \mathbf{\tilde u}_{[0,L-1]}\\
            \mathbf{\tilde y}_{[0,L-1]}
        \end{bmatrix} = 
        \begin{bmatrix}
            \mathcal H_{L} (\mathbf{u}_{[0,T-\delta]})\\
            \mathcal H_{L} (\mathbf{y}_{[0,T-\delta]})
        \end{bmatrix} g.
    \end{equation}
    Moreover, the choice of $g$ in the representation \eqref{FL_des_eq} determines the future values of the input $\tilde u$. More precisely, for $(\tilde u,\tilde y)$ such that \eqref{FL_des_eq} holds one has
    \begin{equation*}\mathbf{\tilde u}_{[0,L+\delta-2]} = \mathcal H_{L+\delta-1}(\mathbf{u}_{[0,T-1]})g.
    \end{equation*}
\end{lemma}
\begin{remark}[The structured nilpotency in the lemma]~\\
    In contrast to Lemma~\ref{lem:FL_des}, the version of the fundamental lemma given by \cite{SchmitzFaulwasserWorthmann22} is formulated w.r.t.\ the nilpotency index. The proof by \cite{SchmitzFaulwasserWorthmann22} relies on the state evolution \eqref{evolution} for the quasi-Weierstraß form~\eqref{qsys} and can be directly adapted to the structured nilpotency index. 

    The second statement in Lemma~\ref{lem:FL_des}, which reflects the non-causality of the system, is not explicitly given by~\cite{SchmitzFaulwasserWorthmann22}. It can be seen, however, in the proof of Lemma~2 by~\cite{SchmitzFaulwasserWorthmann22} that the input signals are artificially trimmed to the length $L$. Therefore, a slight modification of the block matrices $\mathcal U$ and $\mathcal V$ in the proof of Lemma~2 by \cite{SchmitzFaulwasserWorthmann22} yields the above assertion.
    Summing up, modulo minor changes, the proof of Lemma~\ref{lem:FL_des} resembles the one given by~\cite{SchmitzFaulwasserWorthmann22}.
\end{remark}

The next lemma generalizes the insights of~\cite{moonen1989} from the explicit LTI to the descriptor setting. It yields a rank condition for the stacked Hankel matrix in~\eqref{FL_des_eq}.
\begin{lemma}[Rank of the Hankel matrices] \label{lem:dim_zJ}
    Suppose that the system~\eqref{sys_dae} is R-controllable and R-observable. Let $(u,y)$ be a trajectory of length $T-\delta$ of~\eqref{sys_dae}  such that $u$ is persistently exciting at least of order $L+n_J+\delta-1$ with $L\geq n_J$. Then
    \begin{equation}
        \label{rank}
        \operatorname{rk}\left(\begin{bmatrix}
            \mathcal H_{L} (\mathbf{u}_{[0,T-\delta]})\\
            \mathcal H_{L} (\mathbf{y}_{[0,T-\delta]})
        \end{bmatrix}\right) = n_u L + n_J.
    \end{equation}
\end{lemma}
\begin{proof}
    Without loss of generality we assume that the system is in quasi-Weierstraß form \eqref{qsys}. Let $(z,u,y)$ be a state-input-output trajectory of~\eqref{sys_dae}, where $z$ is decomposed into $z^J$ and $z^N$ as in \eqref{transform}. Then the matrix
    \begin{equation*}
        \begin{bmatrix}
            \mathcal H_{z^J}\\ \mathcal H_u
        \end{bmatrix}\doteq\begin{bmatrix}
            \mathcal H_1(\mathbf{z}^J_{[0,T-L-\delta+1]})\\
            \mathcal H_{L+\delta-1}(\mathbf{u}_{[0,T-1]})
        \end{bmatrix}\in\mathbb R^{(n_J + (L+\delta-1)n_u)\times (T-L-\delta+2)}
    \end{equation*}
    has full row rank, cf. \cite[Proof of Lemma~2]{SchmitzFaulwasserWorthmann22}. Let
    \begin{align*}
        \mathcal S = \begin{bmatrix}C_J\\ C_J J\\ \vdots\\ C_J J^{L-1}\end{bmatrix},\quad \mathcal T=\begin{bmatrix}
            0  & \dots & \dots & 0\\
            C_J B_J & \ddots& & \vdots\\
            \vdots & \ddots &\ddots &\vdots\\
            C_J J^{L-2} B_J & \dots & C_J B_J & 0
        \end{bmatrix},\\
        \mathcal R = \begin{bmatrix}
            C_NB_N  & \dots & C_NN^{\delta-1} B_N \\
             & \ddots & & \ddots & \\
             &  & C_NB_N & \dots & C_NN^{\delta-1} B_N
        \end{bmatrix},
    \end{align*}
    where $\mathcal S\in\mathbb R^{n_y L \times n_J}$, $\mathcal T\in\mathbb R^{n_yL\times n_u L}$ and $\mathcal R\in\mathbb R^{n_N L\times n_u(L+\delta-1)}$. Then
    \begin{equation*}
            \mathcal H_L(\mathbf y_{[0,T-\delta]}) =  \mathcal S \mathcal H_{z^J}
            +\left(\begin{bmatrix}
                (\mathcal T+I_{L}\otimes D) & 0_{n_y L \times n_u (\delta-1)}
            \end{bmatrix} +\mathcal R\right) \mathcal H_u.
    \end{equation*}
    Choose a matrix $\mathcal H_u^\bot$ such that the columns of $\mathcal H_u^\bot$ span the kernel of $\mathcal H_u$. Then $\mathcal H_L(\mathbf y_{[0,T-\delta]}) \mathcal H_u^\bot = \mathcal S \mathcal H_{z^J}\mathcal H_u^\bot$.  R-observability together with $L\geq n_J$ implies that $\mathcal S$ has full column rank. Together with the full row rank of the matrix $\begin{bmatrix}
        \mathcal H_{z^J}^\top & \mathcal H_u^\top
    \end{bmatrix}{}^\top$ we have
    \begin{equation}
    \nonumber
        \operatorname{rk}\left(\mathcal H_L(\mathbf y_{[0,T-\delta]}) \mathcal H_u^\bot\right) = \operatorname{rk}\left(\mathcal H_{z^J}\mathcal H_u^\bot\right) = \operatorname{rk}\left(\mathcal H_{z^J}\right)=n_J.
    \end{equation}
    This shows the assertion.
\end{proof}

Moreover, comparison of Lemma~\ref{lem:FL_dae} and Lemma~\ref{lem:FL_des} shows that the explicit LTI case requires more data than the descriptor setting.  In the case of unknown $\delta$ and  $n_J$ Lemma~\ref{lem:dim_zJ} provides a promising approach to estimate these quantities by testing the rank condition \eqref{rank} for various input-output trajectories.

\subsection{The Stochastic Descriptor Fundamental Lemma}

In this subsection, we formulate a corresponding stochastic version of the fundamental lemma. We consider the stochastic system representation
\begin{subequations}\label{eq:RVdynamics_descr}
    \begin{align}
    	E X\inst{k+1} &= AX\inst{k} +\widetilde BV\inst{k}\\
        \label{eq:RVdynamicsb_descr}
   		Y\inst{k} &= C X\inst{k}+\widetilde D V\inst{k}
    \end{align}
\end{subequations}
where the state, exogenous input, and output signals are modelled as stochastic processes, that is $X_k\in L^2((\Omega,\mathcal F,\prob),\mathbb R^{n_x})$, $Y_k\in L^2((\Omega,\mathcal F,\prob),\mathbb R^{n_y})$, and $V_k\in L^2((\Omega,\mathcal F,\prob),\mathbb R^{n_v})$. 
Similar to the deterministic setting (cf.\ \eqref{consistentInit}), invoking the assumed regularity of the pencil $\lambda E-A$, system~\eqref{eq:RVdynamics} can be transformed into quasi-Weierstraß form. Further, the random variable $X_0$ has to be drawn from the set
\begin{equation}
	\label{consistentInitStohastic}
	\mathcal X^0_s \doteq \left\{P\begin{bmatrix}
		\Xi^J\\ \Xi^N
	\end{bmatrix} \,\middle |\, \begin{gathered}\Xi^J\in L^2((\Omega,\mathcal F,\prob),\mathbb R^{n_J}),\\
	\Xi^N = \sum_{i=0}^{\delta-1} N^i \widetilde{B}_N V_{i} \text{ for some }\\V_0,\dots,V_{\delta-1}\in L^2((\Omega,\mathcal F,\prob),\mathbb R^{n_v})\end{gathered}\right\}.
\end{equation}
to be consistent with the system dynamics~\eqref{eq:RVdynamics}. 

As before in Section~\ref{sec:model_Stochastic} the variable $V$ contains the manipulated control inputs~$U$ as well as the exogenous process disturbances~$W$, cf.~\eqref{eq:inputsplit}.

The non-causality (for structured nilpotency index $\delta\geq 2$) of the regular descriptor system~\eqref{eq:RVdynamics_descr} implies that the solution $X=(X_n)_{n\in\mathbb N}$ is not a Markov process with respect to its natural filtration $(\sigma(X_0,\dots, X_k))_{k\in\mathbb N}$. However, by augmenting the state variable with future control inputs (cf.\ Remark~\ref{rem_causal}) one can recover the Markov property.
    
To avoid further cumbersome technicalities, we consider the case where the exogenous disturbance $W$ at future time instances does not influence the present state. More precisely, given system~\eqref{eq:RVdynamics_descr} with split input \eqref{eq:inputsplit}, the matrix $N$ in the quasi-Weierstraß form  with $SF = \begin{bmatrix}
    F_J^\top & F_N^\top
\end{bmatrix}{}^\top$ annihilates $F_N$, i.e.\ $NF_N=0$, cf. \eqref{qweier}. Further, suppose that the noise $W=(W_k)_{k\in\mathbb N}$ is a sequence of independent square-integrable random vectors. In view of the noncausality, cf.~\eqref{evolution}, we think of a control law which assigns the new input action $U_{k+\delta}$ on the basis of the current value of the state $X_k$, i.e.\ $U_{k+\delta}=K_k(X_{k})$ with some measurable map $K_k$. We consider the stochastic processes $\Delta$ and $\Gamma$ given by
\begin{equation*}
        \Delta_k =
        \begin{bmatrix}X_k\\ \mathbf U_{[k,k+\delta-1]}\\ W_k\end{bmatrix},\quad \Gamma_k = 
        \begin{bmatrix}
            \mathbf Y_{[k,k+n_J-1]}\\ \mathbf U_{[k,k+n_J+\delta-2]}\\ \mathbf W_{[k,k+n_J-1]}
        \end{bmatrix}
\end{equation*}
together with their natural filtrations $(\mathcal F_k)_{k\in\mathbb N}$ and $(\mathcal G_k)_{k\in\mathbb N}$, respectively, with $\sigma$-algebras given by $\mathcal F_k=\sigma(\Delta_0,\dots,\Delta_k)$ and $\mathcal G_k = \sigma(\Gamma_0,\dots,\Gamma_k)$. The next proposition discusses the Markov property of both stochastic processes.

\begin{proposition}[Markov property]
        \label{prop:markov_descr}
        Suppose that $NF_N=0$. Further, let {$\Delta_0, W_1,W_2,\dots$ } be mutually independent and let the system be governed by the control law $U_{k+\delta}= K_k(X_k)$ for $k\in\mathbb N$ with a measurable function $K_k$. Then
        \begin{subequations}
        \begin{enumerate}
            \item $W_j$ for $j\geq k+1$ is independent of $\mathcal F_k$ as well as $\sigma(\Delta_k)$ and
            \item $\Delta$ is a Markov process, i.e.\
                \label{eq:markov_descr}
                \begin{equation}
                    \label{eq:markov1_descr}
            \prob[\Delta_{k+1} \in\mathcal A\,|\, \mathcal F_k] = \prob[\Delta_{k+1}\in\mathcal A\,|\, \sigma(\Delta_k)]
                \end{equation}
            for every Borel set $\mathcal A\subset \mathbb R^{n_x+{n_u\delta+n_w}}$.
        \end{enumerate}
            If in addition system~\eqref{eq:RVdynamics_descr} is R-observable, then
        \begin{enumerate}[resume]
            \item $W_j$ for $j\geq k+n_J$ is independent of $\mathcal G_k$ as well as $\sigma(\Gamma_k)$ and
            \item  the stochastic process $\Gamma$ satisfies the Markov property
            \begin{equation}
                \label{eq:markov2_descr}
            \prob[\Gamma_{k+1} \in\mathcal B\,|\, \mathcal G_k] = \prob[\Delta_{k+1}\in\mathcal B\,|\, \sigma(\Delta_k)]
            \end{equation}
            for every Borel set $\mathcal B\subset\mathbb R^{n_y n_J+{n_u(n_J+\delta-1) + n_w n_J}}$.
        \end{enumerate}
    \end{subequations}
    \end{proposition}
    \begin{proof}
        Part~(i). According to the state evolution~\eqref{evolution} the state $X_{k+1}$ can be expressed linearly by $X_k$, $U_k,\dots, U_{k+\delta}$ in combination with $W_k,W_{k+1}$. Together with the control law for $U_{k+\delta}=K_k(X_k)$ we find the functional description
        \begin{equation}
            \label{eq:recursion_descr}
            X_{k+1}= \tilde f(\Delta_k,W_{k+1}), \quad \Delta_{k+1}=f(\Delta_k,W_{k+1})
        \end{equation}
        with linear functions $f$ and $\tilde f$. This shows
        \begin{equation*}\sigma(\Delta_{k+1})\subset\mathcal F_{k+1}\subset\sigma(\Delta_0,\dots,\Delta_k, W_{k+1})\end{equation*} for all $k\in\mathbb N$. By assumption $W_j$ for $j\geq 1$ is independent of $\mathcal F_0=\sigma(\Delta_0)=\sigma(X_0,U_0,\dots, U_{\delta-1}, W_0)$. Given $k\in\mathbb N$, we see by induction that $W_j$ for $j\geq k+2$ is independent of $\sigma(\Delta_0,\dots,\Delta_{k},W_{k+1})$ and therefore, is also independent of $\mathcal F_{k+1}$ as well as $\sigma(\Delta_k)$. 
        
        Part~(ii). The conditional probability can by written in terms of the conditional expectation operator employing the characteristic function $\mathds 1_{\mathcal A}$ of the set $\mathcal A$. The $\mathcal F_k$-measurability of $\Delta_{k}$ and the fact that $W_{k+1}$ is independent of $\mathcal F_k$ imply
        \begin{align*}
            \prob[\Delta_{k+1}\in\mathcal A\,|\,\mathcal F_k] &= \mean[\mathds 1_{\mathcal A}(\Delta_{k+1})\,|\, \mathcal F_k]\\ &= \mean[\mathds 1_{\mathcal A}(f(\Delta_k,W_{k+1}))\,|\, \mathcal F_k] = g(\Delta_k)
        \end{align*}
        with $g$ defined by $g(d)=\mean[\mathds 1_{\mathcal A}(f(d,W_{k+1}))]$ for $d\in\mathbb R^{n_x+n_v\delta}$, cf.\ \cite[Problem~IV-34]{malliavin1995}. Similarly,
        \begin{equation*}
            \prob[\Delta_{k+1}\in\mathcal A\,|\,\sigma(\Delta_k)]
            = \mean[\mathds 1_{\mathcal A}(f(\Delta_k,W_{k+1}))\,|\, \sigma(\Delta_k)] = g(\Delta_k).
        \end{equation*}
        This proves~\eqref{eq:markov1_descr}.

        Part~(iii). By \eqref{eq:RVdynamicsb_descr} the random variable $\Gamma_k$ is given in a linear way by $\Delta_k,\dots,\Delta_{k+n_J-1}$. This yields $\sigma(\Gamma_k)\subset \mathcal G_k\subset \mathcal F_{k+n_J-1}$. This together with (i) implies (iii).

        Part~(iv). The second recursion in \eqref{eq:recursion_descr} yields
        \begin{equation*}
            \Gamma_{k+1}  = \tilde h (\Delta_k, W_{k+1},\dots, W_{k+n_J})
        \end{equation*}
        for all $k\in\mathbb N$, where $\tilde h$ is a linear map. By R-observability of system~\eqref{eq:RVdynamics_descr} there exists a linear map $\hat h$ such that $\Delta_k = \hat h (\Gamma_k)$ for all $k\in\mathbb N$, cf.\ Lemma~\ref{state_align}. Therefore, there is a certain linear map $h$ such that
        \begin{equation*}
            \Gamma_{k+1} = h(\Gamma_k, W_{k+n_J}).
        \end{equation*}
        Assertion~\eqref{eq:markov2_descr} follows with a similar argument as in (ii).\qedhere
    \end{proof}
    \begin{corollary}[Non-anticipativity]
        \label{cor:nonanticipativity_descr}
        Let the assumptions of Proposition~\ref{prop:markov_descr} hold. Then $U_k$ and $W_j$ are independent for all {$k,j\in\mathbb N$
        \begin{itemize}
            \item[(i)] with $0\leq k\leq \delta-1$ and $j\geq 1$;
            \item[(ii)] with $k\geq \delta$ and $j\geq k-\delta+1$.
        \end{itemize}}
    \end{corollary}
    \begin{proof}
        For (i) this follows by assumptions. We show (ii). The control law for $k\geq \delta$ implies $U_k = K_{k-\delta}(X_{k-\delta})$, where the random variable $X_{k-\delta}$ is $\mathcal F_{k-\delta}$-measurable. Therefore, the assertion follows with Proposition~\ref{prop:markov_descr}~(i).
    \end{proof}

\begin{remark}[Completeness and controllability of $\mathfrak C_\infty$ and $\mathfrak S_\infty$]\label{rem:behavior_dae}
    We have shown completeness of the behaviors~$\mathfrak C_\infty$ and~$\mathfrak S_\infty$ in Lemmata~\ref{stochbehav_complete},~\ref{stochbehav_ctrl} and~\ref{lem:compl_ctrl_stochBehav}. However, if the realization behavior $\mathfrak B_\infty$ is controllable with delay $T'=\delta-1+n_J$, then $\mathfrak C_\infty$ and $\mathfrak S_\infty$ are controllable with delay $T'=\delta-1+n_J$ as a consequence of Lemma~\ref{ctrlRbehav}.
\end{remark}

Employing the (deterministic) fundamental lemma for descriptor systems, Lemma~\ref{lem:FL_dae}, the results from Section~\ref{sec:StochFundLem} can be formulated for descriptor systems. 
It should be mentioned that in contrast to the explicit LTI case a lower order of persistent excitation is needed and that the data demand in the Hankel matrix is reduced. Thus, the subsequent results are the analogs of Lemma~\ref{lem:colEqui} and Lemma~\ref{lem:FL_stoch} and they are stated for the sake of completeness.

Notice that, similar to the explicit case, the dynamics in random variables \eqref{eq:RVdynamics_descr} have counterparts in terms of corresponding expansion coefficients and in terms of realizations, cf. \eqref{eq:RVdynamics}, \eqref{eq:Realdynamics}, and \eqref{eq:Coeffdynamics} in Section \ref{sec:model_Stochastic}.
\begin{lemma}[Column-space equivalence]\label{lem:colEqui_descr} 
    Let the   descriptor system \eqref{eq:RVdynamics_descr} be R-controllable and regular with dimension~$n_J$ of the dynamical part as well as structured nilpotency index~$\delta$. For $T \in \mathbb{Z}_+$, let $(V,Y)$ be random-variable input-output trajectories of \eqref{eq:RVdynamics_descr} with corresponding expansion coefficients $(\mathsf v, \mathsf y)$. Let $(\hat v,\hat y)$ be realization trajectories also corresponding to  \eqref{eq:RVdynamics_descr}. Further, assume that $\hat v$ and the coefficients~$\mathsf v^i$, $i \in \mathbb N$, are persistently exciting of order $L+n_J+\delta -1$.
\begin{itemize}
\item[(i)] Then, for all $i \in \mathbb N$
\begin{subequations}
\begin{equation} \label{eq:colEqui_descr}
\text{$\operatorname{colsp}
    \begin{bmatrix}
        \mathcal H_{L}(\mathsf v^i_{[0,T-\delta]})\\
        \mathcal H_L(\mathsf y^i_{[0,T-\delta]})
    \end{bmatrix}
    = \operatorname{colsp}
    \begin{bmatrix}
        \mathcal H_{L}(\hat{\mathbf{v}}_{[0,T-\delta]})\\
        \mathcal H_L(\hat{\mathbf{y}}_{[0,T-\delta]})
    \end{bmatrix}.$}
\end{equation}
\item[(ii)] Moreover, for all $g \in \R^{T-L-\delta+2}$, there exists a function $G\in L^2(\Omega, \mathbb R^{T-L-\delta+2})$ such that
\begin{equation}\label{eq:colEquiRV_descr}
    \begin{bmatrix}
        \mathcal H_{L}(\mathbf V_{[0,T-\delta]})\\
        \mathcal H_L(\mathbf Y_{[0,T-\delta]})
    \end{bmatrix} g = \begin{bmatrix}
        \mathcal H_{L}(\hat{\mathbf{v}}_{[0,T-\delta]})\\
        \mathcal H_L(\hat{\mathbf{y}}_{[0,T-\delta]})
    \end{bmatrix} G. 
\end{equation}
\end{subequations}
\end{itemize}
\end{lemma}
A column-space inclusion result, i.e. a counterpart to Corollary~\ref{cor:inclusion}, can be obtained without further difficulties and is hence not detailed. Finally, the fundamental lemma for stochastic descriptor systems extends and combines the developments of~\cite{Pan21s} and~\cite{SchmitzFaulwasserWorthmann22}.
\begin{lemma}[Stochastic descriptor fundamental lemma]~\\
    \label{lem:FL_stoch_descr}
    Let the   descriptor system \eqref{eq:RVdynamics_descr} be R-controllable and regular with dimension~$n_J$ of the dynamical part as well as structured nilpotency index~$\delta$. For $T \in \mathbb{Z}_+$, let $(\tilde V, \tilde Y)$ be random-variable input-output trajectories of \eqref{eq:RVdynamics_descr} with corresponding expansion coefficients $(\tilde{\mathsf{v}},\tilde{\mathsf{y}})$. Let $(v,y)$ be realization data corresponding to \eqref{eq:RVdynamics_descr} such that $v$ is persistently exciting of order $L+n_J+\delta-1$. Then, the following statements hold:
     \begin{subequations}
    \begin{enumerate}
        \item $(\tilde{\mathsf v},\tilde{\mathsf y})$ are expansion coefficient trajectories corresponding to \eqref{eq:RVdynamics_descr}   if and only if there is $\mathsf g\in \ell^2(\mathbb R^{T-L-\delta+2})$ such that
        \begin{equation}
            \label{eq:FL_coef_descr}
            \begin{bmatrix}
                \mathsf{\tilde v}^i_{[0,L-1]}\\ \mathsf{\tilde y}^i_{[0,L-1]}
            \end{bmatrix} = \begin{bmatrix}
                \mathcal H_{L} (\mathbf{v}_{[0,T-\delta]}) \\ \mathcal H_{L} (\mathbf{y}_{[0,T-\delta]})
            \end{bmatrix}\mathsf g^i
        \end{equation}
        for all $i\in\mathbb N$.
        \item $(\tilde{V},\tilde{Y})$ are random variable trajectories of \eqref{eq:RVdynamics_descr} if and only if there is $G\in L^2(\Omega,\mathbb R^{T-L-\delta+2})$ such that
        \begin{equation}
            \label{eq:FL_stoch_eq_descr}
            \begin{bmatrix}
                \mathbf{\tilde V}_{[0,L-1]}\\ \mathbf{\tilde Y}_{[0,L-1]}
            \end{bmatrix} = \begin{bmatrix}
                \mathcal H_{L} (\mathbf{v}_{[0,T-\delta]}) \\ \mathcal H_{L} (\mathbf{y}_{[0,T-\delta]})
            \end{bmatrix} G.
        \end{equation}
    \end{enumerate}
     \end{subequations}
\end{lemma}

\subsection{Stochastic Optimal Control for Descriptor Systems}

Next, we consider the stochastic system \eqref{eq:RVdynamics_descr} with input partition~\eqref{eq:inputsplit}. We model the input $U\inst{k}$ as a stochastic process adapted to the filtration $(\mathcal G_k)_{k\in\mathbb N}$ as in Proposition~\ref{prop:markov_descr} and according to the causality condition w.r.t.\ the disturbance stated in Corollary~\ref{cor:nonanticipativity_descr}. 
That is, for the sake of avoiding tedious technicalities, we suppose that $NF_N = 0$ holds in the quasi-Weierstraß form, cf.\ \eqref{qsys}.

As a preliminary step to derive the counterpart of OCP~\eqref{eq:OCPPCE} for stochastic descriptor systems, we state the Hankel matrix description analogously to OCP~\eqref{eq:OCPdata}:
\begin{subequations}\label{eq:OCPdata_descr}
\begin{align}
    &\operatorname*{minimize}_{U,Y, G }   \,
    \sum_{k=\delta+n_J-1}^{N+\delta+n_J-2} \mean[Y_k^\top QY_k + U^\top_k RU_k]\\
     &\qquad \text{ subject to }\qquad \nonumber
        \end{align}
         \begin{align}
  & \qquad\begin{bmatrix}
         \mathbf{Y}_{[0,N+\delta+n_J-2]}\\
         \mathbf{U}_{[0,N+\delta+n_J-2]}\\
        \widehat{\mathbf{W}}_{[0,N+\delta+n_J-2]}
    \end{bmatrix}
    =
        \begin{bmatrix}
            \mathcal H_{N+\delta+n_J-1}\left(\mathbf{y}^\text{d}_{[0,T-\delta]}\right)\\
            \mathcal H_{N+\delta+n_J-1}\left(\mathbf{u}^\text{d}_{[0,T-\delta]}\right)\\
            \mathcal H_{N+\delta+n_J-1}\left(\mathbf{w}^\text{d}_{[0,T-\delta]}\right)
        \end{bmatrix}
        G \label{eq:OCPdata_hankel_descr},
\\
    & \qquad  \begin{bmatrix}
             \mathbf{ Y}_{[0,n_J-1]}\\
             \mathbf{ U}_{[0,\delta+n_J-1]}    
             \end{bmatrix} 
             = \begin{bmatrix}
             \mathbf{\widehat Y}_{[0,n_J-1]}\\
             \mathbf{\widehat U}_{[0,\delta+n_J-1]} 
         \end{bmatrix}.   \label{eq:OCPdata_consist_cond_descr}
     \end{align}
\end{subequations}
The \emph{consistency condition}~\eqref{eq:OCPdata_consist_cond_descr} together with  R-observability of the underlying system~\eqref{eq:RVdynamics_descr} guarantees uniqueness of the latent state trajectory for sufficiently large horizon~$N$. In contrast to the explicit case, the OCP requires $n_J+\delta-1$ consistent random-variable input and $n_J$ output pairs, cf.\ Lemma~\ref{state_align}, which explains the required length $N+\delta+n_J-1$ in constraint~\eqref{eq:OCPdata_hankel_descr}.

Next, we adjust the sufficient conditions for exact uncertainty propagation given in Lemma~\ref{lem:no_truncation_error}. Subsequently $(\phi^i)_{i\in\mathbb N}$ denotes a PCE basis. 

\begin{lemma}[Exact uncertainty propagation via expansions]\label{lem:no_truncation_error_descr}
    Consider the stochastic descriptor LTI system~\eqref{eq:RVdynamics_descr} and suppose that $\widehat{W}_{k}$ for $k \in \I_{[0,N+\delta+n_J-2]}$, 
    $\hat{Y}_{k}$ for $k \in \I_{[0,n_J-1]}$, and 
    $\widehat{U}_{k}$ for $k \in \I_{[0,\delta+n_J-1]}$ admit exact PCEs with finite dimensions~$p_w$ and~$p_{\text{ini}}$, i.e., $\widehat{W}_{k}= \sum_{i=0}^{p_w-1}\hat{\pce{w}}^i \phi_{k}^i$, $\widehat{Y}_{k} = \sum_{i=0}^{p_{\text{ini}}-1}\hat{\pce{y}}_k^i \phi_{\text{ini}}^i$, and $\widehat{U}_{k}=  \sum_{i=0}^{p_{\text{ini}}-1}\hat{\pce{u}}_k^i \phi_{\text{ini}}^i$, respectively. Assume that $\phi_{\text{ini}}^0=\phi_{k}^0=1$ for all for $k \in \I_{[0,N+\delta+n_J-2]}$.  Then,
\begin{itemize}
    \item[(i)] the optimal solution $({U}^\star ,{Y}^\star,G^\star)$ of OCP~\eqref{eq:OCPdata_descr} with horizon~$N$ admits exact finite-dimensional PCEs with $p$ terms, where $p$ is given by
    \begin{equation*}
    	p = p_{\text{ini}}
    	+( N+ \delta+n_J-1)(p_w-1) \in \mathbb Z_+,
    \end{equation*}
    \item[(ii)] and the finite-dimensional joint basis $(\phi^i)_{i=0}^{p-1}$ reads
    \begin{align*}
        (\phi^i)_{i=0}^{p-1} = (1, \phi_\text{ini}^1,\dots, \phi_\text{ini}^{p_\text{ini}-1}&, \phi_0^1,\dots, \phi_0^{p_w-1},\dots\\
         & \phi_{N+\delta+n_J-2}^1,\dots, \phi_{N+\delta+n_J-2}^{p_w-1}).
    \end{align*}
\end{itemize}
\end{lemma}

\begin{remark}[Filtered stochastic processes with PCE]\label{rem:causalityPCE_descr}~\\
Considering the polynomial basis given in Lemma~\ref{lem:no_truncation_error_descr}, the cau\-sality (non-anti\-pacitivity) of the filtration $(\mathcal G_k)_{k\in\mathbb N}$ used in Proposition~\ref{prop:markov_descr} implies (see Corollary~\ref{cor:nonanticipativity_descr}) that the PCE coefficients of the inputs satisfy 
\begin{align*}
	\pce{u}_{k}^{i} &= 0, \forall i\in \I_{[p_\text{ini}+(k-\delta+1)(p_w-1),p-1]},\, \forall k \in \I_{[\delta,N+\delta+n_J-1]},\\
	\pce{u}_{k}^{i} &= 0, \forall i\in \I_{[p_\text{ini}+{p_w}-1,p-1]},\, \forall k \in \I_{[0,\delta-1]}.
\end{align*}
This condition ensures that $U_0,\dots, U_{\delta-1}$ do not depend on the disturbance $W_j$ for $j\geq 1$ and that $U_k$ for $k\geq \delta$ does not depend on $W_j$ at time instances $j\geq k-\delta+1$.
\end{remark}

Similar to Section~\ref{sec:OCP_explicit}, we arrive at the data-driven reformulation:
\begin{subequations}\label{eq:OCPdataPCE_descr}
\begin{align}
    &\operatorname*{minimize}_{\mathbf{\mathsf u},\mathbf{\mathsf y}, \mathbf{ \mathsf g} }  \,
    \sum_{i=0}^{p-1}\sum_{k=\delta+n_J-1}^{N+\delta+n_J-2} \mean[\phi^i \phi^i]\left((\mathsf y_k^i)^\top Q\mathsf y^i_k + (\mathsf u^i_k)^\top R\mathsf u^i_k\right)  \\
   & \text{ subject to } \forall i \in \I_{[0,p-1]} \quad \nonumber
      \end{align}
         \begin{align}
  &\begin{bmatrix}
         \mathbf{\mathsf y}^i_{[0,N+\delta+n_J-2]}\\
         \mathbf{\mathsf u}^i_{[0,N+\delta+n_J-2]}\\
        \hat{\mathbf{\mathsf w}}^i_{[0,N+\delta+n_J-2]}
    \end{bmatrix}
     =
     \begin{bmatrix}
        \mathcal H_{N+\delta+n_J-1}\left(\mathbf{y}^\text{d}_{[0,T-\delta]}\right)\\
        \mathcal H_{N+\delta+n_J-1}\left(\mathbf{u}^\text{d}_{[0,T-\delta]}\right)\\
        \mathcal H_{N+\delta+n_J-1}\left(\mathbf{w}^\text{d}_{[0,T-\delta]}\right)
    \end{bmatrix}
     \mathsf   g^i \label{eq:OCPdataPCE_hankel_descr}\\
     &  \begin{bmatrix}
        \mathbf{ \mathsf y}^i_{[0, n_J-1]}\\
        \mathbf{ \mathsf u}^i_{[0,\delta+n_J-2]}\\ 
        \end{bmatrix} 
        = \begin{bmatrix}
        \mathbf{\hat {\mathsf y}}^i_{[0,n_J-1]}\\
        \mathbf{\hat {\mathsf u}}^i_{[0,\delta+n_J-2]} \\
    \end{bmatrix}
      \label{eq:OCPdataPCE_consist_cond_descr} \\
&     	\pce{u}_{k}^{i'} = 0, \forall i '\in \I_{[p_\text{ini}+(k-\delta+1)(p_w-1),p-1]},\, \forall k \in \I_{[\delta,N+\delta+n_J-1]},
     \label{eq:OCPPCE_causality_desrc_descr}\\
     &    	\pce{u}_{k}^{i'} = 0, \forall i'\in \I_{[p_\text{ini}+p_w-1,p-1]},\, \forall k \in \I_{[0,\delta-1]}.
     \end{align}
\end{subequations}
Notice that the crucial difference between the explicit LTI case in OCP~\eqref{eq:OCPdataPCE} and OCP~\eqref{eq:OCPdataPCE_descr} are the length of the horizons in the constraints \eqref{eq:OCPdataPCE_hankel_descr}--\eqref{eq:OCPPCE_causality_desrc_descr}. These constraints directly depend on the descriptor structure, specifically on the structured nilpotency index $\delta$ and on the dimension $n_J$.

The results presented in this section show that---while in the behavioral context there is no difference between explicit and descriptor LTI systems---there are distinctive aspects when it comes to data-driven stochastic optimal control. 
\section{Numerical Examples}\label{sec:Examples}
To illustrate our findings, we discuss two examples:  a scalar system subject to disturbances of alternating structure and a fourth-order system subject to Gaussian noise. The scalar example showcases the flexibility to model stochastic disturbances via PCE and the proposed stochastic fundamental lemma. The second example illustrates our findings for descriptor systems. 
In both cases, the implementation is done in \textsf{Matlab R2021b}.

\subsection{Scalar Example with Alternating Disturbance Sequence}
We consider the scalar system
\[
X_{k+1}=2X_k+U_k+W_k
\]
similar to \cite{Ou21}. The stochastic disturbance switches between two distributions, 
i.e.
\[
W_k =\begin{cases}
W_k \sim \mathcal{N}(0,0.1^2)& \quad \text{if } k =2i,\phantom{+1} \quad \hspace{1.5mm} i,k  \in \mbb{N} \\
W_k \sim\mathcal{U}(-0.2,0.2)&\quad \text{if } k = 2i+1, \quad i,k  \in \mbb{N}
\end{cases},
\]
where an i.i.d. Gaussian noise models the disturbance for even time index $k$ and an i.i.d. uniform distribution at each odd value of $k$. 
We suppose that the disturbance distribution is known, sufficient past realization data of $W$ is also available, while its future realizations are not known. The example illustrates the flexibility of PCE to model stochastic disturbances beyond the purely Gaussian setting.

The matrices $Q$ and $R$ are  $Q=R=1$. Additionally, we add the term $2 \sum\nolimits_{k=0}^{N-1} \mean [(U_{k+1}-U_k)^2]$ to smoothen the input sequence. We record $60$ state-input-disturbance measurements to construct the Hankel matrices. We solve with horizon $N=20$ for a randomly sampled initial deterministic condition. The dimension of the PCE basis is $p=N+1=21$.

Figure~\ref{fig:ScalarMoments} depicts the first two moments of the solutions in terms of $X^\star$ and $U^\star$, while Figure~\ref{fig:ScalarPCE} shows the PCE coefficient trajectories. Moreover, we sample a total of $20$ sequences of noise realizations and compute all the corresponding state and input trajectories, see Figure~\ref{fig:ScalarTraj}. As one can see, the state and the input trajectories in terms of expectation converge to $0$. Interestingly, the variance of the state and the input does not converge to $0$, cf. Figure~\ref{fig:ScalarMoments}.
Figure~\ref{fig:ScalarTraj} shows realizations for $20$ distinct disturbance sequences. Observe that the increase in variance towards the end of the horizon is also visible for the realizations. 
In terms of realizations and moments, this is reminiscent of a turnpike property, cf.~\cite{Ou21,tudo:faulwasser22a}. A detailed discussion of the phenomenon in the stochastic setting is, however, beyond the scope of the current paper.

\begin{figure}[!t]
\centering
\subfloat[Mean and variance obtained via PCE coefficients.]{
    \label{fig:ScalarMoments}
    \includegraphics[width=1\columnwidth]{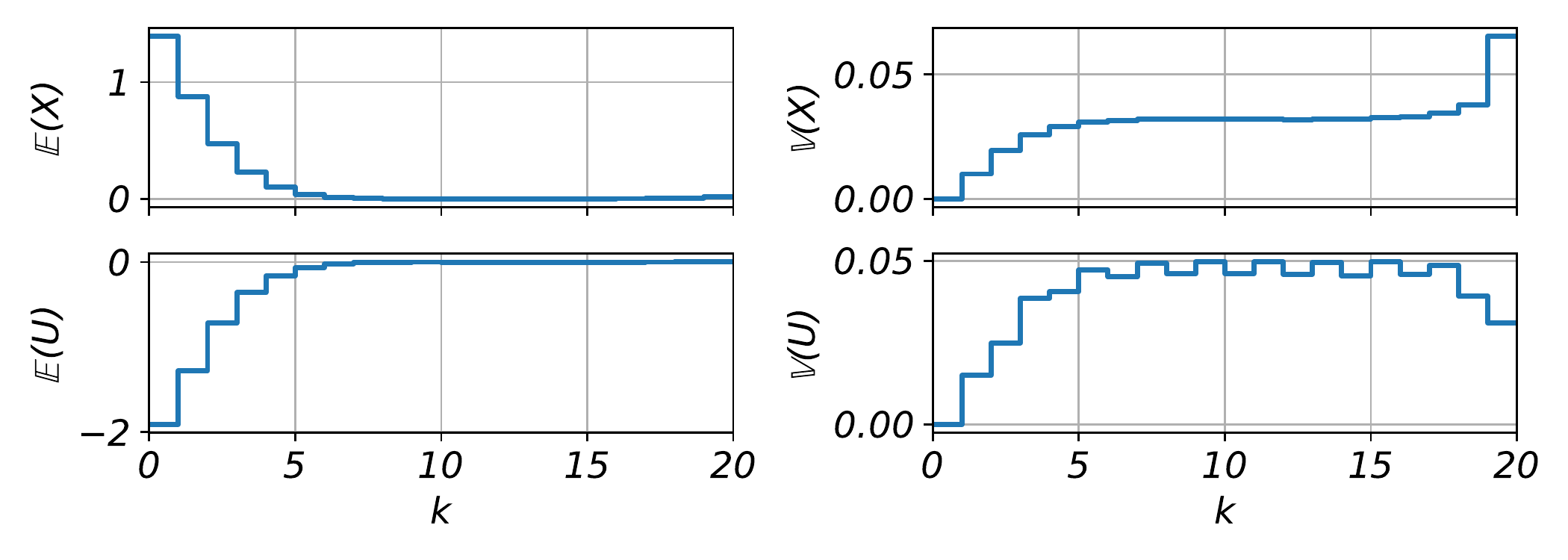}
}

\subfloat[PCE coefficient trajectories.]{
    \label{fig:ScalarPCE}
     \includegraphics[width=1\columnwidth]{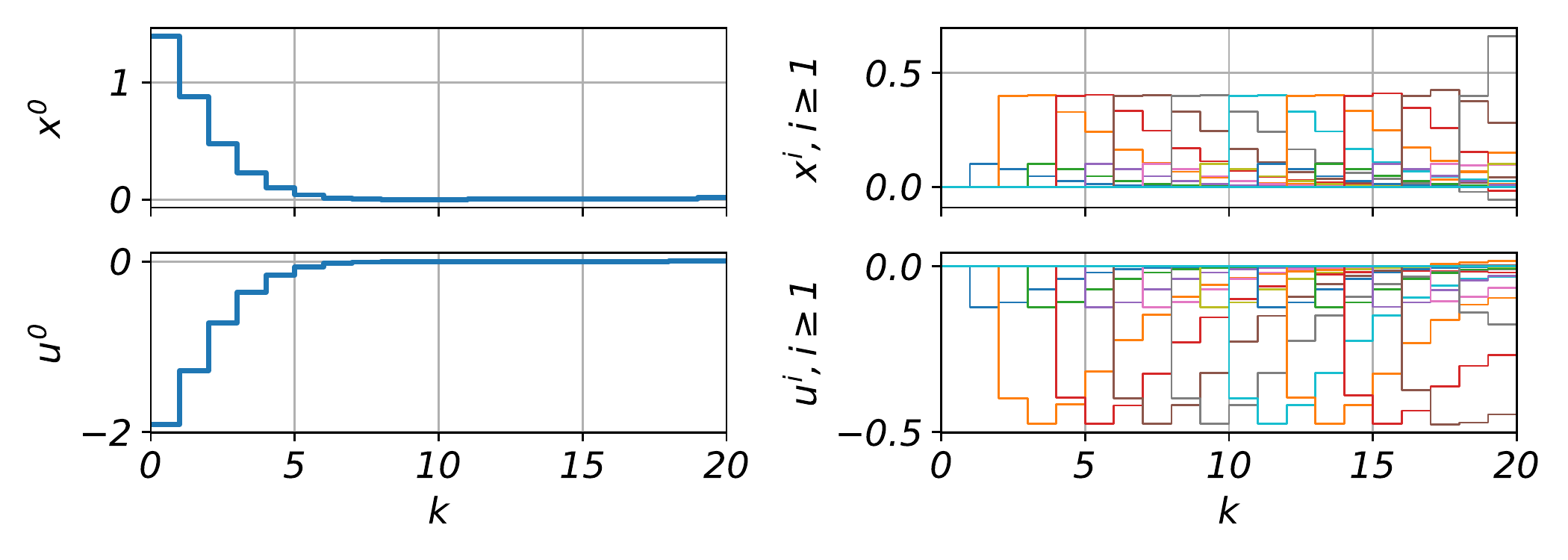}
}

\subfloat[$20$ different realizations trajectories.]{
    \label{fig:ScalarTraj}
     \includegraphics[width=1\columnwidth]{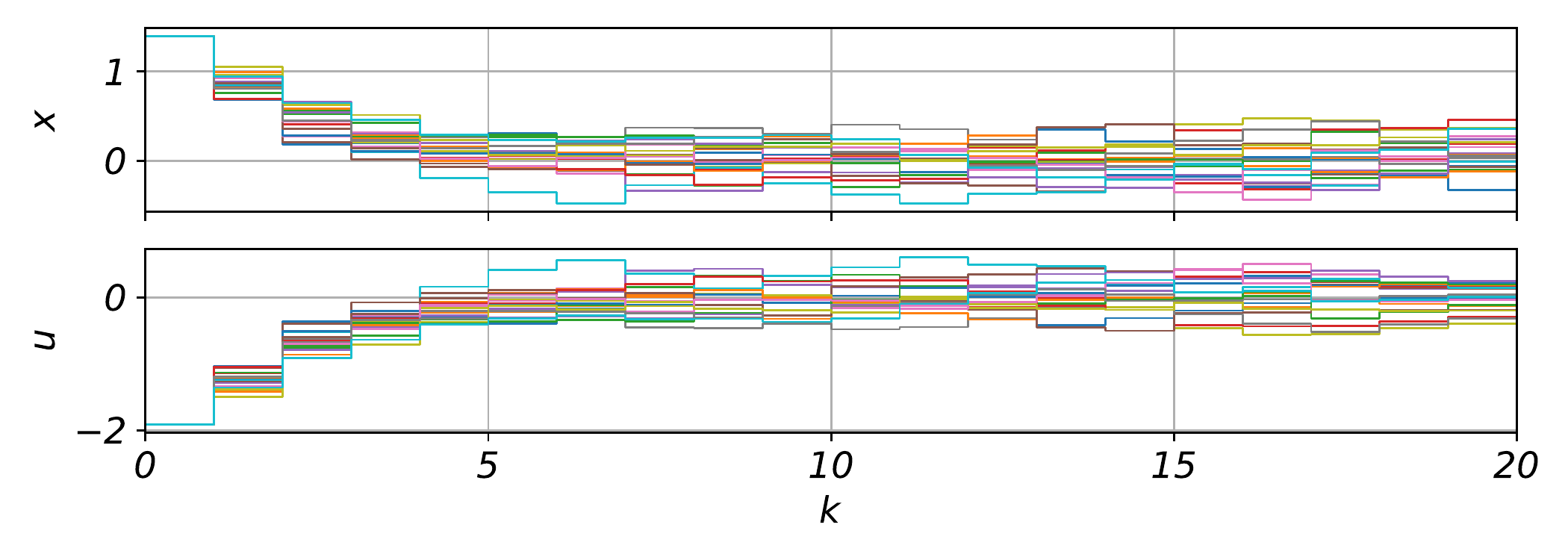}
}
\caption{Results for the scalar example.}
\end{figure}

\subsection{Descriptor Example}
We consider a stochastic extension of the fourth-order linear descriptor system considered in \cite{SchmitzFaulwasserWorthmann22}. The system matrices read
\begin{gather*}
    E = \begin{bmatrix} 0 & 0 & 1 & 0 \\ 
                        1 & 2 & 0 & 2 \\ 
                        2 & 3 & 1 & 3 \\ 
                        1 & 2 & 0 & 2 \end{bmatrix},~
    A = \begin{bmatrix} \phantom{-}1 & 1 & 0 &  2\\ 
                        \phantom{-}0 & 2 & 1 & 1 \\ 
                        \phantom{-}1 & 4 & 2 & 3 \\ 
                        -1 & 1 & 1 & 0 \end{bmatrix},~
    B = \begin{bmatrix} -1 \\ \phantom{-}2 \\ \phantom{-}2 \\ \phantom{-}3 \end{bmatrix},~
    F = \begin{bmatrix} 1\\ 2 \\ 3 \\ 2 \end{bmatrix},\\
    C = \begin{bmatrix} 1 & 2 & 1 & 2 \\ 
                        0 & 1 & 0 & 1 \end{bmatrix},~
    D = H = 0_{2\times 1}.
\end{gather*}
The system can be transformed into quasi-Weierstra{\ss} form with $n_{J}=\delta=2$ via the matrices
\[
    P = \begin{bmatrix} \phantom{-}0 & -1 & \phantom{-}0 & \phantom{-}1 \\ 
                        -1 & \phantom{-}0 & \phantom{-}1 & \phantom{-}1 \\ 
                        \phantom{-}1 & \phantom{-}0 & \phantom{-}0 & -1 \\ 
                        \phantom{-}1 & \phantom{-}1 & -1 & -1 \end{bmatrix},~
    S = \begin{bmatrix} \phantom{-}0 & -1 & \phantom{-}1 & \phantom{-}0 \\
            \phantom{-}1 & \phantom{-}2 & -1 & \phantom{-}0 \\ 
            -1 & -1 & \phantom{-}1 & 0 \\ 
            \phantom{-}0 & \phantom{-}1 & \phantom{-}0 & -1 \end{bmatrix}.
\]
R-controllability and R-observability are easily verified via Remark~\ref{rem:controllability/observability_weierstrass}. Note that in this example we have $F_N=0$, which is a special case of $NF_N=0$ and hence tightens the causality condition of the input. Thus, causality with respect to the disturbance $W$ is easily obtained in the PCE problem formulation as sketched in Remark~\ref{rem:causalityPCE_descr}. For all $k \in \mathbb{I}_{[0,N-1]}$, the disturbance $W_k$ is modelled as $\textit{i.i.d.}$ Gaussian random variables with  distribution $\mathcal{N}(0,0.1^2)$.

We want to steer the system to the point $(y,u)=([20,0]^\top,0)$ and solve the OCP in form of \eqref{eq:OCPdataPCE} with horizon $N=20$.  The weighting matrices in the objective function are chosen to be $Q=I_{2}$ and $R=1$.

In the data collection phase, we record $160$ output-input-noise measurements to construct the Hankel matrices. With respect to the initial condition, we assume no noise measurement is available at run-time.  Thus, we model the noise in OCP~\eqref{eq:OCPdataPCE} via its PCE. Moreover, the input applied to the system is randomly sampled from a uniform distribution with support $[-1,1]$. In sum, we obtain the uncertain initial condition~\eqref{eq:OCPdataPCE_consist_cond} which is modelled via PCE. 

Slack variables are added to the initial condition as \eqref{eq:IniSlack}, while we penalize them via the $1$-norm with a weighting parameter. Moreover, since the PCE coefficient of $W$ is known, we employ the null-space method to reduce the dimension of $\pce{g}^i$ and thus accelerate the computation. For further details we refer to \cite{Pan21s}. 

The trajectories of first two moments of $Y$ and $U$ are depicted in Figure~\ref{fig:DescriptorMoments} while the trajectories of $Y$ and $U$ for $20$ different noise realizations are  shown in Figure~\ref{fig:DescriptorTraj}. Note that in the figure we plot the solution on the horizon length $N=20$ and leave out the additional time steps  of the initial condition. Furthermore, we sample a total of 1000 initial conditions as well as the noise realization sequences and compute the corresponding output/input trajectories. The evolution of the normalized histograms of output realizations $y_1$ at $k=0,4,9,14,19$ is illustrated in Figure~\ref{fig:DescriptorDist}. As one can see, the proposed data-driven optimal control achieves a narrow distribution of $Y_1$ around $y_1=20$.
 
\begin{figure}[!t]
\centering
\subfloat[Mean and variance obtained via PCE coefficients.]{
    \label{fig:DescriptorMoments}
    \includegraphics[width=1\columnwidth]{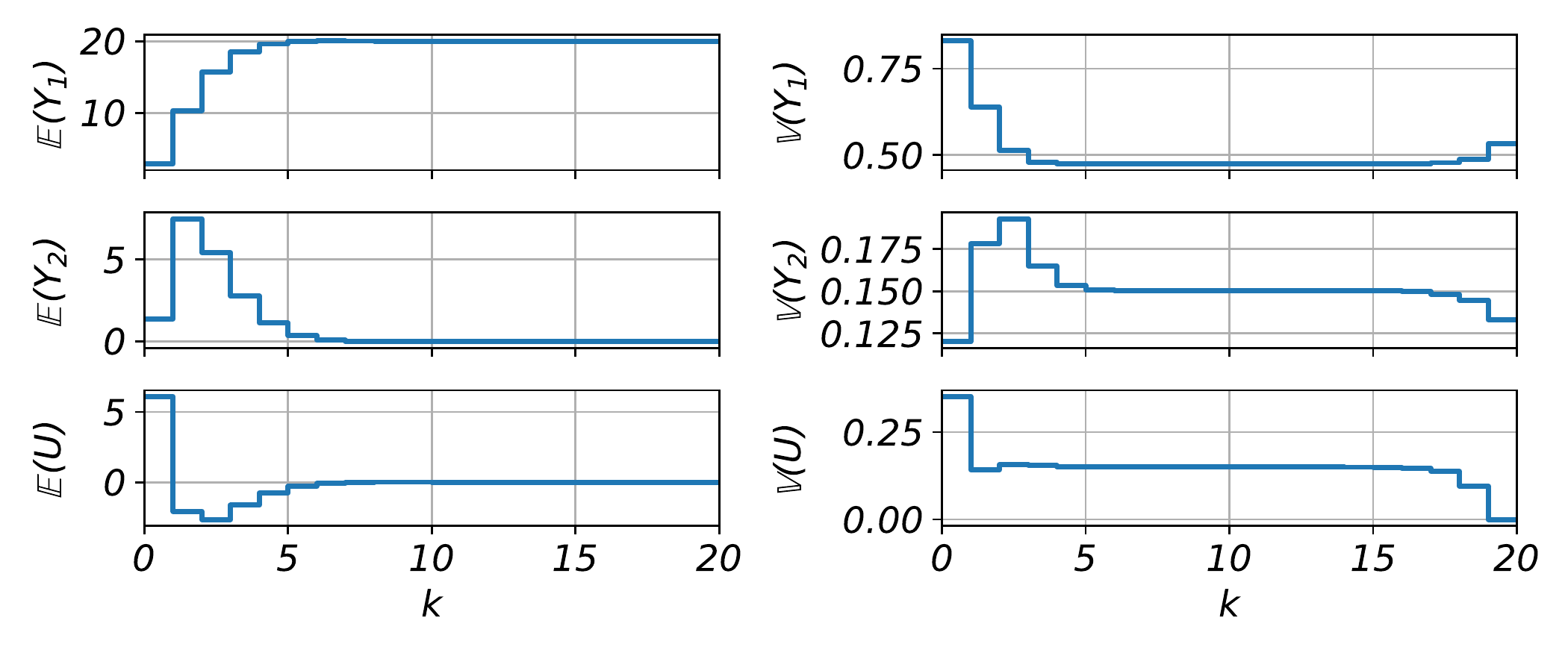}
}
\hfill
\subfloat[$20$ different realizations trajectories.]{
    \label{fig:DescriptorTraj}
     \includegraphics[width=1\columnwidth]{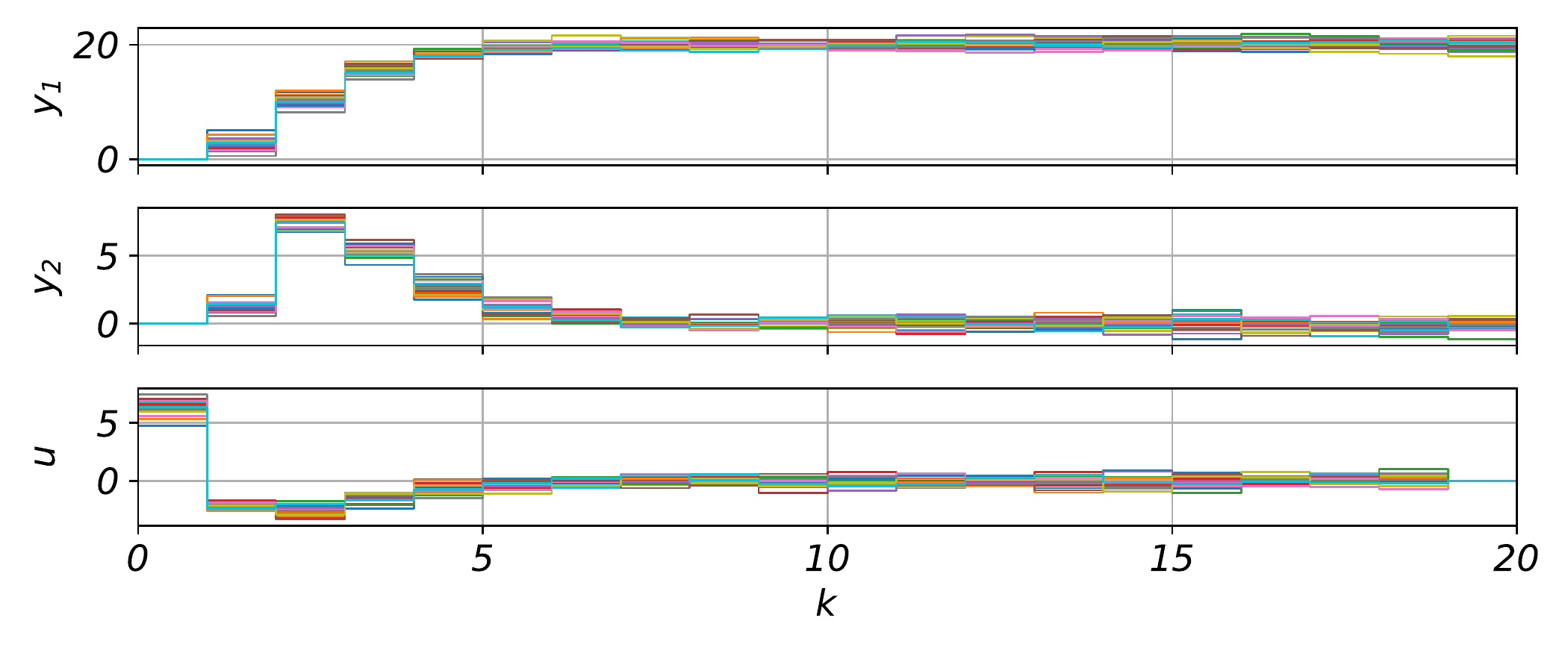}
}
\hfill
\subfloat[Histograms of the output $Y_1$ from $1000$ sampled trajectories.]{
		\label{fig:DescriptorDist}
		\includegraphics[width=0.95\columnwidth]{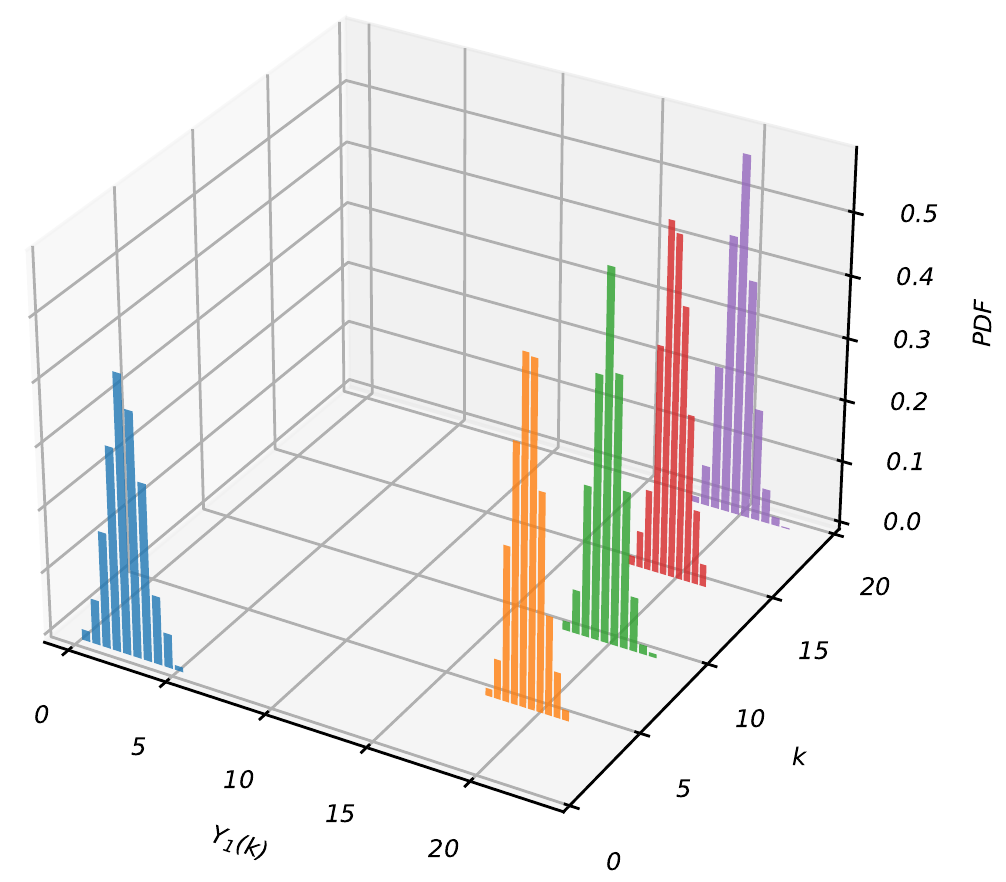}
	}

\caption{Results for the descriptor example.}
\end{figure}
\section{Discussion and Open Problems} \label{sec:Discussion}

As discussed in Section~\ref{sec:recapI} the fundamental lemma as such has deep roots in behavioral systems theory. In this tutorial-style paper we proposed extensions of the fundamental lemma to stochastic settings. To this end, we formalized concepts for behavioral characterizations of stochastic systems. We introduced the stochastic $L^2$-behavior and its counterpart in terms of series expansion coefficients (which for suitable polynomial chaos expansions can be shown to be finite dimensional). We have  analyzed how these behaviors relate to the behavior of sampled trajectories (i.e.\ the realization behavior). We have also shown that the description in stochastic moments is structurally limited: (i) moments are by their very construction nonlinear projections of random variables, and (ii) this projection, especially if limited to the first two moments, is subject to a substantial loss of information. 

Moreover, we remark that the stochastic behavioral description as such does not necessitate the existence of a finite series expansion. Yet, finite-dimensional parametrization of the random-variables fosters numerical tractability in applications.\vspace{2mm}

\noindent\textbf{Coarse $\sigma$-Algebras underlying Stochastic Behaviors}.
The careful reader has surely recognized that going for PCE as a specific and numerically favourable series expansion of $L^2$-random variables is subject to a subtle loss of information in the underlying $\sigma$-algebra, which may potentially jeopardize the equivalence in description established by the lift between the expansion coefficient behavior $\mathfrak C_\infty$ and the random variable behavior $\mathfrak S_\infty$, cf. Remark~\ref{rem:liftPCE}. However, as pointed out by \cite{Willems12} and \cite{Baggio17}, from a behavioral point of view---and actually also from a dynamical-systems perspective---it is advisable to consider \textit{coarse} $\sigma$-algebras instead of finer ones. Besides the structure discussed by \cite{Willems12,Baggio17}, where the coarse $\sigma$-algebra is given along the directions perpendicular to the underlying (deterministic) behavior, the loss of information in the context of PCE seems to result in coarseness of a different nature. Hence, two issues arise:
\begin{enumerate}
    \item [(i)] To which extent is the  information loss in the $\sigma$-algebras induced by the orthogonal PCE basis  of major concern?
    \item [(ii)] How is this  loss related to the \textit{coarse algebras} suggested by \cite{Willems12} and \cite{Baggio17}? 
\end{enumerate}
Indeed, from an engineering perspective, issue~(i) appears very much non-critical. Moreover, one may argue that considering the $\sigma$-algebra $\sigma(\mathfrak F)$, i.e.\ the one induced by a  family $\mathfrak F$ of random variables in the context of PCE (cf.\ Subsection~\ref{sec:PCE}), instead of the full $\sigma$-algebra $\mathcal F$ is a straightforward way of coarsening. Hence, considering $\sigma(\mathfrak F)$ as the $\sigma$-algebra underlying the considered $L^2$-probability space ensures behavioral equivalence. Yet, there may exist other tractable approaches to construct coarse(r) $\sigma$-algebras. Thus issue~(ii) calls for further investigations. \vspace{1mm}

\noindent \textbf{Data-driven Output Feedback Predictive Control}. 
While the fundamental lemma  has already been published in 2005 ~\citep{WRMDM05}, its impact and reception peaked upon realizing that is opens a path towards data-driven output-feedback predictive control \citep{Yang15a,coulson2019data} and towards direct data-driven control. Not only does the data-driven approach alleviate the need for explicit identification of a state-space model---it also works with input-output data. Recall that (as per the leading M) MPC as such is traditionally a model-based state-feedback strategy, i.e., observer design is an inevitable step in almost all implementations. The data-driven approach allows to overcome this burden.

In view of the contributions of this tutorial-style paper our developments pave the road towards data-driven output-feedback sto\-chastic predictive control. First steps towards data-driven sto\-chastic MPC, i.e., problem formulation and stability analysis for the case of usual LTI systems with state feedback, have been done by~\cite{Pan21s,tudo:pan22a}. At the time of writing these lines, first result on the analysis for the output-feedback case with terminal constraints are available as preprints~\citep{pan2022e,pan22d}. The extension to the stochastic descriptor case and to closed-loop guarantees without terminal ingredients is completely open. 

Another dimension, which will be key for the success of data-driven MPC, are tailored numerical methods. While for model-based MPC powerful codes enable computation times down to a few micro-seconds (depending on the problem size and the computational platform), the results on multiple-shoo\-ting formulations of the OCP~\citep{tudo:ou23a}---which we commented on in Section~\ref{sec:StochOCPs}---are first steps in this direction. Yet, they leave substantial room for further refinements and improvements.\vspace*{2mm} 

\noindent \textbf{Robustness}. 
Robustness analysis for data-driven description of deterministic systems has been subject to widespread research efforts, see~\citep{coulson2019regularized,yin2021maximum,Berberich20c,Huang21r} and the overview~\citep{MarkDorf21}. We expect that due to the close relation between the realization behavior and the coefficient behavior, cf.\ Theorem~\ref{thm:equivalence}, these results can be transferred to the stochastic setting.\vspace*{2mm}

\noindent \textbf{Data-driven Analysis of Descriptor Systems}. 
The data-driven approach can be used beyond feedback design. Indeed it also allows analysis of system properties, see, e.g., \citep{romer2019data,romer2019one} for passivity and dissipativity. When it comes to descriptor systems much less has been done. Specifically, the verification of structural properties besides the dimension of $z^J$ in Lemma~\ref{lem:dim_zJ} (nilpotency index, regularity of the matrix pencil, etc.) would be appealing. 
\vspace*{2mm}

\noindent \textbf{Nonlinear Systems}. %
Data-driven control of nonlinear systems is a topic, which has seen substantial progress. One can mention the recent work on data-driven prediction of so-called observables, i.e., real- or complex-valued functions of the state, in the Koopman framework by means of (extended) Dynamic Mode Decomposition (DMD), see~\citep{BrunKutz22} as well as the recent surveys by~\cite{bevanda2021koopman} and~\cite{BrunBudi21} as well as the references therein, or SINDy
~\citep{brunton2016discovering}. The Koopman framework can also be applied for control. To this end, \cite{ProcBrun16} augmented the state dimension~$n_x$ by the number of input variables~$n_u$ to set up a surrogate model, see, e.g., \citep{arbabi18,korda2018linear,MaurSusu20}. To alleviate the curse of dimensionality, \citet{WillHema2016}, \citet{Sura2016}, \cite{klus20} proposed a bilinear approach for control-affine systems which shows a superior performance for systems with coupling between state and control, see also~\cite[Section~4]{MaurSusu20}. Moreover, \citet{SchaWort22} provide rigorous bounds on the prediction error for control explicitly depending on the dictionary size (projection) and the number of employed data points (estimation). The analysis of the estimation error can be extended to the setting based on stochastic differential equations and ergodic sampling, see~\citep{nuske2021finite}.
However, data-driven control and fundamental-lemma-type results for nonlinear stochastic systems are not known yet. Similarly, there is substantial prospect of deriving fundamental-lemma-like results for infinite dimensional systems.

\section{Conclusions}\label{sec:Outlook}
This paper has taken a fresh look at  behavioral theory for stochastic systems. We have constructed equivalent behavioral characterizations of stochastic linear time-invariant systems in terms of $L^2$-random variables and in terms of series expansions of these variables. In other words, we have connected the seminal contributions of Jan C. Willems and co-workers with classic concepts put forward by Norbert Wiener. 

Importantly, our developments show that there is a synergy of linearity of system structures and the linearity of series expansions of random variables, and of polynomial chaos expansions in particular. This synergy enabled the introduction of novel concepts such as behavioral lifts between different behavioral representations of stochastic systems. 
The series expansion approach allows to extend the celebrated fundamental lemma to stochastic systems in explicit form and in descriptor form. Crucially, these extensions are built on Hankel matrices comprising realizations of stochastic variables, i.e., they rely on measurement data only. 

Thus, the presented contributions open up  
new perspectives on data-driven control of stochastic systems, on data-driven uncertainty quantification, and on data-driven uncertainty propagation without being restricted to the Gaussian setting.

\section*{Acknowledgements}
The authors acknowledge the very helpful comments and suggestions made by the anonymous reviewer and they express their gratitude to Paolo Rapisarda for helpful literature suggestions.

Karl Worthmann gratefully acknowledges support from the German Research Foundation (DFG; grant no.\ 507037103).

Ruchan Ou acknowledges partial funding by the German Federal Ministry of Education and Research (BMBF) in the course of the 6GEM research hub under grant number 16KISK038.

Philipp Schmitz is grateful for the support from the TU~Ilmenau (internal excellence funding; project KITE) and the Carl Zeiss Foundation (DeepTurb---Deep Learning in and of Turbulence; project No.\ 2018-02-001).
\section*{Appendix}
\label{appendix}

\noindent \textbf{Explicit Reformulation of Descriptor Systems.}
It is known that descriptor systems allow an equivalent representation in terms of explicit LTI systems---at least in the sense of their manifest behaviors, cf. \citep[Theorem 3]{willems86i} and \citep[Section IX.]{Willems91}. However, attaining an explicit LTI representation comes at the expense of inevitable reinterpretation (permutation) of input and output signals which may be unfavorable if chosen inputs and outputs have physical interpretations reflecting on the technical possibilities to measure and to actuate. 

Given a manifest behavior $\mathfrak B_\infty^\text{i/o}$ realized by a regular descriptor system~\eqref{sys_dae} we show in the following that there exists a realization $\mathfrak B_\infty^\text{i/o}$ in terms of an explicit LTI system up to some permutation of the input and output variables. Since we consider the manifest behavior, we assume without loss of generality that the descriptor system is in quasi-Weierstraß from~\eqref{qweier}.
We have $(u,y)\in\mathfrak B_\infty^\text{i/o}$ if and only if there is $x^J:\mathbb N\rightarrow \mathbb R^{n_J}$ such that
    \begin{subequations}
    \label{eq:syszz}
    \begin{align}
        x_{k+1}^J &= J x_k^J + B_J u_k\\
        \label{eq:syszzb}
        y_k &= C_J x^J_k + Du_k -\sum_{i=0}^{\delta -1}C_NN^iB_Nu_{k+i}
    \end{align}      
    \end{subequations}
    for all $k\in\mathbb N$, cf.\ \eqref{qsysb} and \eqref{evolution}.
    Observe that it may happen because of the output matrix that $C_N N^{\delta-1}B_N=0$. Thus deviating from the definition of the structured nilpotency index here we take $\delta$ as the smallest positive integer such that $C_N N^{\delta-1}B_N\neq 0$ and $C_N N^{\delta}B_N=0$, or when $C_NB_N =0$ we set $\delta=1$. In the case $\delta=1$ one sees that \eqref{eq:syszz} is already an explicit LTI system.
    
    Suppose that $\delta>1$ and let $r\doteq\operatorname{rk}(C_N N^{\delta-1} B_N)\geq 1$. We find permutation matrices $T\in \mathbb R^{n_y\times n_y}$, $S\in\mathbb R^{n_u\times n_u}$ such that for
    \begin{equation}
    \label{block}
       T(D -  C_N  B_N) S = \begin{bmatrix}
           \Phi_0 & \Gamma_0\\ F_0 & G_0
       \end{bmatrix},\quad -T C_N N^{i} B_N S = \begin{bmatrix}
            \Phi_{i} & \Gamma_{i} \\ F_{i} & G_{i}
        \end{bmatrix},
    \end{equation}
    with $\Phi_0,\Phi_i \in \R^{r\times r}$ and $i=1,\dots,\delta-1$, the matrix $\Phi_{\delta-1}$ is invertible. As $\operatorname{rk}(C_N N^{\delta-1} B_N)=\operatorname{rk}(\Phi_{\delta-1})=r$ the Schur complement of $-TC_N N^{\delta-1} B_NS$ with respect to $\Phi_{\delta-1}$ satisfies
    \begin{equation}
    \label{eq:schur}
        G_{\delta-1}-F_{\delta-1} \Phi_{\delta-1}^{-1} \Gamma_{\delta-1} = 0.
    \end{equation}
    Let $T^\top=\begin{bmatrix}
       T_1^\top & T_2^\top
    \end{bmatrix}$ and $S=\begin{bmatrix}
       S_1 & S_2
    \end{bmatrix}$ with $T_1\in \mathbb R^{r\times n_y}$, $S_1\in\mathbb R^{n_u\times r}$ and
    decompose the (permuted) input and output signals accordingly,
    \begin{equation*}
        u_k = S S^{-1}u_k = \begin{bmatrix}
            S_1 & S_2
        \end{bmatrix}
        \begin{bmatrix}
            \vartheta_k\\ \hat u_k
        \end{bmatrix}, \quad \begin{bmatrix}
            \eta_k\\
            \hat y_k
        \end{bmatrix} = \begin{bmatrix}
           T_1\\ T_2
        \end{bmatrix}y_k = 
        Ty_k.
    \end{equation*}
    Then \eqref{eq:syszzb} is equivalent to
    \begin{subequations}
    \begin{align}
        \label{eq:resolvea}
        \eta_k &= T_1 C_J x_k^J - \sum_{i=0}^{\delta-1} \bigl(\Phi_i \vartheta_{k+i} + \Gamma_i \hat u_{k+i}\bigr)\\
        \label{eq:resolveb}
        \hat y_k &= T_2 C_J x_k^J - \sum_{i=0}^{\delta-1} \bigl(F_i\vartheta_{k+i} + G_i \hat u_{k+i}\bigr)
    \end{align}    
    \end{subequations}
    Solving equation~\eqref{eq:resolvea} for $\vartheta_{k+\delta-1}$ and plugging this into \eqref{eq:resolveb} one obtains together with~\eqref{eq:schur}
    \begin{subequations}
    \label{panama}
    \begin{align}
        \vartheta_{k+\delta-1} ={ }& \Phi_{\delta-1}^{-1} \left( \eta_k -T_1 C_J x_k^J - \sum_{i=0}^{\delta-2} \Phi_i \vartheta_{k+i} - \sum_{i=0}^{\delta-1}\Gamma_i \hat u_{k+i}\right)
        \\
        \begin{split}
        \hat y_k ={ } &(T_2 -F_{\delta-1} \Phi_{\delta-1}^{-1}T_1) C_Jx_k^J + F_{\delta-1} \Phi_{\delta-1}^{-1} \eta_k\\
        &- \sum_{i=0}^{\delta-2} (F_i-F_{\delta-1}\Phi_{\delta-1}^{-1}\Phi_i) \vartheta_{k+i}\\
        &- \sum_{i=0}^{\delta-2} (G_i-F_{\delta-1}\Phi_{\delta-1}^{-1}\Gamma_i) \hat u_{k+i}
        \end{split}
    \end{align}
    \end{subequations}
    We introduce the augmented state variable
    \begin{equation*}
    \label{argentinia}
        \tilde x_k = \begin{bmatrix}
            x_k^J\\\vartheta_k\\\vdots\\\vartheta_{k+\delta-2}
        \end{bmatrix}
        \end{equation*}
    and the matrix
    \begin{equation}
        \mathcal A = \begin{bmatrix}
            J & B_J S_1 &0\\ 
            0 &  0 & I_{r(\delta-2)} \\
            -\Phi_{\delta-1}^{-1}T_1C_J &  \Phi_{\delta-1}^{-1}\Phi_0 & \Phi
        \end{bmatrix}
    \end{equation}
    with $\Phi = \begin{bmatrix}
            \Phi_{\delta-1}^{-1} \Phi_1&\dots &\Phi_{\delta-1}^{-1} \Phi_{\delta-2}
        \end{bmatrix}$. The idea of augmentation is borrowed from the continuous-time setting \citep{ilchmann2018model,IlchLebe19}, where the state is augmented by derivatives of the control input. Further, let
         \begin{equation}
        B = \begin{bmatrix}
            0\\0\\\Phi_{\delta-1}^{-1}
        \end{bmatrix},\quad \widetilde B_i = \begin{bmatrix}
            0\\0\\-\Phi_{\delta-1}^{-1}\Gamma_i
        \end{bmatrix},
    \end{equation}
    \begin{equation}
    \label{peru}
        \mathcal C = \begin{bmatrix}
            0 & I_{r} & 0\\
            (T_2 -F_{\delta-1} \Phi_{\delta-1}^{-1}T_1) C_J & \Psi_0 & \Psi
        \end{bmatrix}
    \end{equation}
    with $\Psi = \begin{bmatrix}
            \Psi_1 & \dots & \Psi_{\delta-2}
        \end{bmatrix}$, where $\Psi_i = -(F_i -F_{\delta-1}\Phi_{\delta-1}^{-1}\Phi_i)$.
    The equations in \eqref{panama} can be reformulated as
    \begin{subequations}
    \label{ecuador}
    \begin{align}
    \label{ecuador_a}
        \tilde x_{k+1} &= \mathcal A \tilde x_{k} + B\eta_k + \sum_{i=0}^{\delta-1}\tilde B_{i} \hat u_{k+i}\\
        \label{ecuador_b}
        \vartheta_k &= \begin{bmatrix}
     0 & I_{r} & 0
 \end{bmatrix} \tilde x_k\\
    \hat y_k &= \mathcal C \tilde x_k + F_{\delta-1}\Phi_{\delta-1}^{-1} \eta_k - \sum_{i=0}^{\delta-2} (G_i-F_{\delta-1}\Phi_{\delta-1}^{-1}\Gamma_i) \hat u_{k+i}.
    \end{align} 
    \end{subequations}
    We adjust the augmented state variable in order to eliminate of the non-causality in \eqref{ecuador_a},
    \begin{equation}
        \hat x_k \doteq \tilde x_k - \sum_{i=1}^{\delta-1}\sum_{j=1}^{i} \mathcal A^{j-1}\widetilde B_i \hat u_{k+i-j} = \tilde x_k - \sum_{\gamma=0}^{\delta-2}\sum_{i=\gamma+1}^{\delta-1} \mathcal A^{i-\gamma-1}\widetilde B_i \hat u_{k+\gamma},
    \end{equation}
    where the second equality results from a reordering of the summands. We obtain
    \begin{equation}
        \begin{split}
            &\tilde x_{k+1}-\hat x_{k+1} =\sum_{i=1}^{\delta-1}\sum_{j=1}^{i} \mathcal A^{j-1}\widetilde B_i \hat u_{k+i-j+1} \\
            &= \sum_{i=1}^{\delta-1}\sum_{j=2}^{i} \mathcal A^{j-1}\widetilde B_i \hat u_{k+i-j+1} + \sum_{i=1}^{\delta-1} \widetilde B_i \hat u_{k+i}\\
            &= \mathcal A\sum_{i=1}^{\delta-1}\sum_{j=1}^{i-1} \mathcal A^{j-1}\widetilde B_i \hat u_{k+i-j} + \sum_{i=1}^{\delta-1} \widetilde B_i \hat u_{k+i}\\
             &= \mathcal A\sum_{i=1}^{\delta-1}\sum_{j=1}^{i} \mathcal A^{j-1}\widetilde B_i \hat u_{k+i-j} + \sum_{i=1}^{\delta-1} \widetilde B_i \hat u_{k+i} - \sum_{i=1}^{\delta-1} \mathcal A^{i}\widetilde B_i \hat u_{k}.
        \end{split}
    \end{equation}
    Together with \eqref{ecuador_a} this shows
    \begin{equation}
    \label{venezuala}
        \hat x_{k+1} = \mathcal A \hat x_k + B\eta_k + \sum_{i=0}^{\delta-1} \mathcal A^i\widetilde B_i \hat u_k
    \end{equation}
 Further, ones sees from \eqref{ecuador_b} that
 \begin{equation*}
     \vartheta_k = \begin{bmatrix}
     0 & I_{r} & 0
 \end{bmatrix} \tilde x_k =  \begin{bmatrix}
     0 & I_{r} & 0
 \end{bmatrix}\left(\hat x_k + \sum_{i=1}^{\delta-1}\sum_{j=1}^{i} \mathcal A^{j-1}\widetilde B_i \hat u_{k+i-j}\right).
 \end{equation*}
 As $\begin{bmatrix}
     0 & I_{r} & 0
 \end{bmatrix}\mathcal A^{j-1}\widetilde B_i=0$ for $j<\delta-1$ and all $i=0,\dots, \delta-1$ this simplifies to
 \begin{equation}
 \label{colombia}
     \vartheta_k = \begin{bmatrix}
     0 & I_{r} & 0
 \end{bmatrix}\left(\hat x_k + \mathcal A^{\delta-2}\widetilde B_{\delta-1} \hat u_{k}\right)
 \end{equation}
 Consequently, \eqref{ecuador}, \eqref{venezuala} and \eqref{colombia} yield
    \begin{subequations}
    \label{uruguay}
        \begin{align}
            \hat x_{k+1} &= \mathcal A \hat x_{k} + \mathcal B \begin{bmatrix}
                \eta_k\\ \hat u_k
            \end{bmatrix}\\
            \begin{bmatrix}
                \vartheta_k\\
                \hat y_k
            \end{bmatrix} &= \mathcal C \hat x_k + \mathcal D \begin{bmatrix}
                \eta_k\\ \hat u_k
            \end{bmatrix} + \sum_{i=1}^{\delta-2} \begin{bmatrix}
                0&0 \\0& \Lambda_i
            \end{bmatrix} \begin{bmatrix}
                \eta_{k+i}\\
                \hat u_{k+i}
            \end{bmatrix}
        \end{align}
    \end{subequations}
    with $\mathcal A$ and $\mathcal C$ as in \eqref{argentinia} and \eqref{peru}, respectively and 
    \begin{equation*}
        \mathcal B = \begin{bmatrix}
            B & \sum_{i=0}^{\delta-1} \mathcal A^i \widetilde B_i
        \end{bmatrix},\quad \mathcal D = \begin{bmatrix}
            0 & \mathcal A^{\delta-2}\widetilde B_{\delta-1}\\
            F_{\delta-1}\Phi_{\delta-1}^{-1} & \Lambda_0
        \end{bmatrix},
    \end{equation*}
    \begin{equation*}
        \Lambda_i = -(G_i - F_{\delta-1} \Phi_{\delta-1}^{-1} \Gamma_i) + \mathcal C\sum_{\gamma=i+1}^{\delta-1} A^{\gamma-i-1} \widetilde B_\gamma.
    \end{equation*}
    Observe that the manifest variables, that are inputs and outputs jointly together, of~\eqref{panama} are up to some permutation the same as for \eqref{uruguay}. Moreover, \eqref{panama} and \eqref{uruguay} are structural similar, but in contrast the \eqref{panama} the highest order of future inputs in \eqref{uruguay} is $\delta-2$, and, thus reduces by one order. Although, \eqref{uruguay} may still not be in explicit LTI form, further repetitions of the preceding steps, i.e.\ partially exchanging the role of inputs and outputs variables regarding the rank of $\Lambda_{\delta-2}$ followed by augmentation and adjustment of the state variable, leads to a further decrease in the order. At least after $\delta-2$ additional iterations this results in an explicit LTI realization for $\mathfrak B_\infty^\text{i/o}$ (up to permutation of variables).

\bibliography{refs.bib}

\begin{thebibliography}{149}
\expandafter\ifx\csname natexlab\endcsname\relax\def\natexlab#1{#1}\fi
\providecommand{\url}[1]{\texttt{#1}}
\providecommand{\href}[2]{#2}
\providecommand{\path}[1]{#1}
\providecommand{\DOIprefix}{doi:}
\providecommand{\ArXivprefix}{arXiv:}
\providecommand{\URLprefix}{URL: }
\providecommand{\Pubmedprefix}{pmid:}
\providecommand{\doi}[1]{\href{http://dx.doi.org/#1}{\path{#1}}}
\providecommand{\Pubmed}[1]{\href{pmid:#1}{\path{#1}}}
\providecommand{\bibinfo}[2]{#2}
\ifx\xfnm\relax \def\xfnm[#1]{\unskip,\space#1}\fi
%Type = Misc
\bibitem[{AIw(2005)}]{AIwinter}
 (\bibinfo{year}{2005}).
\newblock \bibinfo{title}{{AI} {N}ewsletter}.
\newblock \URLprefix
  \url{http://www.ainewsletter.com/newsletters/aix_0501.htm#w}
  \bibinfo{note}{archived newsletter, accessed 2022-10-10}.
%Type = Article
\bibitem[{Ahbe et~al.(2020)Ahbe, Iannelli \& Smith}]{Ahbe20}
\bibinfo{author}{Ahbe, E.}, \bibinfo{author}{Iannelli, A.}, \&
  \bibinfo{author}{Smith, R.~S.} (\bibinfo{year}{2020}).
\newblock \bibinfo{title}{Region of attraction analysis of nonlinear stochastic
  systems using polynomial chaos expansion}.
\newblock {\it \bibinfo{journal}{Automatica}\/},  {\it
  \bibinfo{volume}{122}\/}, \bibinfo{pages}{109187}.
%Type = Inproceedings
\bibitem[{Alsalti et~al.(2021)Alsalti, Berberich, Allg{\"o}wer \&
  M{\"u}ller}]{Alsalti21}
\bibinfo{author}{Alsalti, M.}, \bibinfo{author}{Berberich, V.~G., J.and~Lopez},
  \bibinfo{author}{Allg{\"o}wer, F.}, \& \bibinfo{author}{M{\"u}ller, M.~A.}
  (\bibinfo{year}{2021}).
\newblock \bibinfo{title}{Data-based system analysis and control of flat
  nonlinear systems}.
\newblock In {\it \bibinfo{booktitle}{Proc. 2021 60th IEEE Conference on
  Decision and Control}\/} (pp. \bibinfo{pages}{1484--1489}).
\newblock \bibinfo{organization}{IEEE}.
%Type = Inproceedings
\bibitem[{Arbabi et~al.(2018)Arbabi, Korda \& Mezi{\'c}}]{arbabi18}
\bibinfo{author}{Arbabi, H.}, \bibinfo{author}{Korda, M.}, \&
  \bibinfo{author}{Mezi{\'c}, I.} (\bibinfo{year}{2018}).
\newblock \bibinfo{title}{A data-driven {Koopman} model predictive control
  framework for nonlinear partial differential equations}.
\newblock In {\it \bibinfo{booktitle}{Proc. 2018 57th IEEE Conference on
  Decision and Control}\/} (pp. \bibinfo{pages}{6409--6414}).
\newblock \bibinfo{organization}{IEEE}.
%Type = Article
\bibitem[{Axehill(2015)}]{Axehill15}
\bibinfo{author}{Axehill, D.} (\bibinfo{year}{2015}).
\newblock \bibinfo{title}{Controlling the level of sparsity in {MPC}}.
\newblock {\it \bibinfo{journal}{Systems \& Control Letters}\/},  {\it
  \bibinfo{volume}{76}\/}, \bibinfo{pages}{1--7}.
%Type = Article
\bibitem[{Baggio \& Sepulchre(2017)}]{Baggio17}
\bibinfo{author}{Baggio, G.}, \& \bibinfo{author}{Sepulchre, R.}
  (\bibinfo{year}{2017}).
\newblock \bibinfo{title}{{LTI} stochastic processes: {A} behavioral
  perspective}.
\newblock {\it \bibinfo{journal}{IFAC-PapersOnLine}\/},  {\it
  \bibinfo{volume}{50}\/}(1), \bibinfo{pages}{2806--2811}.
%Type = Article
\bibitem[{Ball \& Staffans(2006)}]{ball2006conservative}
\bibinfo{author}{Ball, J.~A.}, \& \bibinfo{author}{Staffans, O.~J.}
  (\bibinfo{year}{2006}).
\newblock \bibinfo{title}{Conservative state-space realizations of dissipative
  system behaviors}.
\newblock {\it \bibinfo{journal}{Integral Equations and Operator Theory}\/},
  {\it \bibinfo{volume}{54}\/}(2), \bibinfo{pages}{151--213}.
%Type = Incollection
\bibitem[{Baudin et~al.(2017)Baudin, Dutfoy, Iooss \&
  Popelin}]{baudin2017openturns}
\bibinfo{author}{Baudin, M.}, \bibinfo{author}{Dutfoy, A.},
  \bibinfo{author}{Iooss, B.}, \& \bibinfo{author}{Popelin, A.-L.}
  (\bibinfo{year}{2017}).
\newblock \bibinfo{title}{{OpenTURNS}: An industrial software for uncertainty
  quantification in simulation}.
\newblock In {\it \bibinfo{booktitle}{Handbook of Uncertainty
  Quantification}\/} (pp. \bibinfo{pages}{2001--2038}).
\newblock \bibinfo{address}{Cham}: \bibinfo{publisher}{Springer International
  Publishing}.
%Type = Book
\bibitem[{{Belov} et~al.(2018){Belov}, {Andrianova} \&
  {Kurdyukov}}]{BelovAndrianovaKurdyukov18}
\bibinfo{author}{{Belov}, A.~A.}, \bibinfo{author}{{Andrianova}, O.~G.}, \&
  \bibinfo{author}{{Kurdyukov}, A.~P.} (\bibinfo{year}{2018}).
\newblock {\it \bibinfo{title}{{Control of Discrete-Time Descriptor Systems. An
  Anisotropy-Based Approach}}\/} volume \bibinfo{volume}{157}.
\newblock \bibinfo{publisher}{Cham: Springer}.
\newblock \DOIprefix\doi{10.1007/978-3-319-78479-3}.
%Type = Inproceedings
\bibitem[{Berberich \& Allg{\"o}wer(2020)}]{Berberich20}
\bibinfo{author}{Berberich, J.}, \& \bibinfo{author}{Allg{\"o}wer, F.}
  (\bibinfo{year}{2020}).
\newblock \bibinfo{title}{A trajectory-based framework for data-driven system
  analysis and control}.
\newblock In {\it \bibinfo{booktitle}{Proc. 2020 European Control
  Conference}\/} (pp. \bibinfo{pages}{1365--1370}).
\newblock \bibinfo{organization}{IEEE}.
%Type = Article
\bibitem[{Berberich et~al.(2022{\natexlab{a}})Berberich, Köhler, Müller \&
  Allgöwer}]{Berberich22}
\bibinfo{author}{Berberich, J.}, \bibinfo{author}{Köhler, J.},
  \bibinfo{author}{Müller, M.~A.}, \& \bibinfo{author}{Allgöwer, F.}
  (\bibinfo{year}{2022}{\natexlab{a}}).
\newblock \bibinfo{title}{Linear tracking {MPC} for nonlinear systems {Part}
  {II}: The data-driven case}.
\newblock {\it \bibinfo{journal}{{IEEE} Transactions on Automatic Control}\/},
  {\it \bibinfo{volume}{67}\/}(9), \bibinfo{pages}{4406--4421}.
  \DOIprefix\doi{10.1109/tac.2022.3166851}.
%Type = Article
\bibitem[{Berberich et~al.(2022{\natexlab{b}})Berberich, Scherer \&
  Allg{\"o}wer}]{Berberich20c}
\bibinfo{author}{Berberich, J.}, \bibinfo{author}{Scherer, C.~W.}, \&
  \bibinfo{author}{Allg{\"o}wer, F.} (\bibinfo{year}{2022}{\natexlab{b}}).
\newblock \bibinfo{title}{Combining prior knowledge and data for robust
  controller design}.
\newblock {\it \bibinfo{journal}{IEEE Transactions on Automatic Control}\/},
  (pp. \bibinfo{pages}{1--16}). \DOIprefix\doi{10.1109/TAC.2022.3209342}.
%Type = Article
\bibitem[{{Berger} et~al.(2012){Berger}, {Ilchmann} \&
  {Trenn}}]{BergerIlchmannTrenn12}
\bibinfo{author}{{Berger}, T.}, \bibinfo{author}{{Ilchmann}, A.}, \&
  \bibinfo{author}{{Trenn}, S.} (\bibinfo{year}{2012}).
\newblock \bibinfo{title}{{The quasi-Weierstrass form for regular matrix
  pencils}}.
\newblock {\it \bibinfo{journal}{{{Linear Algebra and its Applications}}}\/},
  {\it \bibinfo{volume}{436}\/}(10), \bibinfo{pages}{4052--4069}.
  \DOIprefix\doi{10.1016/j.laa.2009.12.036}.
%Type = Incollection
\bibitem[{{Berger} \& {Reis}(2013)}]{BergerReis13}
\bibinfo{author}{{Berger}, T.}, \& \bibinfo{author}{{Reis}, T.}
  (\bibinfo{year}{2013}).
\newblock \bibinfo{title}{{Controllability of linear differential-algebraic
  systems -- A survey}}.
\newblock In {\it \bibinfo{booktitle}{Surveys in Differential-Algebraic
  Equations I}\/} (pp. \bibinfo{pages}{1--61}).
\newblock \bibinfo{publisher}{Berlin: Springer}.
%Type = Article
\bibitem[{Bevanda et~al.(2021)Bevanda, Sosnowski \&
  Hirche}]{bevanda2021koopman}
\bibinfo{author}{Bevanda, P.}, \bibinfo{author}{Sosnowski, S.}, \&
  \bibinfo{author}{Hirche, S.} (\bibinfo{year}{2021}).
\newblock \bibinfo{title}{Koopman operator dynamical models: Learning, analysis
  and control}.
\newblock {\it \bibinfo{journal}{Annual Reviews in Control}\/},  {\it
  \bibinfo{volume}{52}\/}, \bibinfo{pages}{197--212}.
%Type = Book
\bibitem[{Biegler et~al.(2012)Biegler, Campbell \& Mehrmann}]{Biegler12}
\bibinfo{author}{Biegler, L.~T.}, \bibinfo{author}{Campbell, S.~L.}, \&
  \bibinfo{author}{Mehrmann, V.} (\bibinfo{year}{2012}).
\newblock {\it \bibinfo{title}{Control and optimization with
  differential-algebraic constraints}\/}.
\newblock \bibinfo{publisher}{SIAM}.
%Type = Article
\bibitem[{Bienstock et~al.(2014)Bienstock, Chertkov \& Harnett}]{Bienstock14}
\bibinfo{author}{Bienstock, D.}, \bibinfo{author}{Chertkov, M.}, \&
  \bibinfo{author}{Harnett, S.} (\bibinfo{year}{2014}).
\newblock \bibinfo{title}{Chance-constrained optimal power flow: Risk-aware
  network control under uncertainty}.
\newblock {\it \bibinfo{journal}{{SIAM} Review}\/},  {\it
  \bibinfo{volume}{56}\/}(3), \bibinfo{pages}{461--495}.
  \DOIprefix\doi{10.1137/130910312}.
%Type = Article
\bibitem[{Bilgic et~al.(2022)Bilgic, Koch, Pan \& Faulwasser}]{BILGIC2022}
\bibinfo{author}{Bilgic, D.}, \bibinfo{author}{Koch, A.}, \bibinfo{author}{Pan,
  G.}, \& \bibinfo{author}{Faulwasser, T.} (\bibinfo{year}{2022}).
\newblock \bibinfo{title}{Toward data-driven predictive control of multi-energy
  distribution systems}.
\newblock {\it \bibinfo{journal}{Electric Power Systems Research}\/},  {\it
  \bibinfo{volume}{212}\/}, \bibinfo{pages}{108311}.
  \DOIprefix\doi{https://doi.org/10.1016/j.epsr.2022.108311}.
%Type = Article
\bibitem[{Bock \& Plitt(1984)}]{Bock84a}
\bibinfo{author}{Bock, H.~G.}, \& \bibinfo{author}{Plitt, K.-J.}
  (\bibinfo{year}{1984}).
\newblock \bibinfo{title}{A multiple shooting algorithm for direct solution of
  optimal control problems}.
\newblock {\it \bibinfo{journal}{IFAC Proceedings Volumes}\/},  {\it
  \bibinfo{volume}{17}\/}(2), \bibinfo{pages}{1603--1608}.
%Type = Book
\bibitem[{Brezis(2011)}]{Brezis11}
\bibinfo{author}{Brezis, H.} (\bibinfo{year}{2011}).
\newblock {\it \bibinfo{title}{Functional Analysis, Sobolev Spaces and Partial
  Differential Equations}\/} volume~\bibinfo{volume}{2}.
\newblock \bibinfo{publisher}{Springer}.
%Type = Article
\bibitem[{Brunton et~al.(2022)Brunton, Budi\v{s}i\'{c}, Kaiser \&
  Kutz}]{BrunBudi21}
\bibinfo{author}{Brunton, S.~L.}, \bibinfo{author}{Budi\v{s}i\'{c}, M.},
  \bibinfo{author}{Kaiser, E.}, \& \bibinfo{author}{Kutz, J.~N.}
  (\bibinfo{year}{2022}).
\newblock \bibinfo{title}{Modern {K}oopman theory for dynamical systems}.
\newblock {\it \bibinfo{journal}{SIAM Review}\/},  {\it
  \bibinfo{volume}{64}\/}(2), \bibinfo{pages}{229--340}.
  \DOIprefix\doi{10.1137/21M1401243}.
%Type = Book
\bibitem[{Brunton \& Kutz(2022)}]{BrunKutz22}
\bibinfo{author}{Brunton, S.~L.}, \& \bibinfo{author}{Kutz, J.~N.}
  (\bibinfo{year}{2022}).
\newblock {\it \bibinfo{title}{Data-driven science and engineering: Machine
  learning, dynamical systems, and control}\/}.
\newblock \bibinfo{publisher}{Cambridge University Press}.
%Type = Article
\bibitem[{Brunton et~al.(2016)Brunton, Proctor \&
  Kutz}]{brunton2016discovering}
\bibinfo{author}{Brunton, S.~L.}, \bibinfo{author}{Proctor, J.~L.}, \&
  \bibinfo{author}{Kutz, J.~N.} (\bibinfo{year}{2016}).
\newblock \bibinfo{title}{Discovering governing equations from data by sparse
  identification of nonlinear dynamical systems}.
\newblock {\it \bibinfo{journal}{Proceedings of the National Academy of
  Sciences}\/},  {\it \bibinfo{volume}{113}\/}(15),
  \bibinfo{pages}{3932--3937}.
%Type = Article
\bibitem[{Calafiore \& Ghaoui(2006)}]{calafiore06}
\bibinfo{author}{Calafiore, G.~C.}, \& \bibinfo{author}{Ghaoui, L.~E.}
  (\bibinfo{year}{2006}).
\newblock \bibinfo{title}{On distributionally robust chance-constrained linear
  programs}.
\newblock {\it \bibinfo{journal}{Journal of Optimization Theory and
  Applications}\/},  {\it \bibinfo{volume}{130}\/}(1), \bibinfo{pages}{1--22}.
%Type = Article
\bibitem[{Cameron \& Martin(1947)}]{cameron1947orthogonal}
\bibinfo{author}{Cameron, R.~H.}, \& \bibinfo{author}{Martin, W.~T.}
  (\bibinfo{year}{1947}).
\newblock \bibinfo{title}{The orthogonal development of non-linear functionals
  in series of fourier-hermite functionals}.
\newblock {\it \bibinfo{journal}{Annals of Mathematics}\/},  (pp.
  \bibinfo{pages}{385--392}).
%Type = Book
\bibitem[{Campbell et~al.(2019)Campbell, Ilchmann, Mehrmann, Reis
  et~al.}]{campbell2019applications}
\bibinfo{author}{Campbell, S.}, \bibinfo{author}{Ilchmann, A.},
  \bibinfo{author}{Mehrmann, V.}, \bibinfo{author}{Reis, T.} et~al.
  (\bibinfo{year}{2019}).
\newblock {\it \bibinfo{title}{Applications of differential-algebraic
  equations: examples and benchmarks}\/}.
\newblock \bibinfo{publisher}{Springer}.
%Type = Book
\bibitem[{Campi \& Garatti(2018)}]{campi2018introduction}
\bibinfo{author}{Campi, M.~C.}, \& \bibinfo{author}{Garatti, S.}
  (\bibinfo{year}{2018}).
\newblock {\it \bibinfo{title}{Introduction to the Scenario Approach}\/}.
\newblock \bibinfo{publisher}{SIAM}.
\newblock \DOIprefix\doi{10.1137/1.9781611975444}.
%Type = Article
\bibitem[{Campi et~al.(2009)Campi, Garatti \& Prandini}]{campi2009scenario}
\bibinfo{author}{Campi, M.~C.}, \bibinfo{author}{Garatti, S.}, \&
  \bibinfo{author}{Prandini, M.} (\bibinfo{year}{2009}).
\newblock \bibinfo{title}{The scenario approach for systems and control
  design}.
\newblock {\it \bibinfo{journal}{Annual Reviews in Control}\/},  {\it
  \bibinfo{volume}{33}\/}(2), \bibinfo{pages}{149--157}.
%Type = Inproceedings
\bibitem[{Carlet et~al.(2020)Carlet, Favato, Bolognani \&
  D{\"o}rfler}]{Carlet20}
\bibinfo{author}{Carlet, P.~G.}, \bibinfo{author}{Favato, A.},
  \bibinfo{author}{Bolognani, S.}, \& \bibinfo{author}{D{\"o}rfler, F.}
  (\bibinfo{year}{2020}).
\newblock \bibinfo{title}{Data-driven predictive current control for
  synchronous motor drives}.
\newblock In {\it \bibinfo{booktitle}{Proc. 2020 IEEE Energy Conversion
  Congress and Exposition (ECCE)}\/} (pp. \bibinfo{pages}{5148--5154}).
\newblock \bibinfo{organization}{IEEE}.
%Type = Inproceedings
\bibitem[{Coulson et~al.(2019{\natexlab{a}})Coulson, Lygeros \&
  D{\"o}rfler}]{coulson2019data}
\bibinfo{author}{Coulson, J.}, \bibinfo{author}{Lygeros, J.}, \&
  \bibinfo{author}{D{\"o}rfler, F.} (\bibinfo{year}{2019}{\natexlab{a}}).
\newblock \bibinfo{title}{{Data-enabled predictive control: In the shallows of
  the DeePC}}.
\newblock In {\it \bibinfo{booktitle}{Proc. 2019 European Control
  Conference}\/} (pp. \bibinfo{pages}{307--312}).
\newblock \bibinfo{organization}{IEEE}.
%Type = Inproceedings
\bibitem[{Coulson et~al.(2019{\natexlab{b}})Coulson, Lygeros \&
  D{\"o}rfler}]{coulson2019regularized}
\bibinfo{author}{Coulson, J.}, \bibinfo{author}{Lygeros, J.}, \&
  \bibinfo{author}{D{\"o}rfler, F.} (\bibinfo{year}{2019}{\natexlab{b}}).
\newblock \bibinfo{title}{Regularized and distributionally robust data-enabled
  predictive control}.
\newblock In {\it \bibinfo{booktitle}{Proc. 2019 58th IEEE Conference on
  Decision and Control}\/} (pp. \bibinfo{pages}{2696--2701}).
\newblock \bibinfo{organization}{IEEE}.
%Type = Article
\bibitem[{Cremer et~al.(2018)Cremer, Konstantelos, Tindemans \&
  Strbac}]{Cremer2018data}
\bibinfo{author}{Cremer, J.~L.}, \bibinfo{author}{Konstantelos, I.},
  \bibinfo{author}{Tindemans, S.~H.}, \& \bibinfo{author}{Strbac, G.}
  (\bibinfo{year}{2018}).
\newblock \bibinfo{title}{Data-driven power system operation: Exploring the
  balance between cost and risk}.
\newblock {\it \bibinfo{journal}{IEEE Transactions on Power Systems}\/},  {\it
  \bibinfo{volume}{34}\/}(1), \bibinfo{pages}{791--801}.
%Type = Book
\bibitem[{{Dai}(1989)}]{Dai89}
\bibinfo{author}{{Dai}, L.} (\bibinfo{year}{1989}).
\newblock {\it \bibinfo{title}{{Singular Control Systems}}\/} volume
  \bibinfo{volume}{118}.
\newblock \bibinfo{publisher}{Berlin etc.: Springer-Verlag}.
%Type = Article
\bibitem[{De~Persis \& Tesi(2019)}]{dePersis19}
\bibinfo{author}{De~Persis, C.}, \& \bibinfo{author}{Tesi, P.}
  (\bibinfo{year}{2019}).
\newblock \bibinfo{title}{Formulas for data-driven control: Stabilization,
  optimality, and robustness}.
\newblock {\it \bibinfo{journal}{IEEE Transactions on Automatic Control}\/},
  {\it \bibinfo{volume}{65}\/}(3), \bibinfo{pages}{909--924}.
%Type = Article
\bibitem[{Deisenroth et~al.(2013)Deisenroth, Fox \& Rasmussen}]{Rasmussen}
\bibinfo{author}{Deisenroth, M.~P.}, \bibinfo{author}{Fox, D.}, \&
  \bibinfo{author}{Rasmussen, C.~E.} (\bibinfo{year}{2013}).
\newblock \bibinfo{title}{Gaussian processes for data-efficient learning in
  robotics and control}.
\newblock {\it \bibinfo{journal}{IEEE Transactions on Pattern Analysis and
  Machine Intelligence}\/},  {\it \bibinfo{volume}{37}\/}(2),
  \bibinfo{pages}{408--423}.
%Type = Article
\bibitem[{D{\"o}rfler et~al.(2022)D{\"o}rfler, Coulson \&
  Markovsky}]{Dorfler22a}
\bibinfo{author}{D{\"o}rfler, F.}, \bibinfo{author}{Coulson, J.}, \&
  \bibinfo{author}{Markovsky, I.} (\bibinfo{year}{2022}).
\newblock \bibinfo{title}{Bridging direct \& indirect data-driven control
  formulations via regularizations and relaxations}.
\newblock {\it \bibinfo{journal}{IEEE Transactions on Automatic Control}\/},
  (pp. \bibinfo{pages}{1--1}). \DOIprefix\doi{10.1109/TAC.2022.3148374}.
%Type = Misc
\bibitem[{D{\"o}rfler et~al.(2021)D{\"o}rfler, Tesi \& De~Persis}]{Dorfler21}
\bibinfo{author}{D{\"o}rfler, F.}, \bibinfo{author}{Tesi, P.}, \&
  \bibinfo{author}{De~Persis, C.} (\bibinfo{year}{2021}).
\newblock \bibinfo{title}{On the certainty-equivalence approach to direct
  data-driven {LQR} design}.
\newblock \bibinfo{note}{ArXiv preprint arXiv:2109.06643}.
%Type = Article
\bibitem[{Ernst et~al.(2012)Ernst, Mugler, Starkloff \&
  Ullmann}]{ErnstMuglerStarkloffUllmann12}
\bibinfo{author}{Ernst, O.~G.}, \bibinfo{author}{Mugler, A.},
  \bibinfo{author}{Starkloff, H.-J.}, \& \bibinfo{author}{Ullmann, E.}
  (\bibinfo{year}{2012}).
\newblock \bibinfo{title}{On the convergence of generalized polynomial chaos
  expansions}.
\newblock {\it \bibinfo{journal}{ESAIM: Mathematical Modelling and Numerical
  Analysis}\/},  {\it \bibinfo{volume}{46}\/}(2), \bibinfo{pages}{317--339}.
%Type = Inproceedings
\bibitem[{Fagiano \& Khammash(2012)}]{fagiano2012nonlinear}
\bibinfo{author}{Fagiano, L.}, \& \bibinfo{author}{Khammash, M.}
  (\bibinfo{year}{2012}).
\newblock \bibinfo{title}{Nonlinear stochastic model predictive control via
  regularized polynomial chaos expansions}.
\newblock In {\it \bibinfo{booktitle}{Proc. 2012 51st IEEE conference on
  decision and control}\/} (pp. \bibinfo{pages}{142--147}).
\newblock \bibinfo{organization}{IEEE}.
%Type = Article
\bibitem[{Farina et~al.(2016)Farina, Giulioni \& Scattolini}]{Farina16a}
\bibinfo{author}{Farina, M.}, \bibinfo{author}{Giulioni, L.}, \&
  \bibinfo{author}{Scattolini, R.} (\bibinfo{year}{2016}).
\newblock \bibinfo{title}{Stochastic linear model predictive control with
  chance constraints--{A review}}.
\newblock {\it \bibinfo{journal}{Journal of Process Control}\/},  {\it
  \bibinfo{volume}{44}\/}, \bibinfo{pages}{53--67}.
%Type = Incollection
\bibitem[{Faulwasser \& Gr\"une(2022)}]{tudo:faulwasser22a}
\bibinfo{author}{Faulwasser, T.}, \& \bibinfo{author}{Gr\"une, L.}
  (\bibinfo{year}{2022}).
\newblock \bibinfo{title}{Turnpike properties in optimal control: An overview
  of discrete-time and continuous-time results}.
\newblock In \bibinfo{editor}{E.~Zuazua}, \& \bibinfo{editor}{E.~Trelat}
  (Eds.), {\it \bibinfo{booktitle}{Handbook of Numerical Analysis}\/}
  chapter~\bibinfo{chapter}{11}. (pp. \bibinfo{pages}{367--400}).
\newblock \bibinfo{publisher}{Elsevier} volume~\bibinfo{volume}{23}.
\newblock \DOIprefix\doi{10.1016/bs.hna.2021.12.011} \bibinfo{note}{arxiv:
  2011.13670.}
%Type = Article
\bibitem[{Favoreel et~al.(1999)Favoreel, De~Moor \& Gevers}]{Favoreel99}
\bibinfo{author}{Favoreel, W.}, \bibinfo{author}{De~Moor, B.}, \&
  \bibinfo{author}{Gevers, M.} (\bibinfo{year}{1999}).
\newblock \bibinfo{title}{{SPC: Subspace predictive control}}.
\newblock {\it \bibinfo{journal}{IFAC Proceedings Volumes}\/},  {\it
  \bibinfo{volume}{32}\/}(2), \bibinfo{pages}{4004--4009}.
%Type = Article
\bibitem[{Feinberg \& Langtangen(2015)}]{feinberg2015chaospy}
\bibinfo{author}{Feinberg, J.}, \& \bibinfo{author}{Langtangen, H.~P.}
  (\bibinfo{year}{2015}).
\newblock \bibinfo{title}{Chaospy: An open source tool for designing methods of
  uncertainty quantification}.
\newblock {\it \bibinfo{journal}{Journal of Computational Science}\/},  {\it
  \bibinfo{volume}{11}\/}, \bibinfo{pages}{46--57}.
%Type = Inproceedings
\bibitem[{Fiedler \& Lucia(2021)}]{Fiedler2021}
\bibinfo{author}{Fiedler, F.}, \& \bibinfo{author}{Lucia, S.}
  (\bibinfo{year}{2021}).
\newblock \bibinfo{title}{On the relationship between data-enabled predictive
  control and subspace predictive control}.
\newblock In {\it \bibinfo{booktitle}{Proc. 2021 European Control
  Conference}\/} (pp. \bibinfo{pages}{222--229}).
\newblock \bibinfo{organization}{IEEE}.
%Type = Article
\bibitem[{Field \& Grigoriu(2004)}]{Field04}
\bibinfo{author}{Field, R.~V.}, \& \bibinfo{author}{Grigoriu, M.}
  (\bibinfo{year}{2004}).
\newblock \bibinfo{title}{On the accuracy of the polynomial chaos
  approximation}.
\newblock {\it \bibinfo{journal}{Probabilistic Engineering Mechanics}\/},  {\it
  \bibinfo{volume}{19}\/}(1), \bibinfo{pages}{65--80}.
  \DOIprefix\doi{https://doi.org/10.1016/j.probengmech.2003.11.017}.
%Type = Inproceedings
\bibitem[{Fisher \& Bhattacharya(2008)}]{Fisher08}
\bibinfo{author}{Fisher, J.}, \& \bibinfo{author}{Bhattacharya, R.}
  (\bibinfo{year}{2008}).
\newblock \bibinfo{title}{On stochastic {LQR} design and polynomial chaos}.
\newblock In {\it \bibinfo{booktitle}{Proc. 2008 American Control Conference
  (ACC)}\/} (pp. \bibinfo{pages}{95--100}).
\newblock \bibinfo{organization}{IEEE}.
%Type = Book
\bibitem[{Ghanem \& Spanos(2003)}]{GhanSpan03}
\bibinfo{author}{Ghanem, R.~G.}, \& \bibinfo{author}{Spanos, P.~D.}
  (\bibinfo{year}{2003}).
\newblock {\it \bibinfo{title}{{Stochastic Finite Elements: A Spectral
  Approach}}\/}.
\newblock (\bibinfo{edition}{Revised} ed.).
\newblock \bibinfo{publisher}{Springer New York}.
%Type = Article
\bibitem[{Givens \& Shortt(1984)}]{givens1984class}
\bibinfo{author}{Givens, C.~R.}, \& \bibinfo{author}{Shortt, R.~M.}
  (\bibinfo{year}{1984}).
\newblock \bibinfo{title}{A class of {Wasserstein} metrics for probability
  distributions}.
\newblock {\it \bibinfo{journal}{Michigan Mathematical Journal}\/},  {\it
  \bibinfo{volume}{31}\/}(2), \bibinfo{pages}{231--240}.
%Type = Article
\bibitem[{Gro{\ss} et~al.(2016)Gro{\ss}, Trenn \& Wirsen}]{Gross16}
\bibinfo{author}{Gro{\ss}, T.}, \bibinfo{author}{Trenn, S.}, \&
  \bibinfo{author}{Wirsen, A.} (\bibinfo{year}{2016}).
\newblock \bibinfo{title}{Solvability and stability of a power system dae
  model}.
\newblock {\it \bibinfo{journal}{Systems \& Control Letters}\/},  {\it
  \bibinfo{volume}{97}\/}, \bibinfo{pages}{12--17}.
%Type = Article
\bibitem[{Heirung et~al.(2018)Heirung, Paulson, O’Leary \&
  Mesbah}]{heirung2018stochastic}
\bibinfo{author}{Heirung, T. A.~N.}, \bibinfo{author}{Paulson, J.~A.},
  \bibinfo{author}{O’Leary, J.}, \& \bibinfo{author}{Mesbah, A.}
  (\bibinfo{year}{2018}).
\newblock \bibinfo{title}{Stochastic model predictive control—how does it
  work?}
\newblock {\it \bibinfo{journal}{Computers \& Chemical Engineering}\/},  {\it
  \bibinfo{volume}{114}\/}, \bibinfo{pages}{158--170}.
%Type = Inproceedings
\bibitem[{Huang et~al.(2019)Huang, Coulson, Lygeros \& D{\"o}rfler}]{Huang19}
\bibinfo{author}{Huang, L.}, \bibinfo{author}{Coulson, J.},
  \bibinfo{author}{Lygeros, J.}, \& \bibinfo{author}{D{\"o}rfler, F.}
  (\bibinfo{year}{2019}).
\newblock \bibinfo{title}{Data-enabled predictive control for grid-connected
  power converters}.
\newblock In {\it \bibinfo{booktitle}{Proc. 2019 58th IEEE Conference on
  Decision and Control}\/} (pp. \bibinfo{pages}{8130--8135}).
\newblock \bibinfo{organization}{IEEE}.
%Type = Article
\bibitem[{Huang et~al.(2021)Huang, Coulson, Lygeros \&
  D{\"o}rfler}]{Huang21decentralized}
\bibinfo{author}{Huang, L.}, \bibinfo{author}{Coulson, J.},
  \bibinfo{author}{Lygeros, J.}, \& \bibinfo{author}{D{\"o}rfler, F.}
  (\bibinfo{year}{2021}).
\newblock \bibinfo{title}{Decentralized data-enabled predictive control for
  power system oscillation damping}.
\newblock {\it \bibinfo{journal}{IEEE Transactions on Control Systems
  Technology}\/},  {\it \bibinfo{volume}{30}\/}(3),
  \bibinfo{pages}{1065--1077}.
%Type = Misc
\bibitem[{Huang et~al.(2022)Huang, Lygeros \& D{\"o}rfler}]{Huang22}
\bibinfo{author}{Huang, L.}, \bibinfo{author}{Lygeros, J.}, \&
  \bibinfo{author}{D{\"o}rfler, F.} (\bibinfo{year}{2022}).
\newblock \bibinfo{title}{Robust and kernelized data-enabled predictive control
  for nonlinear systems}.
\newblock \bibinfo{note}{ArXiv preprint arXiv:2206.01866}.
%Type = Article
\bibitem[{Huang et~al.(2023)Huang, Zhen, Lygeros \& D{\"o}rfler}]{Huang21r}
\bibinfo{author}{Huang, L.}, \bibinfo{author}{Zhen, J.},
  \bibinfo{author}{Lygeros, J.}, \& \bibinfo{author}{D{\"o}rfler, F.}
  (\bibinfo{year}{2023}).
\newblock \bibinfo{title}{Robust data-enabled predictive control: Tractable
  formulations and performance guarantees}.
\newblock {\it \bibinfo{journal}{IEEE Transactions on Automatic Control}\/}, .
%Type = Article
\bibitem[{H{\"u}llermeier \& Waegeman(2021)}]{hullermeier2021aleatoric}
\bibinfo{author}{H{\"u}llermeier, E.}, \& \bibinfo{author}{Waegeman, W.}
  (\bibinfo{year}{2021}).
\newblock \bibinfo{title}{Aleatoric and epistemic uncertainty in machine
  learning: An introduction to concepts and methods}.
\newblock {\it \bibinfo{journal}{Machine Learning}\/},  {\it
  \bibinfo{volume}{110}\/}(3), \bibinfo{pages}{457--506}.
%Type = Article
\bibitem[{Ilchmann et~al.(2019)Ilchmann, Leben, Witschel \&
  Worthmann}]{IlchLebe19}
\bibinfo{author}{Ilchmann, A.}, \bibinfo{author}{Leben, L.},
  \bibinfo{author}{Witschel, J.}, \& \bibinfo{author}{Worthmann, K.}
  (\bibinfo{year}{2019}).
\newblock \bibinfo{title}{Optimal control of differential-algebraic equations
  from an ordinary differential equation perspective}.
\newblock {\it \bibinfo{journal}{Optimal Control Applications and Methods}\/},
  {\it \bibinfo{volume}{40}\/}(2), \bibinfo{pages}{351--366}.
%Type = Article
\bibitem[{Ilchmann et~al.(2018)Ilchmann, Witschel \&
  Worthmann}]{ilchmann2018model}
\bibinfo{author}{Ilchmann, A.}, \bibinfo{author}{Witschel, J.}, \&
  \bibinfo{author}{Worthmann, K.} (\bibinfo{year}{2018}).
\newblock \bibinfo{title}{Model predictive control for linear
  differential-algebraic equations}.
\newblock {\it \bibinfo{journal}{IFAC-PapersOnLine}\/},  {\it
  \bibinfo{volume}{51}\/}(20), \bibinfo{pages}{98--103}.
%Type = Inproceedings
\bibitem[{Katayama(2006)}]{katayama2006system}
\bibinfo{author}{Katayama, T.} (\bibinfo{year}{2006}).
\newblock \bibinfo{title}{A system theoretic interpretation of {LQ}
  decomposition in subspace identification methods}.
\newblock In {\it \bibinfo{booktitle}{Proc. 2006 17th Int. Symposium on
  Mathematical Theory of Networks and Systems (MTNS). Kyoto, Japan}\/} (pp.
  \bibinfo{pages}{1089--1095}).
%Type = Misc
\bibitem[{Kerz et~al.(2021)Kerz, Teutsch, Br{\"u}digam, Wollherr \&
  Leibold}]{Kerz21}
\bibinfo{author}{Kerz, S.}, \bibinfo{author}{Teutsch, J.},
  \bibinfo{author}{Br{\"u}digam, T.}, \bibinfo{author}{Wollherr, D.}, \&
  \bibinfo{author}{Leibold, M.} (\bibinfo{year}{2021}).
\newblock \bibinfo{title}{Data-driven stochastic model predictive control}.
\newblock \bibinfo{note}{ArXiv preprint arXiv:2112.04439}.
%Type = Article
\bibitem[{Kim et~al.(2013)Kim, Shen, Nagy \& Braatz}]{Kim13a}
\bibinfo{author}{Kim, K.-K.~K.}, \bibinfo{author}{Shen, D.~E.},
  \bibinfo{author}{Nagy, Z.~K.}, \& \bibinfo{author}{Braatz, R.~D.}
  (\bibinfo{year}{2013}).
\newblock \bibinfo{title}{Wiener's polynomial chaos for the analysis and
  control of nonlinear dynamical systems with probabilistic uncertainties
  [historical perspectives]}.
\newblock {\it \bibinfo{journal}{IEEE Control Systems Magazine}\/},  {\it
  \bibinfo{volume}{33}\/}(5), \bibinfo{pages}{58--67}.
%Type = Book
\bibitem[{Klenke(2013)}]{klenke13prob}
\bibinfo{author}{Klenke, A.} (\bibinfo{year}{2013}).
\newblock {\it \bibinfo{title}{{Probability Theory: A Comprehensive
  Course}}\/}.
\newblock \bibinfo{publisher}{Springer Science \& Business Media}.
%Type = Article
\bibitem[{Klus et~al.(2020)Klus, N{\"u}ske, Peitz, Niemann, Clementi \&
  Sch{\"u}tte}]{klus20}
\bibinfo{author}{Klus, S.}, \bibinfo{author}{N{\"u}ske, F.},
  \bibinfo{author}{Peitz, S.}, \bibinfo{author}{Niemann, J.-H.},
  \bibinfo{author}{Clementi, C.}, \& \bibinfo{author}{Sch{\"u}tte, C.}
  (\bibinfo{year}{2020}).
\newblock \bibinfo{title}{Data-driven approximation of the {Koopman} generator:
  {Model} reduction, system identification, and control}.
\newblock {\it \bibinfo{journal}{Physica D: Nonlinear Phenomena}\/},  {\it
  \bibinfo{volume}{406}\/}, \bibinfo{pages}{132416}.
%Type = Techreport
\bibitem[{Koekoek \& Swarttouw(1998)}]{Koekoek96}
\bibinfo{author}{Koekoek, R.}, \& \bibinfo{author}{Swarttouw, R.~F.}
  (\bibinfo{year}{1998}).
\newblock {\it \bibinfo{title}{The {Askey-scheme} of hypergeometric orthogonal
  polynomials and its q-analogue}\/}.
\newblock \bibinfo{type}{Technical Report} \bibinfo{number}{98-17} Department
  of Technical Mathematics and Informatics, Delft University of Technology
  \bibinfo{address}{Delft, The Netherlands}.
%Type = Article
\bibitem[{Korda \& Mezi{\'c}(2018)}]{korda2018linear}
\bibinfo{author}{Korda, M.}, \& \bibinfo{author}{Mezi{\'c}, I.}
  (\bibinfo{year}{2018}).
\newblock \bibinfo{title}{Linear predictors for nonlinear dynamical systems:
  Koopman operator meets model predictive control}.
\newblock {\it \bibinfo{journal}{Automatica}\/},  {\it \bibinfo{volume}{93}\/},
  \bibinfo{pages}{149--160}.
%Type = Article
\bibitem[{Kuehn(2016)}]{Kuehn16}
\bibinfo{author}{Kuehn, C.} (\bibinfo{year}{2016}).
\newblock \bibinfo{title}{{Moment Closure—A Brief Review}}.
\newblock {\it \bibinfo{journal}{Control of Self-Organizing Nonlinear
  Systems}\/},  (pp. \bibinfo{pages}{253--271}).
%Type = Book
\bibitem[{Kunkel \& Mehrmann(2006)}]{Kunkel06}
\bibinfo{author}{Kunkel, P.}, \& \bibinfo{author}{Mehrmann, V.}
  (\bibinfo{year}{2006}).
\newblock {\it \bibinfo{title}{Differential-algebraic equations: Analysis and
  numerical solution}\/} volume~\bibinfo{volume}{2}.
\newblock \bibinfo{publisher}{European Mathematical Society}.
%Type = Article
\bibitem[{Lefebvre(2020)}]{Lefebvre20}
\bibinfo{author}{Lefebvre, T.} (\bibinfo{year}{2020}).
\newblock \bibinfo{title}{On moment estimation from polynomial chaos expansion
  models}.
\newblock {\it \bibinfo{journal}{IEEE Control Systems Letters}\/},  {\it
  \bibinfo{volume}{5}\/}(5), \bibinfo{pages}{1519--1524}.
%Type = Inproceedings
\bibitem[{Lian \& Jones(2021)}]{lian2021nonlinear}
\bibinfo{author}{Lian, Y.}, \& \bibinfo{author}{Jones, C.~N.}
  (\bibinfo{year}{2021}).
\newblock \bibinfo{title}{Nonlinear data-enabled prediction and control}.
\newblock In {\it \bibinfo{booktitle}{Proc. 3rd Conference on Learning for
  Dynamics and Control}\/} (pp. \bibinfo{pages}{523--534}).
\newblock \bibinfo{publisher}{PMLR} volume \bibinfo{volume}{144} of {\it
  \bibinfo{series}{Proceedings of Machine Learning Research}\/}.
%Type = Misc
\bibitem[{Lian et~al.(2021)Lian, Shi, Koch \& Jones}]{Lian21}
\bibinfo{author}{Lian, Y.}, \bibinfo{author}{Shi, J.}, \bibinfo{author}{Koch,
  M.~P.}, \& \bibinfo{author}{Jones, C.~N.} (\bibinfo{year}{2021}).
\newblock \bibinfo{title}{Adaptive robust data-driven building control via
  bi-level reformulation: an experimental result}.
\newblock \bibinfo{note}{ArXiv preprint arXiv:2106.05740}.
%Type = Book
\bibitem[{Malliavin et~al.(1995)Malliavin, Airault, Kay \&
  Letac}]{malliavin1995}
\bibinfo{author}{Malliavin, P.}, \bibinfo{author}{Airault, H.},
  \bibinfo{author}{Kay, L.}, \& \bibinfo{author}{Letac, G.}
  (\bibinfo{year}{1995}).
\newblock {\it \bibinfo{title}{Integration and Probability}\/} volume
  \bibinfo{volume}{157}.
\newblock \bibinfo{publisher}{Springer Science \& Business Media}.
%Type = Article
\bibitem[{Markovsky(2021)}]{Markovsky21}
\bibinfo{author}{Markovsky, I.} (\bibinfo{year}{2021}).
\newblock \bibinfo{title}{Data-driven simulation of nonlinear systems via
  linear time-invariant embedding}.
\newblock {\it \bibinfo{journal}{Vrije Universiteit Brussel}\/}, .
\newblock \bibinfo{note}{Technical Report}.
%Type = Article
\bibitem[{Markovsky \& D{\"o}rfler(2021)}]{MarkDorf21}
\bibinfo{author}{Markovsky, I.}, \& \bibinfo{author}{D{\"o}rfler, F.}
  (\bibinfo{year}{2021}).
\newblock \bibinfo{title}{Behavioral systems theory in data-driven analysis,
  signal processing, and control}.
\newblock {\it \bibinfo{journal}{Annual Reviews in Control}\/},  {\it
  \bibinfo{volume}{52}\/}, \bibinfo{pages}{42--64}.
%Type = Book
\bibitem[{Markovsky et~al.(2006)Markovsky, Willems, Van~Huffel \&
  De~Moor}]{MarkovskyWillemsHuffelDeMoor06}
\bibinfo{author}{Markovsky, I.}, \bibinfo{author}{Willems, J.~C.},
  \bibinfo{author}{Van~Huffel, S.}, \& \bibinfo{author}{De~Moor, B.}
  (\bibinfo{year}{2006}).
\newblock {\it \bibinfo{title}{{Exact and Approximate Modeling of Linear
  Systems. A Behavioral Approach}}\/} volume~\bibinfo{volume}{11} of {\it
  \bibinfo{series}{Math. Model. Comput.}\/}.
\newblock \bibinfo{publisher}{Philadelphia, PA: Society for Industrial {and}
  Applied Mathematics (SIAM)}.
\newblock \DOIprefix\doi{10.1137/1.9780898718263}.
%Type = Article
\bibitem[{Martinelli et~al.(2022)Martinelli, Gargiani, Draskovic \&
  Lygeros}]{Martinelli22}
\bibinfo{author}{Martinelli, A.}, \bibinfo{author}{Gargiani, M.},
  \bibinfo{author}{Draskovic, M.}, \& \bibinfo{author}{Lygeros, J.}
  (\bibinfo{year}{2022}).
\newblock \bibinfo{title}{{Data-driven optimal control of affine systems: A
  linear programming perspective}}.
\newblock {\it \bibinfo{journal}{IEEE Control Systems Letters}\/},  {\it
  \bibinfo{volume}{6}\/}, \bibinfo{pages}{3092--3097}.
  \DOIprefix\doi{10.1109/LCSYS.2022.3180898}.
%Type = Book
\bibitem[{Mauroy et~al.(2020)Mauroy, Susuki \& Mezi{\'c}}]{MaurSusu20}
\bibinfo{author}{Mauroy, A.}, \bibinfo{author}{Susuki, Y.}, \&
  \bibinfo{author}{Mezi{\'c}, I.} (\bibinfo{year}{2020}).
\newblock {\it \bibinfo{title}{Koopman operator in systems and control}\/}.
\newblock \bibinfo{publisher}{Springer}.
%Type = Article
\bibitem[{Mesbah(2016)}]{Mesbah16a}
\bibinfo{author}{Mesbah, A.} (\bibinfo{year}{2016}).
\newblock \bibinfo{title}{Stochastic model predictive control: An overview and
  perspectives for future research}.
\newblock {\it \bibinfo{journal}{IEEE Control Systems}\/},  {\it
  \bibinfo{volume}{36}\/}(6), \bibinfo{pages}{30--44}.
%Type = Article
\bibitem[{Mesbah \& Streif(2015)}]{Mesbah15a}
\bibinfo{author}{Mesbah, A.}, \& \bibinfo{author}{Streif, S.}
  (\bibinfo{year}{2015}).
\newblock \bibinfo{title}{A probabilistic approach to robust optimal experiment
  design with chance constraints}.
\newblock {\it \bibinfo{journal}{IFAC-PapersOnLine}\/},  {\it
  \bibinfo{volume}{48}\/}(8), \bibinfo{pages}{100--105}.
%Type = Article
\bibitem[{Milano \& Z{\'a}rate-Mi{\~n}ano(2013)}]{milano13}
\bibinfo{author}{Milano, F.}, \& \bibinfo{author}{Z{\'a}rate-Mi{\~n}ano, R.}
  (\bibinfo{year}{2013}).
\newblock \bibinfo{title}{A systematic method to model power systems as
  stochastic differential algebraic equations}.
\newblock {\it \bibinfo{journal}{IEEE Transactions on Power Systems}\/},  {\it
  \bibinfo{volume}{28}\/}(4), \bibinfo{pages}{4537--4544}.
%Type = Article
\bibitem[{Mohajerin~Esfahani \& Kuhn(2018)}]{mohajerin2018data}
\bibinfo{author}{Mohajerin~Esfahani, P.}, \& \bibinfo{author}{Kuhn, D.}
  (\bibinfo{year}{2018}).
\newblock \bibinfo{title}{Data-driven distributionally robust optimization
  using the {Wasserstein} metric: Performance guarantees and tractable
  reformulations}.
\newblock {\it \bibinfo{journal}{Mathematical Programming}\/},  {\it
  \bibinfo{volume}{171}\/}(1), \bibinfo{pages}{115--166}.
%Type = Article
\bibitem[{Moonen et~al.(1989)Moonen, De~Moor, Vandenberghe \&
  Vandewalle}]{moonen1989}
\bibinfo{author}{Moonen, M.}, \bibinfo{author}{De~Moor, B.},
  \bibinfo{author}{Vandenberghe, L.}, \& \bibinfo{author}{Vandewalle, J.}
  (\bibinfo{year}{1989}).
\newblock \bibinfo{title}{On- and off-line identification of linear state-space
  models}.
\newblock {\it \bibinfo{journal}{International Journal of Control}\/},  {\it
  \bibinfo{volume}{49}\/}(1), \bibinfo{pages}{219--232}.
%Type = Article
\bibitem[{Muandet et~al.(2017)Muandet, Fukumizu, Sriperumbudur, Sch{\"o}lkopf
  et~al.}]{Muandet17}
\bibinfo{author}{Muandet, K.}, \bibinfo{author}{Fukumizu, K.},
  \bibinfo{author}{Sriperumbudur, B.}, \bibinfo{author}{Sch{\"o}lkopf, B.}
  et~al. (\bibinfo{year}{2017}).
\newblock \bibinfo{title}{Kernel mean embedding of distributions: A review and
  beyond}.
\newblock {\it \bibinfo{journal}{Foundations and Trends{\textregistered} in
  Machine Learning}\/},  {\it \bibinfo{volume}{10}\/}(1-2),
  \bibinfo{pages}{1--141}.
%Type = Article
\bibitem[{M{\"u}hlpfordt et~al.(2018{\natexlab{a}})M{\"u}hlpfordt, Faulwasser
  \& Hagenmeyer}]{kit:muehlpfordt18c}
\bibinfo{author}{M{\"u}hlpfordt, T.}, \bibinfo{author}{Faulwasser, T.}, \&
  \bibinfo{author}{Hagenmeyer, V.} (\bibinfo{year}{2018}{\natexlab{a}}).
\newblock \bibinfo{title}{A generalized framework for chance-constrained
  optimal power flow}.
\newblock {\it \bibinfo{journal}{Sustainable Energy, Grids and Networks}\/},
  {\it \bibinfo{volume}{16}\/}, \bibinfo{pages}{231--242}.
  \DOIprefix\doi{10.1016/j.segan.2018.08.002}.
\newblock \bibinfo{note}{ArXiv:1803.08299}.
%Type = Inproceedings
\bibitem[{M{\"u}hlpfordt et~al.(2017)M{\"u}hlpfordt, Faulwasser, Roald \&
  Hagenmeyer}]{kit:muehlpfordt17b}
\bibinfo{author}{M{\"u}hlpfordt, T.}, \bibinfo{author}{Faulwasser, T.},
  \bibinfo{author}{Roald, L.}, \& \bibinfo{author}{Hagenmeyer, V.}
  (\bibinfo{year}{2017}).
\newblock \bibinfo{title}{Solving optimal power flow with non-{G}aussian
  uncertainties via polynomial chaos expansion}.
\newblock In {\it \bibinfo{booktitle}{Proc. 2017 56th IEEE Conference on
  Decision and Control}\/} (pp. \bibinfo{pages}{4490--4496}).
\newblock \DOIprefix\doi{10.1109/CDC.2017.8264321}.
%Type = Article
\bibitem[{M{\"u}hlpfordt et~al.(2018{\natexlab{b}})M{\"u}hlpfordt, Findeisen,
  Hagenmeyer \& Faulwasser}]{Muehlpfordt18}
\bibinfo{author}{M{\"u}hlpfordt, T.}, \bibinfo{author}{Findeisen, R.},
  \bibinfo{author}{Hagenmeyer, V.}, \& \bibinfo{author}{Faulwasser, T.}
  (\bibinfo{year}{2018}{\natexlab{b}}).
\newblock \bibinfo{title}{Comments on quantifying truncation errors for
  polynomial chaos expansions}.
\newblock {\it \bibinfo{journal}{IEEE Control Systems Letters}\/},  {\it
  \bibinfo{volume}{2}\/}(1), \bibinfo{pages}{169--174}.
  \DOIprefix\doi{10.1109/LCSYS.2017.2778138}.
%Type = Article
\bibitem[{M{\"u}hlpfordt et~al.(2020)M{\"u}hlpfordt, Zahn, Hagenmeyer \&
  Faulwasser}]{tudo:muehlpfordt20c}
\bibinfo{author}{M{\"u}hlpfordt, T.}, \bibinfo{author}{Zahn, F.},
  \bibinfo{author}{Hagenmeyer, V.}, \& \bibinfo{author}{Faulwasser, T.}
  (\bibinfo{year}{2020}).
\newblock \bibinfo{title}{{PolyChaos.jl}—{A} julia package for polynomial
  chaos in systems and control}.
\newblock {\it \bibinfo{journal}{IFAC-PapersOnLine}\/},  {\it
  \bibinfo{volume}{53}\/}(2), \bibinfo{pages}{7210--7216}.
  \DOIprefix\doi{10.1016/j.ifacol.2020.12.552}.
\newblock \bibinfo{note}{21th IFAC World Congress}.
%Type = Article
\bibitem[{Nagy \& Braatz(2007)}]{Nagy07}
\bibinfo{author}{Nagy, Z.~K.}, \& \bibinfo{author}{Braatz, R.~D.}
  (\bibinfo{year}{2007}).
\newblock \bibinfo{title}{Distributional uncertainty analysis using power
  series and polynomial chaos expansions}.
\newblock {\it \bibinfo{journal}{Journal of Process Control}\/},  {\it
  \bibinfo{volume}{17}\/}(3), \bibinfo{pages}{229--240}.
%Type = Misc
\bibitem[{Nortmann \& Mylvaganam(2021)}]{Nortmann21}
\bibinfo{author}{Nortmann, B.}, \& \bibinfo{author}{Mylvaganam, T.}
  (\bibinfo{year}{2021}).
\newblock \bibinfo{title}{Direct data-driven control of linear time-varying
  systems}.
\newblock \bibinfo{note}{ArXiv preprint arXiv:2111.02342}.
%Type = Article
\bibitem[{N{\"u}ske et~al.(2023)N{\"u}ske, Peitz, Philipp, Schaller \&
  Worthmann}]{nuske2021finite}
\bibinfo{author}{N{\"u}ske, F.}, \bibinfo{author}{Peitz, S.},
  \bibinfo{author}{Philipp, F.}, \bibinfo{author}{Schaller, M.}, \&
  \bibinfo{author}{Worthmann, K.} (\bibinfo{year}{2023}).
\newblock \bibinfo{title}{Finite-data error bounds for {Koopman-based}
  prediction and control}.
\newblock {\it \bibinfo{journal}{Journal of Nonlinear Science}\/},  {\it
  \bibinfo{volume}{33}\/}, \bibinfo{pages}{14}.
%Type = Article
\bibitem[{Ou et~al.(2021)Ou, Baumann, Gr{\"u}ne \& Faulwasser}]{Ou21}
\bibinfo{author}{Ou, R.}, \bibinfo{author}{Baumann, M.~H.},
  \bibinfo{author}{Gr{\"u}ne, L.}, \& \bibinfo{author}{Faulwasser, T.}
  (\bibinfo{year}{2021}).
\newblock \bibinfo{title}{A simulation study on turnpikes in stochastic {LQ}
  optimal control}.
\newblock {\it \bibinfo{journal}{IFAC-PapersOnLine}\/},  {\it
  \bibinfo{volume}{54}\/}(3), \bibinfo{pages}{516--521}.
  \DOIprefix\doi{https://doi.org/10.1016/j.ifacol.2021.08.294}.
%Type = Article
\bibitem[{Ou et~al.(2023)Ou, Pan \& Faulwasser}]{tudo:ou23a}
\bibinfo{author}{Ou, R.}, \bibinfo{author}{Pan, G.}, \&
  \bibinfo{author}{Faulwasser, T.} (\bibinfo{year}{2023}).
\newblock \bibinfo{title}{Data-driven multiple shooting for stochastic optimal
  control}.
\newblock {\it \bibinfo{journal}{IEEE Control Systems Letters}\/},  {\it
  \bibinfo{volume}{7}\/}, \bibinfo{pages}{313--318}.
  \DOIprefix\doi{10.1109/LCSYS.2022.3185841}.
%Type = Article
\bibitem[{O’Dwyer et~al.(2022)O’Dwyer, Kerrigan, Falugi, Zagorowska \&
  Shah}]{ODwyer21}
\bibinfo{author}{O’Dwyer, E.}, \bibinfo{author}{Kerrigan, E.~C.},
  \bibinfo{author}{Falugi, P.}, \bibinfo{author}{Zagorowska, M.}, \&
  \bibinfo{author}{Shah, N.} (\bibinfo{year}{2022}).
\newblock \bibinfo{title}{Data-driven predictive control with improved
  performance using segmented trajectories}.
\newblock {\it \bibinfo{journal}{IEEE Transactions on Control Systems
  Technology}\/},  (pp. \bibinfo{pages}{1--11}).
  \DOIprefix\doi{10.1109/TCST.2022.3224330}.
%Type = Article
\bibitem[{O’Hagan(2013)}]{o2013polynomial}
\bibinfo{author}{O’Hagan, A.} (\bibinfo{year}{2013}).
\newblock \bibinfo{title}{Polynomial chaos: A tutorial and critique from a
  statistician’s perspective}.
\newblock {\it \bibinfo{journal}{SIAM/ASA J. Uncertainty Quantification}\/},
  {\it \bibinfo{volume}{20}\/}, \bibinfo{pages}{1--20}.
%Type = Misc
\bibitem[{Pan et~al.(2022{\natexlab{a}})Pan, Ou \& Faulwasser}]{pan2022e}
\bibinfo{author}{Pan, G.}, \bibinfo{author}{Ou, R.}, \&
  \bibinfo{author}{Faulwasser, T.} (\bibinfo{year}{2022}{\natexlab{a}}).
\newblock \bibinfo{title}{Data-driven stochastic output-feedback predictive
  control: Recursive feasibility through interpolated initial conditions}.
\newblock \bibinfo{note}{ArXiv preprint arXiv:2212.07661}.
%Type = Article
\bibitem[{Pan et~al.(2022{\natexlab{b}})Pan, Ou \& Faulwasser}]{Pan21s}
\bibinfo{author}{Pan, G.}, \bibinfo{author}{Ou, R.}, \&
  \bibinfo{author}{Faulwasser, T.} (\bibinfo{year}{2022}{\natexlab{b}}).
\newblock \bibinfo{title}{On a stochastic fundamental lemma and its use for
  data-driven optimal control}.
\newblock {\it \bibinfo{journal}{IEEE Transactions on Automatic Control}\/},
  (pp. \bibinfo{pages}{1--16}). \DOIprefix\doi{10.1109/TAC.2022.3232442}.
%Type = Misc
\bibitem[{Pan et~al.(2022{\natexlab{c}})Pan, Ou \& Faulwasser}]{pan22d}
\bibinfo{author}{Pan, G.}, \bibinfo{author}{Ou, R.}, \&
  \bibinfo{author}{Faulwasser, T.} (\bibinfo{year}{2022}{\natexlab{c}}).
\newblock \bibinfo{title}{On data-driven stochastic output-feedback predictive
  control}.
\newblock \bibinfo{note}{ArXiv preprint arXiv:2211.17074}.
%Type = Misc
\bibitem[{Pan et~al.(2022{\natexlab{d}})Pan, Ou \& Faulwasser}]{tudo:pan22a}
\bibinfo{author}{Pan, G.}, \bibinfo{author}{Ou, R.}, \&
  \bibinfo{author}{Faulwasser, T.} (\bibinfo{year}{2022}{\natexlab{d}}).
\newblock \bibinfo{title}{Towards data-driven stochastic predictive control}.
\newblock \bibinfo{note}{ArXiv preprint arXiv:2212.10663}.
%Type = Article
\bibitem[{Petzke et~al.(2020)Petzke, Mesbah \& Streif}]{petzke2020pocet}
\bibinfo{author}{Petzke, F.}, \bibinfo{author}{Mesbah, A.}, \&
  \bibinfo{author}{Streif, S.} (\bibinfo{year}{2020}).
\newblock \bibinfo{title}{{PoCET:} a polynomial chaos expansion toolbox for
  {Matlab}}.
\newblock {\it \bibinfo{journal}{IFAC-PapersOnLine}\/},  {\it
  \bibinfo{volume}{53}\/}(2), \bibinfo{pages}{7256--7261}.
%Type = Article
\bibitem[{Pillai \& Shankar(1999)}]{Pillai99b}
\bibinfo{author}{Pillai, H.~K.}, \& \bibinfo{author}{Shankar, S.}
  (\bibinfo{year}{1999}).
\newblock \bibinfo{title}{A behavioral approach to control of distributed
  systems}.
\newblock {\it \bibinfo{journal}{SIAM Journal on Control and Optimization}\/},
  {\it \bibinfo{volume}{37}\/}(2), \bibinfo{pages}{388--408}.
%Type = Inproceedings
\bibitem[{Pillai \& Willems(1999)}]{Pillai99a}
\bibinfo{author}{Pillai, H.~K.}, \& \bibinfo{author}{Willems, J.~C.}
  (\bibinfo{year}{1999}).
\newblock \bibinfo{title}{The behavioural approach to distributed systems}.
\newblock In {\it \bibinfo{booktitle}{Procc 1999 38th IEEE Conference on
  Decision and Control}\/} (pp. \bibinfo{pages}{626--630}).
\newblock \bibinfo{organization}{IEEE}.
%Type = Article
\bibitem[{Pillai \& Willems(2002)}]{Pillai02a}
\bibinfo{author}{Pillai, H.~K.}, \& \bibinfo{author}{Willems, J.~C.}
  (\bibinfo{year}{2002}).
\newblock \bibinfo{title}{Lossless and dissipative distributed systems}.
\newblock {\it \bibinfo{journal}{SIAM Journal on Control and Optimization}\/},
  {\it \bibinfo{volume}{40}\/}(5), \bibinfo{pages}{1406--1430}.
%Type = Inproceedings
\bibitem[{Pola(2017)}]{Pola17a}
\bibinfo{author}{Pola, G.} (\bibinfo{year}{2017}).
\newblock \bibinfo{title}{On achievable behavior of stochastic descriptor
  systems}.
\newblock In {\it \bibinfo{booktitle}{Proc. 2017 56th IEEE Conference on
  Decision and Control}\/} (pp. \bibinfo{pages}{3164--3169}).
\newblock \bibinfo{organization}{IEEE}.
%Type = Inproceedings
\bibitem[{Pola et~al.(2016)Pola, Manes \& Di~Benedetto}]{Pola16a}
\bibinfo{author}{Pola, G.}, \bibinfo{author}{Manes, C.}, \&
  \bibinfo{author}{Di~Benedetto, M.~D.} (\bibinfo{year}{2016}).
\newblock \bibinfo{title}{On external behavior equivalence of continuous-time
  stochastic linear control systems}.
\newblock In {\it \bibinfo{booktitle}{Proc. 2016 IEEE 55th Conference on
  Decision and Control}\/} (pp. \bibinfo{pages}{6583--6588}).
\newblock \bibinfo{organization}{IEEE}.
%Type = Inproceedings
\bibitem[{Pola et~al.(2015)Pola, Manes, van~der Schaft \&
  Di~Benedetto}]{Pola15a}
\bibinfo{author}{Pola, G.}, \bibinfo{author}{Manes, C.},
  \bibinfo{author}{van~der Schaft, A.~J.}, \& \bibinfo{author}{Di~Benedetto,
  M.~D.} (\bibinfo{year}{2015}).
\newblock \bibinfo{title}{On equivalence notions for discrete-time stochastic
  control systems}.
\newblock In {\it \bibinfo{booktitle}{Proc. 2015 54th IEEE Conference on
  Decision and Control}\/} (pp. \bibinfo{pages}{1180--1185}).
\newblock \bibinfo{organization}{IEEE}.
%Type = Book
\bibitem[{{Polderman} \& {Willems}(1997)}]{PoldermanWillmes98}
\bibinfo{author}{{Polderman}, J.~W.}, \& \bibinfo{author}{{Willems}, J.~C.}
  (\bibinfo{year}{1997}).
\newblock {\it \bibinfo{title}{{Introduction to Mathematical Systems Theory. A
  Behavioral Approach}}\/} volume~\bibinfo{volume}{26}.
\newblock \bibinfo{publisher}{New York, NY: Springer}.
%Type = Article
\bibitem[{Proctor et~al.(2016)Proctor, Brunton \& Kutz}]{ProcBrun16}
\bibinfo{author}{Proctor, J.~L.}, \bibinfo{author}{Brunton, S.~L.}, \&
  \bibinfo{author}{Kutz, J.~N.} (\bibinfo{year}{2016}).
\newblock \bibinfo{title}{Dynamic mode decomposition with control}.
\newblock {\it \bibinfo{journal}{SIAM J. Applied Dynamical Systems}\/},  {\it
  \bibinfo{volume}{15}\/}(1), \bibinfo{pages}{142--161}.
%Type = Book
\bibitem[{Rainville(1960)}]{Rainville60}
\bibinfo{author}{Rainville, E.} (\bibinfo{year}{1960}).
\newblock {\it \bibinfo{title}{Special functions}\/}.
\newblock \bibinfo{publisher}{Macmillan}.
%Type = Article
\bibitem[{Romer et~al.(2019{\natexlab{a}})Romer, Berberich, K{\"o}hler \&
  Allg{\"o}wer}]{romer2019one}
\bibinfo{author}{Romer, A.}, \bibinfo{author}{Berberich, J.},
  \bibinfo{author}{K{\"o}hler, J.}, \& \bibinfo{author}{Allg{\"o}wer, F.}
  (\bibinfo{year}{2019}{\natexlab{a}}).
\newblock \bibinfo{title}{One-shot verification of dissipativity properties
  from input--output data}.
\newblock {\it \bibinfo{journal}{IEEE Control Systems Letters}\/},  {\it
  \bibinfo{volume}{3}\/}(3), \bibinfo{pages}{709--714}.
%Type = Inproceedings
\bibitem[{Romer et~al.(2019{\natexlab{b}})Romer, Trimpe \&
  Allg{\"o}wer}]{romer2019data}
\bibinfo{author}{Romer, A.}, \bibinfo{author}{Trimpe, S.}, \&
  \bibinfo{author}{Allg{\"o}wer, F.} (\bibinfo{year}{2019}{\natexlab{b}}).
\newblock \bibinfo{title}{Data-driven inference of passivity properties via
  {G}aussian process optimization}.
\newblock In {\it \bibinfo{booktitle}{Proc. 2019 European Control
  Conference}\/} (pp. \bibinfo{pages}{29--35}).
\newblock \bibinfo{organization}{IEEE}.
%Type = Article
\bibitem[{Rueda-Escobedo et~al.(2022)Rueda-Escobedo, Fridman \&
  Schiffer}]{rueda2021data}
\bibinfo{author}{Rueda-Escobedo, J.~G.}, \bibinfo{author}{Fridman, E.}, \&
  \bibinfo{author}{Schiffer, J.} (\bibinfo{year}{2022}).
\newblock \bibinfo{title}{Data-driven control for linear discrete-time delay
  systems}.
\newblock {\it \bibinfo{journal}{IEEE Transactions on Automatic Control}\/},
  {\it \bibinfo{volume}{67}\/}(7), \bibinfo{pages}{3321--3336}.
  \DOIprefix\doi{10.1109/TAC.2021.3096896}.
%Type = Inproceedings
\bibitem[{Rueda-Escobedo \& Schiffer(2020)}]{rueda2020data}
\bibinfo{author}{Rueda-Escobedo, J.~G.}, \& \bibinfo{author}{Schiffer, J.}
  (\bibinfo{year}{2020}).
\newblock \bibinfo{title}{Data-driven internal model control of second-order
  discrete volterra systems}.
\newblock In {\it \bibinfo{booktitle}{Proc. 2020 59th IEEE Conference on
  Decision and Control}\/} (pp. \bibinfo{pages}{4572--4579}).
\newblock \bibinfo{organization}{IEEE}.
%Type = Article
\bibitem[{Schaller et~al.(2023)Schaller, Worthmann, Philipp, Peitz \&
  N{\"u}ske}]{SchaWort22}
\bibinfo{author}{Schaller, M.}, \bibinfo{author}{Worthmann, K.},
  \bibinfo{author}{Philipp, F.}, \bibinfo{author}{Peitz, S.}, \&
  \bibinfo{author}{N{\"u}ske, F.} (\bibinfo{year}{2023}).
\newblock \bibinfo{title}{Towards reliable data-based optimal and predictive
  control using extended {DMD}}.
\newblock {\it \bibinfo{journal}{IFAC-PapersOnLine}\/},  {\it
  \bibinfo{volume}{56}\/}(1), \bibinfo{pages}{169--174}.
%Type = Article
\bibitem[{Schmitz et~al.(2022{\natexlab{a}})Schmitz, Engelmann, Faulwasser \&
  Worthmann}]{Schmitz22}
\bibinfo{author}{Schmitz, P.}, \bibinfo{author}{Engelmann, A.},
  \bibinfo{author}{Faulwasser, T.}, \& \bibinfo{author}{Worthmann, K.}
  (\bibinfo{year}{2022}{\natexlab{a}}).
\newblock \bibinfo{title}{{Data-driven MPC of descriptor systems: A case study
  for power networks}}.
\newblock {\it \bibinfo{journal}{IFAC PapersOnLine}\/},  {\it
  \bibinfo{volume}{55}\/}(30), \bibinfo{pages}{359--364}.
%Type = Article
\bibitem[{Schmitz et~al.(2022{\natexlab{b}})Schmitz, Faulwasser \&
  Worthmann}]{SchmitzFaulwasserWorthmann22}
\bibinfo{author}{Schmitz, P.}, \bibinfo{author}{Faulwasser, T.}, \&
  \bibinfo{author}{Worthmann, K.} (\bibinfo{year}{2022}{\natexlab{b}}).
\newblock \bibinfo{title}{Willems’ fundamental lemma for linear descriptor
  systems and its use for data-driven output-feedback {MPC}}.
\newblock {\it \bibinfo{journal}{IEEE Control Systems Letters}\/},  {\it
  \bibinfo{volume}{6}\/}, \bibinfo{pages}{2443--2448}.
  \DOIprefix\doi{10.1109/LCSYS.2022.3161054}.
%Type = Incollection
\bibitem[{Seiler \& Zerz(2015)}]{seiler2015algebraic}
\bibinfo{author}{Seiler, W.~M.}, \& \bibinfo{author}{Zerz, E.}
  (\bibinfo{year}{2015}).
\newblock \bibinfo{title}{Algebraic theory of linear systems: A survey}.
\newblock In {\it \bibinfo{booktitle}{Surveys in Differential-Algebraic
  Equations II}\/} (pp. \bibinfo{pages}{287--333}).
\newblock \bibinfo{publisher}{Springer}.
%Type = Article
\bibitem[{Singh \& Hespanha(2010)}]{Singh10}
\bibinfo{author}{Singh, A.}, \& \bibinfo{author}{Hespanha, J.~P.}
  (\bibinfo{year}{2010}).
\newblock \bibinfo{title}{Approximate moment dynamics for chemically reacting
  systems}.
\newblock {\it \bibinfo{journal}{IEEE Transactions on Automatic Control}\/},
  {\it \bibinfo{volume}{56}\/}(2), \bibinfo{pages}{414--418}.
%Type = Book
\bibitem[{Sontag(1998)}]{Sontag98a}
\bibinfo{author}{Sontag, E.} (\bibinfo{year}{1998}).
\newblock {\it \bibinfo{title}{Mathematical Control Theory: Deterministic
  Finite Dimensional Systems}\/} volume~\bibinfo{volume}{6}.
\newblock \bibinfo{publisher}{Springer}.
%Type = Phdthesis
\bibitem[{{Stykel}(2002)}]{Stykel02}
\bibinfo{author}{{Stykel}, T.} (\bibinfo{year}{2002}).
\newblock {\it \bibinfo{title}{{Analysis and numerical solution of generalized
  Lyapunov equations}}\/}.
\newblock \bibinfo{type}{dissertation} TU Berlin.
%Type = Book
\bibitem[{Sullivan(2015)}]{Sullivan15}
\bibinfo{author}{Sullivan, T.~J.} (\bibinfo{year}{2015}).
\newblock {\it \bibinfo{title}{{Introduction to Uncertainty Quantification}}\/}
  volume~\bibinfo{volume}{63}.
\newblock \bibinfo{publisher}{Springer}.
%Type = Inproceedings
\bibitem[{Surana(2016)}]{Sura2016}
\bibinfo{author}{Surana, A.} (\bibinfo{year}{2016}).
\newblock \bibinfo{title}{Koopman operator based observer synthesis for
  control-affine nonlinear systems}.
\newblock In {\it \bibinfo{booktitle}{Proc. 2016 55th IEEE Conference Decision
  Control (CDC)}\/} (pp. \bibinfo{pages}{6492--6499}).
%Type = Misc
\bibitem[{Umlauft et~al.(2020)Umlauft, Lederer, Beckers \&
  Hirche}]{umlauft2020real}
\bibinfo{author}{Umlauft, J.}, \bibinfo{author}{Lederer, A.},
  \bibinfo{author}{Beckers, T.}, \& \bibinfo{author}{Hirche, S.}
  (\bibinfo{year}{2020}).
\newblock \bibinfo{title}{Real-time uncertainty decomposition for online
  learning control}.
\newblock \bibinfo{note}{ArXiv preprint arXiv:2010.02613}.
%Type = Article
\bibitem[{Venzke et~al.(2021)Venzke, Molzahn \&
  Chatzivasileiadis}]{venzke2021efficient}
\bibinfo{author}{Venzke, A.}, \bibinfo{author}{Molzahn, D.~K.}, \&
  \bibinfo{author}{Chatzivasileiadis, S.} (\bibinfo{year}{2021}).
\newblock \bibinfo{title}{Efficient creation of datasets for data-driven power
  system applications}.
\newblock {\it \bibinfo{journal}{Electric Power Systems Research}\/},  {\it
  \bibinfo{volume}{190}\/}, \bibinfo{pages}{106614}.
%Type = Incollection
\bibitem[{Verhaegen(2015)}]{verhaegen2015subspace}
\bibinfo{author}{Verhaegen, M.} (\bibinfo{year}{2015}).
\newblock \bibinfo{title}{Subspace techniques in system identification}.
\newblock In {\it \bibinfo{booktitle}{Encyclopedia of Systems and Control}\/}
  (pp. \bibinfo{pages}{1386--1396}).
\newblock \bibinfo{publisher}{Springer}.
\newblock \DOIprefix\doi{10.1007/978-1-4471-5058-9_107}.
%Type = Inproceedings
\bibitem[{Verhoek et~al.(2021)Verhoek, T{\'o}th, Haesaert \& Koch}]{Verhoek21}
\bibinfo{author}{Verhoek, C.}, \bibinfo{author}{T{\'o}th, R.},
  \bibinfo{author}{Haesaert, S.}, \& \bibinfo{author}{Koch, A.}
  (\bibinfo{year}{2021}).
\newblock \bibinfo{title}{Fundamental lemma for data-driven analysis of linear
  parameter-varying systems}.
\newblock In {\it \bibinfo{booktitle}{Proc. 2021 60th IEEE Conference on
  Decision and Control}\/} (pp. \bibinfo{pages}{5040--5046}).
\newblock \bibinfo{organization}{IEEE}.
%Type = Article
\bibitem[{van Waarde et~al.(2022)van Waarde, Camlibel \& Mesbahi}]{Van20}
\bibinfo{author}{van Waarde, H.~J.}, \bibinfo{author}{Camlibel, M.~K.}, \&
  \bibinfo{author}{Mesbahi, M.} (\bibinfo{year}{2022}).
\newblock \bibinfo{title}{{From noisy data to feedback controllers:
  non-conservative design via a matrix S-lemma}}.
\newblock {\it \bibinfo{journal}{IEEE Transactions on Automatic Control}\/},
  {\it \bibinfo{volume}{67}\/}(1), \bibinfo{pages}{162--175}.
  \DOIprefix\doi{10.1109/TAC.2020.3047577}.
%Type = Article
\bibitem[{van Waarde et~al.(2020)van Waarde, De~Persis, Camlibel \&
  Tesi}]{WDPCT20}
\bibinfo{author}{van Waarde, H.~J.}, \bibinfo{author}{De~Persis, C.},
  \bibinfo{author}{Camlibel, M.~K.}, \& \bibinfo{author}{Tesi, P.}
  (\bibinfo{year}{2020}).
\newblock \bibinfo{title}{Willems’ fundamental lemma for state-space systems
  and its extension to multiple datasets}.
\newblock {\it \bibinfo{journal}{IEEE Control Systems Letters}\/},  {\it
  \bibinfo{volume}{4}\/}(3), \bibinfo{pages}{602--607}.
%Type = Article
\bibitem[{Wan et~al.(2023)Wan, Shen, Lucia, Findeisen \& Braatz}]{Wan22}
\bibinfo{author}{Wan, Y.}, \bibinfo{author}{Shen, D.}, \bibinfo{author}{Lucia,
  S.}, \bibinfo{author}{Findeisen, R.}, \& \bibinfo{author}{Braatz, R.~D.}
  (\bibinfo{year}{2023}).
\newblock \bibinfo{title}{A polynomial chaos approach to robust {H\_infinity}
  static output-feedback control with bounded truncation error}.
\newblock {\it \bibinfo{journal}{IEEE Transactions on Automatic Control}\/},
  {\it \bibinfo{volume}{68}\/}(1), \bibinfo{pages}{470--477}.
  \DOIprefix\doi{10.1109/TAC.2022.3140275}.
%Type = Article
\bibitem[{Wan et~al.(2021)Wan, Shen, Lucia, Findeisen \& Braatz}]{Wan21}
\bibinfo{author}{Wan, Y.}, \bibinfo{author}{Shen, D.~E.},
  \bibinfo{author}{Lucia, S.}, \bibinfo{author}{Findeisen, R.}, \&
  \bibinfo{author}{Braatz, R.~D.} (\bibinfo{year}{2021}).
\newblock \bibinfo{title}{Polynomial chaos-based {H2} output-feedback control
  of systems with probabilistic parametric uncertainties}.
\newblock {\it \bibinfo{journal}{Automatica}\/},  {\it
  \bibinfo{volume}{131}\/}, \bibinfo{pages}{109743}.
%Type = Misc
\bibitem[{Wang et~al.(2022)Wang, Zheng, Li \& Xu}]{Wang22}
\bibinfo{author}{Wang, J.}, \bibinfo{author}{Zheng, Y.}, \bibinfo{author}{Li,
  K.}, \& \bibinfo{author}{Xu, Q.} (\bibinfo{year}{2022}).
\newblock \bibinfo{title}{{DeeP-LCC: Data-enabled predictive leading cruise
  control in mixed traffic flow}}.
\newblock \bibinfo{note}{ArXiv preprint arXiv:2203.10639}.
%Type = Article
\bibitem[{Wiener(1938)}]{Wiener38}
\bibinfo{author}{Wiener, N.} (\bibinfo{year}{1938}).
\newblock \bibinfo{title}{The homogeneous chaos}.
\newblock {\it \bibinfo{journal}{American Journal of Mathematics}\/},  (pp.
  \bibinfo{pages}{897--936}).
%Type = Article
\bibitem[{Willems(1986{\natexlab{a}})}]{willems86i}
\bibinfo{author}{Willems, J.~C.} (\bibinfo{year}{1986}{\natexlab{a}}).
\newblock \bibinfo{title}{{From time series to linear system—Part I. Finite
  dimensional linear time invariant systems}}.
\newblock {\it \bibinfo{journal}{Automatica}\/},  {\it
  \bibinfo{volume}{22}\/}(5), \bibinfo{pages}{561--580}.
%Type = Article
\bibitem[{Willems(1986{\natexlab{b}})}]{willems86ii}
\bibinfo{author}{Willems, J.~C.} (\bibinfo{year}{1986}{\natexlab{b}}).
\newblock \bibinfo{title}{{From time series to linear system—Part II. Exact
  modelling}}.
\newblock {\it \bibinfo{journal}{Automatica}\/},  {\it
  \bibinfo{volume}{22}\/}(6), \bibinfo{pages}{675--694}.
%Type = Article
\bibitem[{Willems(1987)}]{willems87}
\bibinfo{author}{Willems, J.~C.} (\bibinfo{year}{1987}).
\newblock \bibinfo{title}{{From time series to linear system—Part III:
  Approximate modelling}}.
\newblock {\it \bibinfo{journal}{Automatica}\/},  {\it
  \bibinfo{volume}{23}\/}(1), \bibinfo{pages}{87--115}.
%Type = Article
\bibitem[{Willems(1991)}]{Willems91}
\bibinfo{author}{Willems, J.~C.} (\bibinfo{year}{1991}).
\newblock \bibinfo{title}{Paradigms and puzzles in the theory of dynamical
  systems}.
\newblock {\it \bibinfo{journal}{IEEE Transactions on Automatic Control}\/},
  {\it \bibinfo{volume}{36}\/}(3), \bibinfo{pages}{259--294}.
%Type = Article
\bibitem[{Willems(2007)}]{willems07}
\bibinfo{author}{Willems, J.~C.} (\bibinfo{year}{2007}).
\newblock \bibinfo{title}{The behavioral approach to open and interconnected
  systems}.
\newblock {\it \bibinfo{journal}{IEEE Control Systems Magazine}\/},  {\it
  \bibinfo{volume}{27}\/}(6), \bibinfo{pages}{46--99}.
%Type = Article
\bibitem[{Willems(2013)}]{Willems12}
\bibinfo{author}{Willems, J.~C.} (\bibinfo{year}{2013}).
\newblock \bibinfo{title}{Open stochastic systems}.
\newblock {\it \bibinfo{journal}{IEEE Transactions on Automatic Control}\/},
  {\it \bibinfo{volume}{58}\/}(2), \bibinfo{pages}{406--421}.
%Type = Article
\bibitem[{{Willems} et~al.(2005){Willems}, {Rapisarda}, {Markovsky} \& {De
  Moor}}]{WRMDM05}
\bibinfo{author}{{Willems}, J.~C.}, \bibinfo{author}{{Rapisarda}, P.},
  \bibinfo{author}{{Markovsky}, I.}, \& \bibinfo{author}{{De Moor}, B. L.~M.}
  (\bibinfo{year}{2005}).
\newblock \bibinfo{title}{{A note on persistency of excitation}}.
\newblock {\it \bibinfo{journal}{{Systems \& Control Letters}}\/},  {\it
  \bibinfo{volume}{54}\/}(4), \bibinfo{pages}{325--329}.
  \DOIprefix\doi{10.1016/j.sysconle.2004.09.003}.
%Type = Article
\bibitem[{Williams et~al.(2016)Williams, Hemati, Dawson, Kevrekidis \&
  Rowley}]{WillHema2016}
\bibinfo{author}{Williams, M.~O.}, \bibinfo{author}{Hemati, M.~S.},
  \bibinfo{author}{Dawson, S.~T.}, \bibinfo{author}{Kevrekidis, I.~G.}, \&
  \bibinfo{author}{Rowley, C.~W.} (\bibinfo{year}{2016}).
\newblock \bibinfo{title}{Extending data-driven {K}oopman analysis to actuated
  systems}.
\newblock {\it \bibinfo{journal}{IFAC-PapersOnLine}\/},  {\it
  \bibinfo{volume}{49}\/}(18), \bibinfo{pages}{704--709}.
%Type = Article
\bibitem[{Wood \& Zerz(1999)}]{WoodZerz99}
\bibinfo{author}{Wood, J.}, \& \bibinfo{author}{Zerz, E.}
  (\bibinfo{year}{1999}).
\newblock \bibinfo{title}{Notes on the definition of behavioural
  controllability}.
\newblock {\it \bibinfo{journal}{Systems \& Control Letters}\/},  {\it
  \bibinfo{volume}{37}\/}(1), \bibinfo{pages}{31--37}.
%Type = Book
\bibitem[{Xiu(2010)}]{xiu2010numerical}
\bibinfo{author}{Xiu, D.} (\bibinfo{year}{2010}).
\newblock {\it \bibinfo{title}{{Numerical Methods for Stochastic
  Computations}}\/}.
\newblock \bibinfo{publisher}{Princeton University Press}.
%Type = Article
\bibitem[{Xiu \& Karniadakis(2002)}]{Xiu02}
\bibinfo{author}{Xiu, D.}, \& \bibinfo{author}{Karniadakis, G.~E.}
  (\bibinfo{year}{2002}).
\newblock \bibinfo{title}{The {W}iener--{A}skey polynomial chaos for stochastic
  differential equations}.
\newblock {\it \bibinfo{journal}{SIAM Journal on Scientific Computing}\/},
  {\it \bibinfo{volume}{24}\/}(2), \bibinfo{pages}{619--644}.
%Type = Inproceedings
\bibitem[{Yamamoto \& Willems(2008)}]{yamamoto2008behavioral}
\bibinfo{author}{Yamamoto, Y.}, \& \bibinfo{author}{Willems, J.~C.}
  (\bibinfo{year}{2008}).
\newblock \bibinfo{title}{Behavioral controllability and coprimeness for a
  class of infinite-dimensional systems}.
\newblock In {\it \bibinfo{booktitle}{Proc. 2008 47th IEEE Conference on
  Decision and Control}\/} (pp. \bibinfo{pages}{1513--1518}).
\newblock \bibinfo{organization}{IEEE}.
%Type = Article
\bibitem[{Yan et~al.(2018)Yan, Duan, Liu \& Xu}]{yan2018gaussian}
\bibinfo{author}{Yan, L.}, \bibinfo{author}{Duan, X.}, \bibinfo{author}{Liu,
  B.}, \& \bibinfo{author}{Xu, J.} (\bibinfo{year}{2018}).
\newblock \bibinfo{title}{{Gaussian processes and polynomial chaos expansion
  for regression problem: linkage via the RKHS and comparison via the KL
  divergence}}.
\newblock {\it \bibinfo{journal}{Entropy}\/},  {\it \bibinfo{volume}{20}\/}(3),
  \bibinfo{pages}{191}.
%Type = Inproceedings
\bibitem[{Yang \& Li(2015)}]{Yang15a}
\bibinfo{author}{Yang, H.}, \& \bibinfo{author}{Li, S.} (\bibinfo{year}{2015}).
\newblock \bibinfo{title}{A data-driven predictive controller design based on
  reduced {Hankel} matrix}.
\newblock In {\it \bibinfo{booktitle}{2015 10th Asian Control Conference
  (ASCC)}\/} (pp. \bibinfo{pages}{1--7}).
\newblock \bibinfo{organization}{IEEE}.
%Type = Inproceedings
\bibitem[{Yin et~al.(2022)Yin, Iannelli \& Smith}]{yin2021data}
\bibinfo{author}{Yin, M.}, \bibinfo{author}{Iannelli, A.}, \&
  \bibinfo{author}{Smith, R.~S.} (\bibinfo{year}{2022}).
\newblock \bibinfo{title}{Data-driven prediction with stochastic data:
  Confidence regions and minimum mean-squared error estimates}.
\newblock In {\it \bibinfo{booktitle}{Proc. 2022 European Control
  Conference}\/} (pp. \bibinfo{pages}{853--858}).
\newblock \DOIprefix\doi{10.23919/ECC55457.2022.9838046}.
%Type = Article
\bibitem[{Yin et~al.(2023)Yin, Iannelli \& Smith}]{yin2021maximum}
\bibinfo{author}{Yin, M.}, \bibinfo{author}{Iannelli, A.}, \&
  \bibinfo{author}{Smith, R.~S.} (\bibinfo{year}{2023}).
\newblock \bibinfo{title}{Maximum likelihood estimation in data-driven modeling
  and control}.
\newblock {\it \bibinfo{journal}{IEEE Transactions on Automatic Control}\/},
  {\it \bibinfo{volume}{68}\/}(1), \bibinfo{pages}{317--328}.
  \DOIprefix\doi{10.1109/TAC.2021.3137788}.
%Type = Inproceedings
\bibitem[{Yu et~al.(2021)Yu, Talebi, van Waarde, Topcu, Mesbahi \&
  A{\c{c}}{\i}kmeșe}]{Yu21}
\bibinfo{author}{Yu, Y.}, \bibinfo{author}{Talebi, S.}, \bibinfo{author}{van
  Waarde, H.~J.}, \bibinfo{author}{Topcu, U.}, \bibinfo{author}{Mesbahi, M.},
  \& \bibinfo{author}{A{\c{c}}{\i}kmeșe, B.} (\bibinfo{year}{2021}).
\newblock \bibinfo{title}{{On controllability and persistency of excitation in
  data-driven control: Extensions of Willems’ fundamental lemma}}.
\newblock In {\it \bibinfo{booktitle}{Proc. 2021 60th IEEE Conference on
  Decision and Control}\/} (pp. \bibinfo{pages}{6485--6490}).
\newblock \bibinfo{organization}{IEEE}.
%Type = Article
\bibitem[{Yuan \& Cort{\'e}s(2022)}]{Yuan22}
\bibinfo{author}{Yuan, Z.}, \& \bibinfo{author}{Cort{\'e}s, J.}
  (\bibinfo{year}{2022}).
\newblock \bibinfo{title}{Data-driven optimal control of bilinear systems}.
\newblock {\it \bibinfo{journal}{IEEE Control Systems Letters}\/},  {\it
  \bibinfo{volume}{6}\/}, \bibinfo{pages}{2479--2484}.
%Type = Article
\bibitem[{Zavala(2014)}]{zavala2014stochastic}
\bibinfo{author}{Zavala, V.~M.} (\bibinfo{year}{2014}).
\newblock \bibinfo{title}{Stochastic optimal control model for natural gas
  networks}.
\newblock {\it \bibinfo{journal}{Computers \& Chemical Engineering}\/},  {\it
  \bibinfo{volume}{64}\/}, \bibinfo{pages}{103--113}.
%Type = Article
\bibitem[{Zhang et~al.(2020)Zhang, Harlim \& Li}]{Zhang20}
\bibinfo{author}{Zhang, H.}, \bibinfo{author}{Harlim, J.}, \&
  \bibinfo{author}{Li, X.} (\bibinfo{year}{2020}).
\newblock \bibinfo{title}{{Estimating linear response statistics using
  orthogonal polynomials: An RKHS formulation}}.
\newblock {\it \bibinfo{journal}{Foundations of Data Science}\/},  {\it
  \bibinfo{volume}{2}\/}(4), \bibinfo{pages}{443--485}.

\end{thebibliography}
\end{document}